\documentclass[reqno,11pt,letterpaper]{amsart}
\usepackage{amsmath,amssymb,amsthm,graphicx,mathrsfs,url,mathtools}
\usepackage[usenames,dvipsnames]{xcolor}
\usepackage{hyperref}
\hypersetup{colorlinks=true,linkcolor=Red,citecolor=Green}
\usepackage{amsxtra}
\usepackage{bbm}

\usepackage[draft]{changes}

\definechangesauthor[color=red]{hr}
\definechangesauthor[color=red]{hx}

\def\?[#1]{\textbf{[#1]}\marginpar{\Large{\textbf{??}}}}

\let\epsilon=\varepsilon 
\let\phi=\varphi

\newcommand{\OO}{{\mathcal O}}
\newcommand{\GG}{{\mathcal G}}
\newcommand{\RR}{{\mathbb R}}

\newcommand{\NN}{{\mathbb N}}
\newcommand{\CC}{{\mathbb C}}
\newcommand{\CI}{{{\mathcal C}^\infty}}

\newcommand{\CIc}{{{\mathcal C}^\infty_{\rm{c}}}}
\newcommand{\wbar}{\overline{w}}
\newcommand{\xbar}{\overline{x}}
\newcommand{\ybar}{\overline{y}}
\newcommand{\zbar}{\overline{z}}
\newcommand{\oo}[1]{\overline{#1}}

\newtheorem{prop}{Proposition}[section]
\newtheorem{thm}[prop]{Theorem}
\newtheorem{defi}[prop]{Definition}

\newtheorem{lem}[prop]{Lemma}
\newtheorem{corr}[prop]{Corollary}
\newtheorem{rem}[prop]{Remark}
\numberwithin{equation}{section}
\numberwithin{prop}{section}

\DeclareMathOperator{\dist}{dist}

\DeclareMathOperator{\Hess}{Hess}
\DeclareMathOperator{\Hol}{Hol}

\DeclareMathOperator{\neigh}{neigh}

\let\Im=\Imag

\DeclareMathOperator{\Real}{Re}
\let\Re=\Real

\DeclareMathOperator{\supp}{supp}

\title[Bergman projections with Gevrey weights]{Semiclassical asymptotics for Bergman projections with Gevrey weights}
\author[Xiong]{Haoren Xiong}
\email{haorenxiong@math.ucla.edu}
\address{Department of Mathematics, University of California,
	Los Angeles, CA 90095, USA}
\author[Xu]{Hang Xu}
\email{{xuhang9@mail.sysu.edu.cn}}
\address{School of Mathematics (Zhuhai), Sun Yat-sen University, Zhuhai, Guangdong 519082, China}

\thanks{The second author was supported in part by the NSFC grant No. 12201040.}

\begin{document}
	
	\begin{abstract}
		We extend the direct approach to the semiclassical asymptotics for Bergman projections, developed by Deleporte--Hitrik--Sj\"ostrand \cite{deleporte2022analytic} for real analytic exponential weights and Hitrik--Stone \cite{hitrik2022smooth} for smooth exponential weights, to the case of Gevrey weights. We prove that the amplitude of the asymptotic Bergman projection forms a Gevrey symbol whose asymptotic coefficients obey certain Gevrey-type growth rate, and it is constructed by an asymptotic inversion of an explicit Fourier integral operator up to a Gevrey-type small remainder. 
	\end{abstract}
	
	\maketitle

	\section{Introduction}
	Let $\Omega \subset \mathbb{C}^n$ be an open pseudoconvex domain and $\Phi\in\CI(\Omega)$ a strictly plurisubharmonic function defined on $\Omega$ (see \eqref{Phi strict plurisubharmonic}). Given a small parameter $h>0$, we consider the weighted $L^2$ space with respect to the Lebesgue measure $L(dx)$ on $\CC^n$:
	\begin{equation*}
		L^2(\Omega, e^{-2\Phi/h}L(dx))=\biggl\{f: \Omega \to \mathbb{C}\ \mbox{measurable} ; \int_{\Omega} |f(x)|^2 e^{-2\Phi(x)/h} L(dx)<\infty\biggr\},
	\end{equation*}
	and its closed subspace consisting of holomorphic functions (the Bergman space)
	\begin{equation*}
		H_{\Phi}(\Omega)=\Hol(\Omega)\cap L^2(\Omega, e^{-2\Phi/h}L(dx)).
	\end{equation*}
	
	The orthogonal projection $\Pi: L^2(\Omega, e^{-2\Phi/h}L(dx)) \rightarrow H_{\Phi}(\Omega)$ is known as the (weighted) Bergman projection. The \emph{Bergman kernel} $K(x, z)$ is a holomorphic function on $\Omega \times \rho(\Omega)$ (here we have written $\rho(\Omega)=\{x:\xbar\in\Omega\}$) such that $K(x,\ybar)e^{-2\Phi(y)/h}$ is the distribution (Schwartz) kernel of the orthogonal projection $\Pi$. 
	
	The asymptotic behavior of $K(x, \ybar)$ as $h\rightarrow 0^+$ plays a crucial role in complex geometry as $K(x, \ybar)$ serves as a local model of the Bergman kernel of a high power of a holomorphic line bundle over a compact complex manifold. In \cite{T90}, via studying the limit of $h\log \left(K(x, \xbar) e^{-2\Phi(x)/h}\right)$ as $h\rightarrow 0^+$, Tian proved the convergence of the Bergman metrics, which settled a conjecture by Yau \cite{Yau87}. A far-reaching generalization of Tian's result was obtained independently by Catlin \cite{Ca99} and Zelditch \cite{Ze98}, providing a complete asymptotic expansion of the Bergman kernel on the anti-diagonal as $h\to 0^+$,
	\begin{equation}\label{CaZe expansion eq}
		K(x, \xbar) e^{-2\Phi(x)/h} \sim \frac{1}{h^n} \left(a_0(x, \xbar)+ a_1(x, \xbar) h+ a_2(x, \xbar) h^2 +\cdots \right).
	\end{equation} 
	In \cite{Lu00}, Lu computed $a_j$ explicitly for $j\leq 3$ and showed that all the coefficients are polynomials of curvatures and their covariant derivatives with respect to the K\"ahler metric $\sqrt{-1}\partial\oo{\partial}\Phi$. 
	
	The original proof of \eqref{CaZe expansion eq} in \cite{Ca99,Ze98} relied on the Boutet de Monvel--Sj\"ostrand parametrix, a characterization of the singularity of the Bergman/Szeg\"o kernel (with no weight) near the boundary of a smoothly bounded strictly pseudoconvex domain in the seminal works \cite{BoSj76,Fe74}, and a reduction idea of Boutet de Monvel, Guillemin \cite{BoGu81}. Using again the reduction method, scaling asymptotics away from the diagonal were obtained later by Bleher, Shiffman, Zelditch \cite{BSZ00} and the full asymptotics by Charles \cite{charles2003berezin}, extending \eqref{CaZe expansion eq} to an off-diagonal asymptotic expansion, for $x, y$ in a small but fixed neighborhood of a point $x_0\in \Omega$, of the form 
	\begin{equation}\label{off-diagonal expansion eq}
		K(x, \ybar)=\frac{1}{h^n}e^{2\Psi(x, \ybar)/h} \sum_{j=0}^{N-1} a_j(x, \ybar)h^j + \OO(h^{N-n})e^{\left(\Phi(x)+\Phi(y)\right)/h},
	\end{equation}
	where $\Psi(x, z)$ and $a_j(x, z)$, respectively, are almost holomorphic extensions of $\Phi(x)$ and $a_j(x, \xbar)$. A direct approach to obtaining the asymptotics as in \eqref{off-diagonal expansion eq}, through constructing the local Bergman kernel based on the local reproducing property, was developed by Berman, Berndtsson and Sj\"ostrand \cite{BBSj08}. 
	
	In the case of real analytic weights $\Phi$, recent works by \cite{RoSjNg20, De21} (cf. \cite{HLX20, hezari2021property, Ch21}) showed that the infinite sum $\sum_{j=0}^{\infty} a_j(x, \ybar)h^j$ forms a classical analytic symbol, leading to an asymptotic expansion with an exponentially small error:	
	\begin{equation}\label{off-diagonal expansion analytic eq}
		K(x, \ybar)=\frac{1}{h^n}e^{2\Psi(x, \ybar)/h} \sum_{j=0}^{[(Ch)^{-1}]} a_j(x, \ybar)h^j + \OO(1)e^{-\delta/h} e^{\left(\Phi(x)+\Phi(y)\right)/h} ,
	\end{equation}
	where $C$ and $\delta$ are some positive constants, and $[(Ch)^{-1}]$ represents the greatest integer less than or equal to $(Ch)^{-1}$. 
	
	If we compare the asymptotic expansion for smooth weights in \eqref{off-diagonal expansion eq} to that for real analytic weights in \eqref{off-diagonal expansion analytic eq}, it is worth noting that the error term improves from a polynomial decaying rate to an exponentially decaying one. Motivated by these results, we aim to explore the expansion for Gevrey weights (see the definition below), which can be thought of as an interpolating case between the real analytic and smooth weights. 
	
	Let $s\geq 1$ and let $U\subset \RR^d$ be open. The Gevrey-$s$ class, $\GG^s(U)$, consists of all functions $u\in\CI(U)$ such that for every compact $K\subset U$,
	\begin{equation*}
		\label{Gevrey definition}
		\exists A_K,\,C_K>0,\quad\text{s.t. }\quad\forall\alpha\in\NN^d,\ x\in K,\quad |\partial^\alpha u(x) | \leq A_K C_K^{|\alpha|} \alpha!^s.
	\end{equation*}
	When $s=1$, $\GG^1(U)$ coincides with the class of real analytic functions on $U$; while for $s>1$, the Gevrey-$s$ class is non-quasianalytic, in the sense that $\GG_{\rm c}^s(U) := \GG^s(U)\cap\CIc(U) \neq \{0\}$.
	
	In this paper, we shall address the following two natural questions:
	\begin{itemize}
		\item[(a)] If $\Phi\in\GG^s(\Omega)$ with $s>1$, then does the infinite sum $\sum_{j=0}^{\infty} a_j(x, \ybar)h^j$ form a Gevrey symbol? If so, which Gevrey symbol class does it belong to?
		\item[(b)] If $\Phi\in\GG^s(\Omega)$ with $s>1$, what kind of estimate can be obtained for the error term in the asymptotic expansion?
	\end{itemize}
	Our main result is summarized as follows.
	\begin{thm}\label{main thm}
		Let $\Omega\subset \mathbb{C}^n$ be open and let $\Phi\in \mathcal{G}^s(\Omega; \mathbb{R})$ with $s>1$ be strictly plurisubharmonic in $\Omega$. For any $x_0\in \Omega$, there exist an elliptic symbol $a(x,\widetilde{y};h)\in\GG^s(\neigh((x_0,\overline{x_0}),\CC^{2n}))$ realizing the following formal $\GG^{s,2s-1}$ symbol (see Definition \ref{defi:formal Gevrey symbol})
		\begin{equation}
			\label{amplitude asymp aj_s}
			a(x,\widetilde{y};h) \sim \sum_{j=0}^{\infty} a_j(x,\widetilde{y}) h^j,
		\end{equation}
		in the sense of \eqref{amplitude a(x,y;h) Gevrey asymp}, with $a_j\in \GG^s(\neigh((x_0,\overline{x_0}),\CC^{2n}))$ being holomorphic to $\infty$--order along the anti-diagonal $\widetilde{y}=\xbar$, such that for some constant $C>0$,
		\begin{equation}
			\label{Aa=1 in main theorem}
			(Aa)(x, \xbar; h)=1 + \OO(1) \exp\bigl(-C^{-1}h^{-1/(2s-1)}\bigr), \quad x\in \neigh(x_0, \mathbb{C}^n),
		\end{equation}
		where $A$ is an elliptic Fourier integral operator, and small open neighborhoods $U\Subset V\Subset \Omega$ of $x_0$, with $C^{\infty}$ boundaries, such that the operator 
		\begin{equation}\label{Pi_V eq}
			\widetilde{\Pi}_Vu(x)=\frac{1}{h^n}\int_V e^{\frac{2}{h}\Psi(x, \ybar)} a(x, \ybar; h) u(y) e^{-\frac{2}{h}\Phi(y)} L(dy)
		\end{equation}
		satisfies 
		\begin{equation}\label{reproducing property eq}
			\widetilde{\Pi}_V-1=\OO(1)\exp\bigl(-C^{-1}h^{-1/(2s-1)}\bigr): H_{\Phi}(V)\rightarrow L^2(U, e^{-2\Phi/h}L(dx)).
		\end{equation}
		Here in \eqref{Pi_V eq}, the function $\Psi\in \mathcal{G}^s(V\times \oo{V})$ is holomorphic to $\infty$--order along the anti-diagonal $\widetilde{y}=\xbar$, $\Psi(x,\xbar)=\Phi(x)$, and $L(dy)$ is the Lebesgue measure on $\mathbb{C}^n$.
	\end{thm}
	
	\begin{rem}
		The existence of an almost holomorphic extension $\Psi$ in $\mathcal{G}^s$, whose Gevrey order is the same as that of $\Phi$, is due to Carleson \cite{carleson1961universal} (see also \cite{GuedesFBIGevrey}). We shall review the relevant results in Section \ref{subsection:almost holomorphic extension}.  
	\end{rem}	
	
	Once the local approximate reproducing property in Theorem \ref{main thm} has been established, a
	global version, namely an approximation for the Bergman projection $\Pi$ or the Bergman kernel $K(x,y)$ (uniformly in any compact subset of $\Omega$) follows from the classical cut-and-paste arguments and, in particular, the H\"ormander's $L^2$--estimates for the $\overline{\partial}$--equations. Such arguments have already been developed carefully in \cite{RoSjNg20}, see also \cite{BBSj08}, \cite{deleporte2022analytic}, \cite{hitrik2022smooth}. For convenience of the reader, we present the following corollary to Theorem \ref{main thm}, showing that the distribution (Schwartz) kernel of the authentic orthogonal (Bergman) projection $\Pi$ is locally approximated by the kernel of the operator $\widetilde{\Pi}_V$ defined in \eqref{Pi_V eq}, up to a $\GG^s$--type small error.
	\begin{corr}\label{corr: asymp to exact Bergman}
		Let $\Omega\subset \mathbb{C}^n$ be a pseudoconvex open set and let $\Phi\in \mathcal{G}^s(\Omega; \mathbb{R})$ with $s>1$ be strictly plurisubharmonic in $\Omega$. Let $K(x,\ybar)e^{-2\Phi(y)/h}$ be the distribution kernel of the orthogonal (Bergman) projection $\Pi: L^2(\Omega, e^{-2\Phi/h}L(dx)) \to H_\Phi(\Omega)$. For any $x_0\in \Omega$, there exists an open neighborhood $\widetilde{U} \Subset U\Subset V \Subset\Omega$ of $x_0$, with $U, V$ as in Theorem \ref{main thm}, such that
		\begin{equation*}
			e^{-\frac{\Phi(x)}{h}} \Bigl(K(x,\ybar) - \frac{1}{h^n}e^{\frac{2}{h}\Psi(x,\ybar)} a(x,\ybar;h)\Bigr) e^{-\frac{\Phi(y)}{h}} = \OO(1) \exp\bigl(-C^{-1}h^{-1/(2s-1)}\bigr),\quad C>0,
		\end{equation*}
		uniformly for $x, y\in\widetilde{U}$. Here the phase $\Psi$ and the amplitude $a$ have been introduced in Theorem \ref{main thm}.
	\end{corr}
	To align with \eqref{off-diagonal expansion eq} and \eqref{off-diagonal expansion analytic eq}, we have the following asymptotic expansion which follows from Theorem \ref{main thm} and Corollary \ref{corr: asymp to exact Bergman},
	\begin{rem}
		For any $x, y\in \widetilde{U}\Subset\Omega$, the Bergman kernel $K(x, \ybar)$ associated to the orthogonal projection $\Pi: L^2(\Omega, e^{-2\Phi/h}L(dx)) \to H_\Phi(\Omega)$ satisfies for some constants $C, \delta>0$,
		\begin{equation}\label{off-diagonal expansion Gevrey eq}
			K(x, \ybar) = \frac{1}{h^n} e^{\frac{2}{h}\Psi(x, \ybar)} \sum_{j=0}^{[(Ch)^{-\frac{1}{2s-1}}]} a_j(x, \ybar)h^j + \OO(1)e^{-\delta h^{-\frac{1}{2s-1}}} e^{\frac{\Phi(x)+\Phi(y)}{h}}.
		\end{equation}
		Here $\Psi$ and $a_j$'s are as in Theorem \ref{main thm}.
	\end{rem}
	
	By taking $s\rightarrow 1^{+}$, the asymptotic expansion \eqref{off-diagonal expansion Gevrey eq} formally recovers that for the real analytic weights. Moreover, the above results allow us to give satisfactory answers to the aforementioned questions. For question (a), the infinite sum $\sum_{j=0}^{\infty} a_j(x, \widetilde{y}) h^j$ is in fact a formal $\GG^{s,2s-1}$ Gevrey symbol. Furthermore, this result is optimal in the sense that there exists $\Phi\in\GG^s$ such that its corresponding Bergman kernel $\sum_{j=0}^{\infty} a_j h^j$ does not lie in the symbol class $\GG^{s,\sigma}$ for any $\sigma<2s-1$, as discussed in Remark \ref{optimality rmk}. For question (b), we have achieved an $\OO(e^{-\delta h^{-\frac{1}{2s-1}}})$ error in the asymptotic expansion \eqref{off-diagonal expansion Gevrey eq}, which is an interpolating result between the $\OO(e^{-\delta h^{-1}})$ error for real analytic weights and the $\OO(h^\infty)$ error for smooth weights.
	
	\begin{rem}
		We mention that our analysis still works if we replace the Lebesgue measure by a general one in the form of $\mu(x)L(dx)$ for some $\mu\in \mathcal{G}^s(\Omega)$. In this case, we can prove similar results for the Bergman kernel associated to the orthogonal projection $\Pi: L^2(\Omega, e^{-2\Phi/h}\mu\, L(dx)) \rightarrow \Hol(\Omega)\cap L^2(\Omega, e^{-2\Phi/h}\mu\, L(dx))$. 
		
		With this generalization, we can further adapt our analysis to the setting of Bergman kernels for tensor powers of a positive line bundle. Let $(L, h)\rightarrow M$ be a positive Hermitian line bundle over a compact complex manifold $M$ of dimension $n$ and let $\omega=\mathrm{Ric}(h)$ be the induced K\"ahler form on $M$. For any $k\in \mathbb{N}$, consider the Bergman kernel $K_k(x, \ybar)$ for the orthogonal projection $\Pi_k: L^2(M, L^k) \rightarrow H^0(M, L^k)$ with respect to the nature inner product induced by the metric $h^k$ and the volume form $\omega^n/n!$. By repeating the argument in the proof of Corollary \ref{corr: asymp to exact Bergman} (cf. \cite[Section 3]{BBSj08}), we can prove
		\begin{corr}
			Let $s>1$. Suppose $V\subset M$ is open and $h=e^{-2\Phi}$ for some $\Phi\in \mathcal{G}^s(V)$. For any $x_0\in V$, there exist a small open neighborhood $U\Subset V$ of $x_0$ and positive constants $C$ and $\delta$, such that for any $x, y\in U$, the Bergman kernel $K_k(x, \ybar)$ satisfies when $k\rightarrow \infty$,
			\begin{equation*}
				K_k(x, \ybar)=k^ne^{2k\Psi(x, \ybar)} \sum_{j=0}^{[(k/C)^{\frac{1}{2s-1}}]} \frac{a_j(x, \ybar)}{k^j} + \OO(1) e^{-\delta k^{1/(2s-1)}} e^{k\left(\Phi(x)+\Phi(y)\right)}.
			\end{equation*}
		\end{corr}
		
		In particular, if we denote by $|K_k(x, \ybar)|_{h^k}$ the pointwise norm of $K_k(x, \ybar)$ with respect to the Hermitian metric $h^k$, then by repeating the computation in \cite[Corollary 1.2]{HLX20} one obtains $|K_k(x, \ybar)|_{h^k}=\frac{k^n}{\pi^n}e^{-kD(x,y)}\bigl(1+\OO(\frac{1}{k})\bigr)$ when $d(x,y)\leq \sqrt{\delta}k^{-\frac{s-1}{2s-1}}$, where $D(x,y)=\Phi(x)+\Phi(y)-2\Real\Psi(x, \ybar)$. This in fact settles \cite[Conjecture 1.6]{HX20} for Gevrey weights.
	\end{rem}
	When proving Theorem \ref{main thm}, we closely follow the framework established in \cite{deleporte2022analytic}. The method in \cite{deleporte2022analytic} gives a more direct route for proving \eqref{off-diagonal expansion analytic eq} in the context of real analytic weights. Notably, it avoids an implicit change of coordinates, involving the Kuranishi trick which transforms $\Psi$DOs expressed with non-standard phase to the standard one. We also mention that \cite{hitrik2022smooth}  implemented this approach and proved \eqref{off-diagonal expansion eq} for smooth weights. While this method is well-established, dealing with the Gevrey weights presents significant challenges. One primary challenge is to quantitatively estimate the growth rate of the Bergman coefficients $a_j$ in \eqref{amplitude asymp aj_s}. The work \cite{deleporte2022analytic} used the deep theory about analytic microlocal analysis in \cite{Sj82} to show that $\sum a_j h^j$ forms a classical analytic symbol. However, the counterpart theory in the setting of Gevrey weights is currently lacking in the literature, to the best of the authors' knowledge. In our work, we adopt a more elementary way to inductively estimate $a_j$ based on a recursive formula observed by L. Charles \cite{charles2003berezin} (see also \cite{hezari2021property}). While several other recursive formulas exist for the coefficients $a_j$ (see e.g., \cite{En00, BBSj08, hitrik2022smooth}), they seem to involve subtle cancellation, hindering the derivation of desired estimates. Throughout the proof, we have also extended some microlocal analysis tools to the Gevrey setting. Notably, we introduced the Gevrey Borel lemma in Proposition \ref{prop:Borel lemma Gevrey} and the Gevrey stationary phase method (Theorem \ref{stationary phase lemma Gevrey lem}). These extensions are believed to be of some independent interest.
	
	Our paper is organized as follows: In Section \ref{Sec Gevrey microlocal analysis}, we adapt some microlocal analysis tools to the Gevrey setting, including almost holomorphic extensions, Borel's theorem, and the complex stationary phase lemma, which serve as key ingredients in our analysis. Section \ref{Sec construction of the amplitude} provides a review of the framework established in \cite{deleporte2022analytic, hitrik2022smooth}. We complete the construction of the symbol $a(x, \widetilde{y}; h)$ by quantitatively estimating the growth rate of the Bergman coefficients. In Section \ref{Sec proof of the main theorem}, we establish the weak local reproducing property and utilize it to prove Theorem \ref{main thm} with the help of H\"ormander's $L^2$ estimates. Appendix \ref{Appendix: combinatorial facts} reviews some elementary combinatorial facts, which are used throughout, while Appendix \ref{Appendix: prove corollary} proves Corollary \ref{corr: asymp to exact Bergman}.
	
	
	The literature on Bergman kernels is vast. Before concluding the introduction, we mention that some more related work can be found in \cite{Ch03, En02, HX20, MM07, SZ02} and the references therein. The setting of Gevrey regularity in the study of (pseudo)differential operators and PDEs has attracted considerable attention in recent years. See for instance the works concerning pseudodifferential operators with Gevrey symbols in \cite{HLSjZe23}, FBI transform in Gevrey classes in \cite{GuedesFBIGevrey}, and Gevrey WKB method in \cite{Lascar2023gevrey}. 
	
	\noindent
	\textbf{Acknowledgment.} We are deeply grateful to Michael Hitrik for the encouragement and valuable suggestions throughout this work. Additionally, we would like to thank Laurent Charles and Alix Deleporte for their helpful comments and suggestions. The second author expresses sincere thanks to Hamid Hezari for the consistent support and many insightful discussions on related topics.

	\section{Microlocal analysis tools in Gevrey classes}\label{Sec Gevrey microlocal analysis}
	
	In this section, we extend several fundamental tools in microlocal analysis to the Gevrey setting. These include almost holomorphic extensions in the Gevrey class, Gevrey symbols and Borel's theorem, and the complex stationary phase method for Gevrey functions. For a comprehensive monograph on Gevrey microlocal analysis, we refer readers to \cite{Ro93}.
	
	\subsection{Almost holomorphic extensions of Gevrey functions}\label{subsection:almost holomorphic extension}
	
	Let $U\subset\RR^d$ be open and let $\widetilde{U}$ be a neighborhood of $U$ in $\CC^d$ such that $U = \widetilde{U}\cap\RR^d$. We recall that $\widetilde{f}\in\CI(\widetilde{U})$ is called an \emph{almost holomorphic extension} of $f\in\CI(U)$ if $\widetilde{f}$ coincides with $f$ on $U$ and $\overline{\partial}\widetilde{f}$ vanishes on $U$ to all orders, i.e., $\overline{\partial}\widetilde{f}$ is flat on $U$. It is a well-known fact, originally due to Carleson \cite[Theorem 2 and Example 2]{carleson1961universal}, that Gevrey functions admit almost holomorphic extensions within the same Gevrey class. This result will play a crucial role in our analysis. Before stating the precise version of the result on almost holomorphic extensions of Gevrey functions that we will use, let us first introduce some notions about Gevrey functions. 
	
	For every $R>0$, $V\subset U$, and $u\in\CI(U)$, we define the semi-norm 
	\begin{equation*}
		\label{Gevrey semi-norm}
		\|u\|_{s,R,V} := \sup_{\substack{x\in V ,\, \alpha\in\NN^d}} \frac{|\partial^\alpha u(x)|}{R^{|\alpha|} \alpha!^s}.
	\end{equation*}
	We also introduce the (global) Gevrey-$s$ class, $\GG_{\rm b}^s(U)\subset\GG^s(U)$, the space of functions $u\in\CI(U)$ for which we have
	\begin{equation}
		\label{Gevrey b estimate}
		\exists A,\,C >0,\quad\text{s.t.}\quad\forall\alpha\in\NN^d,\, x\in U,\quad | \partial^\alpha u(x)| \leq A C^{|\alpha|} \alpha!^s .
	\end{equation}
	
	Let us now recall the following version from \cite[Lemma 1.2]{GuedesFBIGevrey}:
	\begin{prop}
		\label{prop:Gevrey almost holo extension}
		Let $s>1$. Let $U$ be an open subset of $\RR^d$ and let $\widetilde{U}$ be a complex neighborhood of $U$ in $\CC^d$ such that $U = \widetilde{U}\cap\RR^d$. If $u\in\GG_{\rm b}^s(U)$, then $u$ has an almost holomorphic extension $\widetilde{u}\in\GG_{\rm b}^s(\widetilde{U})$. Moreover, there exist constants $C, M>0$ depending only on $U$, $\widetilde{U}$, $d$ and $s$, such that for each $R>0$,
		\begin{equation}
			\label{estimate almost holo ext Gevrey semi-norm}
			\text{if}\quad\|u\|_{s,R,U}<+\infty,\quad \text{then}\quad\|\widetilde{u}\|_{s,MR,\widetilde{U}} \leq C \|u\|_{s,R,U} .
		\end{equation}
	\end{prop}
	Knowing that $\widetilde{u}\in\GG_{\rm b}^s(\widetilde{U})$ is almost holomorphic along $\RR^d$ will allow one to measure the flatness of $\overline{\partial}\widetilde{u}$. In fact, we have the following:
	\begin{prop}
		\label{prop:flatness estimtate}
		Let $s, U, \widetilde{U}$ be as in Proposition \ref{prop:Gevrey almost holo extension}. Suppose that $u\in\GG_{\rm b}^s(U)$ satisfies $\|u\|_{s,R,U} < +\infty$ for some $R>0$. Let $\widetilde{u}\in \GG_{\rm b}^s(\widetilde{U})$ be an almost holomorphic extension of $u$ given in Proposition \ref{prop:Gevrey almost holo extension}, satisfying \eqref{estimate almost holo ext Gevrey semi-norm}. Then there exist constants $C=C(U,\widetilde{U})>0$ and $C_1 = C_1(U,\widetilde{U},R)>0$ such that
		\begin{equation*}
			\label{dbar tilde u estimate}
			\left| \overline{\partial} \widetilde{u}(z) \right| \leq C_1 \|u\|_{s,R,U}\exp\left(-C (R|\Im z|)^{-\frac{1}{s-1}}\right),\quad \mbox{for any } z\in\widetilde{U} .
		\end{equation*}
	\end{prop}
	This can be deduced from the $\GG^s$-type estimate \eqref{Gevrey b estimate} for $\widetilde{u}$, using Taylor's expansions of $\overline{\partial}u$ at $z=x\in U$ and Lemma \ref{lem:Gevrey minimize}. For a complete proof we refer to \cite[Remark 1.7]{GuedesFBIGevrey}. 
	
	\subsection{Gevrey symbols and Borel's theorem}
	\label{subsection:Borel lemma}
	\begin{defi}
		\label{defi:formal Gevrey symbol}
		Let $U\subset\RR^d$ be open and let $s,\,\sigma\geq 1$. Let $\{a_j\}_{j=0}^\infty$ be a sequence of functions in $\GG_{\rm b}^s(U)$. We say that the sequence $\{a_j\}_{j=0}^\infty$ is a formal $\GG^{s,\sigma}$ symbol on $U$ if there exists $C>0$ such that
		\begin{equation}
			\label{formal Gevrey symbol condition}
			\|\partial^\alpha a_j \|_{L^\infty(U)} \leq C^{1+|\alpha|+j} \alpha!^s j!^{\sigma}, \quad \mbox{for any } j\in\NN, \ \alpha\in\NN^d.
		\end{equation}  
	\end{defi}
	
	Let us prove a Gevrey version of Borel's theorem showing that formal $\GG^{s,\sigma}$ symbols can be realized as Gevrey functions in the following sense:
	\begin{prop}
		\label{prop:Borel lemma Gevrey}
		Let $U\subset\RR^d$ be open and let $s,\,\sigma\geq 1$. Suppose that $\{a_j\}_{j=0}^\infty$ is a formal $\GG^{s,\sigma}$ symbol satisfying \eqref{formal Gevrey symbol condition}. Then there exist $a=a(\cdot \,; h)\in\GG_{\rm b}^s(U)$ and $\widetilde{C}>0$ such that for any $\alpha\in\NN^d$, $N>0$ we have uniformly
		\begin{equation}
			\label{realize Gevrey symbol}
			\biggl\| \partial^\alpha \bigl(a - \sum_{j=0}^{N-1} a_j h^j\bigr) \biggr\|_{L^\infty(U)} \leq \widetilde{C}^{1+|\alpha|+N} \alpha!^s N!^\sigma h^N .
		\end{equation}
	\end{prop}
	
	\begin{proof}
		For $h>0$, let us set $N_0 = [ (C h)^{-1/\sigma} ]$ and define
		\begin{equation}
			\label{eqn:a realizing a_j}
			a(x;h) := \sum_{j=0}^{N_0-1} a_j(x) h^j.
		\end{equation}
		If $N>N_0$, then $ a-\sum_{j=0}^{N-1} a_j h^j = -\sum_{j=N_0}^{N-1} a_j h^j$. It follows from \eqref{formal Gevrey symbol condition} that
		\[
		\biggl| \partial^\alpha \big(a - \sum_{j=0}^{N-1} a_j h^j\big) \biggr| \leq \sum_{j=N_0}^{N-1} |\partial^\alpha a_j| h^j \leq C^{1+|\alpha|} \alpha!^s \sum_{j=N_0}^{N-1} C^{j} j!^\sigma h^j .
		\]
		Recalling again \eqref{seq a_k decrease increase} that $C^{j} j!^\sigma h^j \leq C^{N} N!^\sigma h^N$ if $N_0\leq j\leq N-1$, we obtain therefore
		\[
		\biggl| \partial^\alpha \big(a - \sum_{j=0}^{N-1} a_j h^j\big) \biggr| \leq C^{1+|\alpha|+N} (N-N_0) \alpha!^s N!^\sigma h^N . 
		\]
		The inequality \eqref{realize Gevrey symbol} follows by letting $\widetilde{C} = 2C$, since $N < 2^N$.
		
		Let us consider the case where $N\leq N_0$. If further $N\geq N_0 /2$, the above strategy still applies. We note that $ a-\sum_{j=0}^{N-1} a_j h^j = \sum_{j=N}^{N_0-1} a_j h^j$, then by \eqref{formal Gevrey symbol condition},
		\[
		\biggl| \partial^\alpha \big(a - \sum_{j=0}^{N-1} a_j h^j\big) \biggr| \leq \sum_{j=N}^{N_0-1} |\partial^\alpha a_j| h^j \leq C^{1+|\alpha|} \alpha!^s \sum_{j=N}^{N_0-1} C^{j} j!^\sigma h^j .
		\]
		Recalling \eqref{seq a_k decrease increase} that $C^{j} j!^\sigma h^j \leq C^{N} N!^\sigma h^N$ when $N\leq j < N_0$, and noticing that $N_0 - N \leq N < 2^N$, we get
		\[
		\biggl| \partial^\alpha \big(a - \sum_{j=0}^{N-1} a_j h^j\big) \biggr| \leq C^{1+|\alpha|+N} (N_0-N) \alpha!^s N!^\sigma h^N \leq (2C)^{1+|\alpha|+N} \alpha!^s N!^\sigma h^N .
		\]
		It remains to justify \eqref{realize Gevrey symbol} when $N < N_0 /2$. We recall \eqref{eqn:a realizing a_j}, \eqref{formal Gevrey symbol condition}, and write
		\[
		\frac{1}{C^{1+|\alpha|} \alpha!^s} \biggl| \partial^\alpha \big(a - \sum_{j=0}^{N-1} a_j h^j\big) \biggr| \leq  \sum_{j=N}^{N_0-1} C^{j} j!^\sigma h^j 
		= C^{N} h^N \sum_{k=0}^{N_0-N-1} (C h)^k {(N+k)!^\sigma} .
		\]
		Using the AM--GM inequality: $n^{-1}(a_1+\cdots+a_n)\geq \sqrt[n]{a_1\cdots a_n}$, we get
		\[
		(N+1)\cdots(N+k) \leq ((2N+1+k)/2)^k \leq ((N+N_0)/2)^k,\quad 1\leq k \leq N_0-N-1 .
		\]
		Noting that $C h \leq N_0^{-\sigma}$, we obtain therefore, for $1\leq k\leq N_0-N-1$,
		\[
		(C h)^k (N+k)!^\sigma N! \leq N!^\sigma N_0^{-k\sigma} ((N+N_0)/2)^{k\sigma} = N!^\sigma ((N+N_0)/2N_0)^{k\sigma},
		\]
		and this holds trivially (as an equality) when $k=0$. Recalling $N<N_0 /2$, we conclude 
		\[
		\biggl| \partial^\alpha \big(a - \sum_{j=0}^{N-1} a_j h^j\big) \biggr| \leq C^{1+|\alpha|+N} \alpha!^s N!^\sigma h^N \sum_{k=0}^{N_0-N-1} (3/4)^{k\sigma} \leq 4 C^{1+|\alpha|+N} \alpha!^s N!^\sigma h^N,
		\]
		where we used $\sum (3/4)^{k\sigma} \leq \sum (3/4)^{k} = 4$. This proves \eqref{realize Gevrey symbol} by setting $\widetilde{C} = 4C$.
	\end{proof}
	
	\subsection{Complex stationary phase for Gevrey functions}\label{subsection:Gevrey stationary phase}
	
	The method of stationary phase is a fundamental tool in microlocal analysis. When both the phase and the amplitude have Gevrey regularity, one can obtain an asymptotic expansion in certain Gevrey symbol class, as defined in Definition \ref{defi:formal Gevrey symbol}. The method for real-valued Gevrey phase functions has been justified in \cite{GuedesFBIGevrey} and \cite{Lascar2023gevrey}, using the Morse lemma for coordinate change. However, this proof is not applicable when the phase function is complex-valued. In this section, we shall establish the method of stationary phase for the case where both the phase and the amplitude are complex-valued Gevrey functions, by adapting \cite[Theorem 7.7.5]{hormander1985analysis} to the Gevrey setting. This will serve as a key ingredient in the proof of Theorem \ref{thm:asymp invert A}.
	
	\begin{thm}\label{stationary phase lemma Gevrey lem}
		Let $s>1$, $k\in\NN$, and $h>0$ be a small parameter. Let $W\subset\RR^d$ be open. Suppose that $u, f\in \GG^s(W)$ and that $f(x_0)=0$, $f'(x_0)=0$, $\Im f''(x_0)$ is positive definite for $x_0\in W$. Then there exists a small open neighborhood $U\Subset W$ of $x_0\in W$ such that 
		\begin{equation}
			\label{Gevrey stationary phase}
			\biggl|\int_U e^{\frac{i}{h} f(x)} u(x)\,dx - \biggl(\det\Bigl(\frac{f''(x_0)}{2\pi i h}\Bigr)\biggr)^{-\frac{1}{2}}\sum_{j<k} h^j L_j u\biggr| 
			\leq C^{k+1} k!^{2s-1} h^{k+\frac{d}{2}}.
		\end{equation}
		Here the constant $C$ depends only on the dimension $d$ and the Gevrey constants in \eqref{Gevrey b estimate} for $u, f\in\GG_{\rm b}^s(U)$. Denoting by
		\begin{equation}
			\label{eqn:g_x_0 in Section 4}
			g_{x_0}(x) = f(x) - \frac{1}{2}f''(x_0)(x-x_0)\cdot (x-x_0),
		\end{equation}
		which vanishes of third order at $x_0$, we have
		\begin{equation}
			\label{eqn:Lj}
			L_j u = \sum_{\nu-\mu=j}\sum_{2\nu\geq 3\mu} \frac{i^{\mu+\nu}}{2^{\nu} \mu! \nu!} (f''(x_0)^{-1}\partial\cdot \partial)^\nu (g_{x_0}^\mu u)(x_0) .
		\end{equation}
	\end{thm}
	
	\begin{proof}
		For simplicity, we may assume that $x_0 = 0$ throughout the proof. Additionally, we will temporarily assume an extra technical condition $k < h^{-\frac{1}{2s-1}}$, and remove it at the end of the proof.
		
		We first note that under our assumptions on $f$, for a sufficiently small neighborhood $U$ of $x_0$, there exists $C_0>0$ such that
		\begin{equation}
			\label{Im f}
			\Im f(x)\geq C_0 |x|^2,\quad x\in U.
		\end{equation}
		
		$\bullet$ \emph{Introduce a cutoff.} Let us choose a cutoff function $\rho\in \CIc(U)$ such that
		\begin{equation}
			\label{supp rho}
			\rho(x) = 1,\ \text{when }|x|\leq h^{\frac{s-1}{2s-1}};\quad \rho(x)=0,\ \text{when }|x|\geq 2h^{\frac{s-1}{2s-1}}.
		\end{equation}
		It follows from \eqref{Im f} and \eqref{supp rho} that
		\begin{equation}
			\label{exp decay}
			\left|e^{\frac{i}{h} f(x)}\right| = e^{-\frac{\Im f(x)}{h}} \leq e^{-C_0 h^{\frac{2s-2}{2s-1}-1}} = \exp(-C_0 h^{-\frac{1}{2s-1}}),\quad x\in U\cap\supp(1-\rho).
		\end{equation}
		Therefore, for $0<h<1$, we derive from \eqref{exp decay} that
		\begin{equation}
			\label{eqn:outside cutoff}
			\begin{split}
				\left|\int_U e^{\frac{i}{h} f(x)}(1-\rho(x)) u(x)\,dx\right| &\leq \exp\left(-C_0 h^{-\frac{1}{2s-1}}/2\right) \|u\|_{L^\infty(U)} \int_{\RR^d} e^{-\frac{C_0 |x|^2}{2h}} dx \\
				&\leq C h^{d/2}\exp\left(-C_0 h^{-\frac{1}{2s-1}}/2\right).
			\end{split}
		\end{equation}
		
		$\bullet$ \emph{Deform the phase to a quadratic form.} Let us set for $0\leq\sigma \leq 1$,
		\[
		f_\sigma(x) = \frac{1}{2}f''(0)x\cdot x + \sigma g_{0}(x).
		\]
		We remark that $f_1(x) = f(x)$ and $f_0(x)$ is a quadratic form in $x$ with $\Im f_0 > 0$. In view of Taylor's expansion of $f(x)$ at $x=0$, we have $g_{0}(x) = \OO(|x|^3)$, and thus we may assume that $U$ is sufficiently small such that, by possibly shrinking $C_0$,
		\begin{equation}
			\label{Im f_sigma}
			\Im f_\sigma(x)\geq C_0|x|^2,\quad x\in U,\quad 0\leq \sigma \leq 1. 
		\end{equation}
		Let us now consider
		\[
		I(\sigma) := \int_U e^{\frac{i}{h} f_\sigma(x)} \rho(x) u(x)\,dx.
		\]
		Recalling Taylor's formula we have
		\[
		\biggl| I(1) - \sum_{\mu<2k} \frac{1}{\mu!} I^{(\mu)}(0) \biggr| \leq \frac{1}{(2k)!} \sup_{0\leq \sigma \leq 1} |I^{(2k)}(\sigma)|,
		\]
		where
		\[
		I^{(2k)}(\sigma) = (-1)^k h^{-2k} \int e^{\frac{i}{h}f_\sigma(x)}\rho(x) u(x) g_{0}^{2k}(x) dx.
		\]
		Since $|g_{0}(x)|\leq C|x|^3$ for $x\in U$, using \eqref{basic inequality 1} and \eqref{Im f_sigma} we get
		\[
		\begin{split}
			\left| e^{\frac{i}{2h}f_\sigma(x)} \rho(x)u(x) g_0^{2k}(x) \right| \leq C^{2k} |x|^{6k} e^{-\frac{C_0 |x|^2}{2h}} \|u\|_{L^\infty} \leq C^{2k} (C_0/2)^{-3k} \|u\|_{L^\infty} (3k)! h^{3k} .
		\end{split}
		\]
		It then follows that, in view of \eqref{Im f_sigma},
		\[
		\begin{split}
			\left|I^{(2k)}(\sigma)\right| \leq C^{2k} (C_0/2)^{-3k} \|u\|_{L^\infty} (3k)! h^{k} \int_{\RR^d} e^{-\frac{C_0|x|^2}{2h}} dx \leq C_1^{k+1} (3k)! h^{k+\frac{d}{2}} .
		\end{split}
		\]
		Therefore, noting that $\displaystyle \frac{(3k)!}{(2k)!k!} = \binom{3k}{k} \leq 2^{3k}$, we get, by redefining $C=2^3 C_1$,
		\begin{equation}
			\label{est:fix quadratic phase}
			\biggl| I(1) - \sum_{\mu<2k} \frac{1}{\mu!} I^{(\mu)}(0) \biggr| \leq C_1^{k+1} \frac{(3k)!}{(2k)!} h^{k+\frac{d}{2}} \leq C^{k+1} k! h^{k+\frac{d}{2}}.
		\end{equation}
		
		$\bullet$ \emph{Taylor expansion of $g_0^\mu u$: remainder.} Let us now fix $0\leq\mu<2k$ and write
		\[
		I^{(\mu)}(0) = i^{\mu} h^{-\mu} \int e^{\frac{i}{h}f_0(x)}\rho(x) u(x) g_{0}^\mu (x) dx.
		\]
		We introduce the Taylor expansion of order $2k+2\mu$ at $x=0$ of $g_{0}^\mu u$,
		\[
		T_{2k+2\mu}(g_{0}^\mu u)(x) = \sum_{|\alpha|<2k+2\mu} \frac{\partial^\alpha(g_{0}^\mu u)(0)}{\alpha!} x^\alpha,
		\]
		with the remainder $R_{2k+2\mu} (g_{0}^\mu u) =  g_{0}^\mu u - T_{2k+2\mu}(g_{0}^\mu u)$ given by
		\[
		R_{2k+2\mu} (g_{0}^\mu u)(x) = \sum_{|\beta|=2k+2\mu} \frac{|\beta|x^\beta}{\beta!} \int_0^1 (1-t)^{|\beta|-1}\partial^\beta(g_{0}^\mu u)(tx)dt. 
		\]
		It follows that
		\begin{equation}
			\label{remainder estimate}
			|R_{2k+2\mu} (g_{0}^\mu u)(x)|\leq |x|^{2k+2\mu} \sum_{|\beta|=2k+2\mu} \sup_{|y|\leq |x|} \left| \frac{\partial^\beta (g_{0}^\mu u) (y)}{\beta!} \right|.
		\end{equation}
		By the Leibniz rule we write
		\begin{equation}
			\label{Leibniz rule}
			\frac{\partial^\beta (g_{0}^\mu u)}{\beta!} = \sum_{\beta_0 + \cdots+ \beta_\mu = \beta} \frac{\partial^{\beta_0}u \partial^{\beta_1}g_{0} \cdots \partial^{\beta_\mu}g_{0} }{\beta_0! \beta_1! \cdots \beta_\mu!}.
		\end{equation}
		Recalling that $f\in \GG^s(W)\implies g_{0}\in \GG^s(W)$, and $g_{0}(x) = \OO(|x|^3)$, we have
		\begin{equation}
			\label{eqn:Dbetaj}
			\frac{|\partial^{\beta_j} g_{0}(x)|}{\beta_j !} \leq \begin{cases}
				C^{1+|\beta_j|}\beta_j!^{s-1},&\quad\text{if }|\beta_j|\geq 3; \\
				C |x|^{3-|\beta_j|},&\quad\text{if }|\beta_j|<3.
			\end{cases}
		\end{equation}
		Let us fix any $\beta_0,\beta_1,\cdots,\beta_\mu\in \NN^d$, with $|\beta_0|+\cdots + |\beta_\mu|=2k+2\mu$, and denote by $S=\{ 1\leq j\leq\mu : |\beta_j|< 3 \}$. Using \eqref{eqn:Dbetaj} and $u\in\GG^s(W)$ we obtain 
		\begin{equation}
			\label{estimtate product}
			\biggl| \frac{\partial^{\beta_0}u(y)}{\beta_0!} \prod_{j=1}^\mu \frac{\partial^{\beta_j}g_{0}(y)}{\beta_j!}\biggr| \leq C^{k+1} |\beta_0|!^{s-1} |y|^{\sum_{j\in S} (3-|\beta_j|)}\prod_{j\notin S} |\beta_j|!^{s-1}.
		\end{equation}
		We note that 
		\[
		\sum_{j\notin S}|\beta_j| = 2k+2\mu - |\beta_0| - \sum_{j\in S}|\beta_j|,\quad\text{and}\quad |\beta_j|\geq 3,\ j\notin S. 
		\]
		Using Corollary \ref{cor:factorial product max}, and noting that $\# \{1\leq j\leq\mu : j\notin S\} = \mu - |S|$, we have
		\begin{equation}
			\label{betaj factorials j notin S}
			\begin{split}
				\prod_{j\notin S}|\beta_j|! &\leq 3!^{\mu-|S|-1} \Bigl(2k+2\mu - |\beta_0| - \sum_{j\in S}|\beta_j| - 3(\mu-|S|-1) \Bigr)! \\
				&\leq 6^{2k} \Bigl(2k-\mu-|\beta_0|+3 + \sum_{j\in S}(3-|\beta_j|) \Bigr)!,
			\end{split}
		\end{equation}
		where we used $\mu < 2k$ in the last inequality. Recalling \eqref{supp rho} and the assumption $k< h^{-\frac{1}{2s-1}}$, we see that
		\[
		x\in\supp \rho \implies |x| < 2 k^{-(s-1)}.
		\]
		Writing $l_S:= \sum_{j\in S}(3-|\beta_j|)$ (thus $l_S\leq 3\mu \leq 6k-3$), we then deduce from \eqref{estimtate product} and \eqref{betaj factorials j notin S} that for any $|y|\leq |x|$ with $x\in \supp\rho$,
		\begin{equation}
			\label{estimate product 2}
			\begin{split}
				\biggl| \frac{\partial^{\beta_0}u(y)}{\beta_0!} \prod_{j=1}^\mu \frac{\partial^{\beta_j}g_{0}(y)}{\beta_j!}\biggr| &\leq C^{k+1}2^{6k}6^{2k(s-1)} \left(k^{-l_S} |\beta_0|! (2k-\mu-|\beta_0|+3+l_S)!\right)^{s-1} \\
				&\leq C^{k+1} \left(k^{-l_S}(2k-\mu +3 + l_S)!\right)^{s-1}.
			\end{split}
		\end{equation}
		Here we used a basic inequality $a!b!\leq (a+b)!$ and renamed $2^6 6^{2(s-1)} C$ to be $C$. Using the basic inequality $N!\leq N^N$, and $l_S \leq 6k-3$, we see that
		\[
		(2k-\mu +3 + l_S)! \leq (2k-\mu+3+l_S)^{2k-\mu +3 + l_S} \leq 8^{8k} k^{2k-\mu +3 + l_S}.
		\]
		Stirling's approximation \eqref{Stirling apprx} implies that $k^k < e^k k!$, noting that $k^3\leq 3^k$, we get 
		\[
		k^{-l_S}(2k-\mu +3 + l_S)! \leq 8^{8k}k^{2k-\mu+3} \leq (8^8 3)^k (k^k)^2 < (8^8 3 e^2)^k k!^2 .
		\]
		Therefore \eqref{estimate product 2} yields that, for any $|y|\leq|x|$ with $x\in\supp\rho$,
		\begin{equation}
			\label{estimtae product 3}
			\biggl| \frac{\partial^{\beta_0}u(y)}{\beta_0!} \prod_{j=1}^\mu \frac{\partial^{\beta_j}g_{0}(y)}{\beta_j!}\biggr| \leq C^{k+1} (8^8 3 e^2)^{(s-1)k} k!^{2s-2} \leq C^{k+1} k!^{2s-2} ,
		\end{equation}
		where we renamed $(8^8 3 e^2)^{s-1}C$ to be $C$ in the last step. By \eqref{partition of integer} and $\mu<2k$ we have 
		\[
		\#\{\beta\in\NN^d : |\beta|= 2k + 2\mu\} = \binom{2k+2\mu+d-1}{2k+2\mu} \leq 2^{2k+2\mu+d-1} \leq 2^{d-3} 2^{6k} .
		\]
		It follows from Proposition \ref{prop:partition multi-index counting} that the number of terms on the right side of \eqref{Leibniz rule} equals $\# P(\beta,\mu+1)$, which satisfies, noting that $|\beta|=2k+2\mu$ and $\mu<2k$,
		\[
		\# P(\beta,\mu+1) = \binom{\beta + \mu\mathbbm{1}}{\beta} \leq 2^{|\beta| + d\mu} = 2^{2k+(d+2)\mu}\leq 2^{-(d+2)} 2^{(2d+6)k}.
		\]
		Therefore, replacing the constant $C$ in \eqref{estimtae product 3} by $4^{-(d+6)}C$, we conclude from \eqref{remainder estimate}, \eqref{Leibniz rule}, and \eqref{estimtae product 3} that
		\begin{equation}
			\label{R 2k+2mu pointwise bound}
			|R_{2k+2\mu} (g_{0}^\mu u)(x)| \leq C^{k+1} k!^{2s-2} |x|^{2k+2\mu},\quad x\in\supp\rho.
		\end{equation}
		Applying \eqref{basic inequality 1} alongside \eqref{R 2k+2mu pointwise bound} we get
		\[
		\left| e^{-\frac{C_0 |x|^2}{2h}} \rho(x) R_{2k+2\mu} (g_{0}^\mu u)(x) \right| \leq C^{k+1} (2/C_0)^{k+\mu} k!^{2s-2} (k+\mu)! h^{k+\mu},
		\] 
		which implies that, writing $C_1 = 2^3 C/C_0^3$,
		\[
		\begin{split}
			\left|\int e^{\frac{i}{h}f_0(x)}\rho(x) R_{2k+2\mu} (g_{0}^\mu u)(x) \,dx\right| &\leq C_1^{k+1} k!^{2s-2} (k+\mu)! h^{k+\mu} \int_{\RR^d} e^{-\frac{C_0|x|^2}{2h}}dx \\
			&\leq C_1^{k+1} C_2 k!^{2s-2} (k+\mu)! h^{k+\mu+\frac{d}{2}}.
		\end{split}
		\]
		It follows that for any $0\leq\mu<2k$,
		\begin{equation}
			\label{est:restrict to Taylor polynomials}
			\begin{split}
				{ }&\quad \biggl| \frac{I^{(\mu)}(0)}{\mu!} - \frac{i^{\mu} h^{-\mu}}{\mu!}\int e^{\frac{i}{h}f_0(x)}\rho(x) T_{2k+2\mu} (g_{0}^\mu u)(x)\,dx \biggr| \\
				&\leq C_1^{k+1} C_2 k!^{2s-2} \frac{(k+\mu)!}{\mu!} h^{k+\frac{d}{2}} \leq C_1^{k+1} C_2 2^{k+\mu} k!^{2s-1} h^{k+\frac{d}{2}} = C^{k+1} k!^{2s-1} h^{k+\frac{d}{2}},
			\end{split}
		\end{equation}
		where we used the elementary inequality $\displaystyle \frac{(k+\mu)!}{k!\mu!} \leq 2^{k+\mu}$, and redefine $C = 2^3 C_1 C_2$.
		
		$\bullet$ \emph{Integrate Taylor polynomials with quadratic phase: removing the cutoff.} We shall now estimate the following integrals for each $0\leq \mu<2k$,
		\[
		\frac{i^{\mu} h^{-\mu}}{\mu!} \int e^{\frac{i}{h}f_0(x)}(1-\rho(x)) T_{2k+2\mu} (g_{0}^\mu u)(x)\,dx.
		\]
		In view of \eqref{supp rho} and \eqref{Im f_sigma} we have
		\begin{equation}
			\label{est:1-rho and exp terms}
			\left| e^{\frac{i}{3h} f_0(x)}(1-\rho(x)) \right|\leq \exp(-(C_0/3) h^{-\frac{1}{2s-1}}).
		\end{equation}
		Let us recall \eqref{Leibniz rule} and write		
		\begin{equation}
			\label{Leibniz 2}
			\frac{\partial^\alpha(g_{0}^\mu u)(0)}{\alpha!} = \sum_{\alpha_0 + \cdots+ \alpha_\mu = \alpha} \frac{\partial^{\alpha_0} u(0)}{\alpha_0!} \prod_{j=1}^\mu \frac{\partial^{\alpha_j} g_{0} (0)}{\alpha_j !}.
		\end{equation}
		Noting that $g_{0}(x) = \OO(|x|^3)$, letting
		\begin{equation}
			\label{restriction alpha_j}
			A_{\alpha,\mu} := \{(\alpha_0,\cdots,\alpha_\mu) : \alpha_0+\cdots+\alpha_\mu=\alpha,\ |\alpha_1|,\cdots,|\alpha_\mu|\geq 3\},
		\end{equation}
		we can therefore rewrite \eqref{Leibniz 2} as
		\begin{equation}
			\label{Leibniz 3}
			\frac{\partial^\alpha(g_{0}^\mu u)(0)}{\alpha!} = \sum_{A_{\alpha,\mu}} \frac{\partial^{\alpha_0} u(0)}{\alpha_0!} \prod_{j=1}^\mu \frac{\partial^{\alpha_j} g_{0} (0)}{\alpha_j !}.
		\end{equation}
		Recalling that $f\in\GG_{\rm b}^s(U)$, we obtain that, in view of \eqref{eqn:g_x_0 in Section 4}, there exist $A_f, C_f > 0$ such that 
		\begin{equation}
			\label{f(x_0) Gevrey estimates}
			|\partial^\beta g_0(0)| = |\partial^\beta f(0)| \leq A_f C_f^{|\beta|} \beta!^s,\quad |\beta|\geq 3.
		\end{equation}
		Similarly, since $u\in\GG_{\rm b}^s(U)$, there exist $A_u, C_u > 0$ such that
		\begin{equation}
			\label{u(x_0) Gevrey estimates}
			|\partial^\beta u (0)| \leq A_u C_u^{|\beta|} \beta!^s ,\quad \beta\in\NN^d .
		\end{equation}
		Combining \eqref{f(x_0) Gevrey estimates} and \eqref{u(x_0) Gevrey estimates}, we get from \eqref{Leibniz 3},
		\begin{equation}
			\label{Leibniz estimate}
			\frac{|\partial^\alpha(g_{0}^\mu u)(0)|}{\alpha!} \leq A_u A_f^\mu C_{u,f}^{|\alpha|}\sum_{A_{\alpha,\mu}} \bigl(\prod_{j=0}^\mu |\alpha_j|! \bigr)^{s-1} ,
		\end{equation}
		where $C_{u,f} := \max(C_u, C_f)$. It follows from \eqref{restriction alpha_j} and Corollary \ref{cor:factorial product max} that
		\[
		|\alpha_1|!\cdots|\alpha_\mu|! \leq 3!^{\mu-1} \left( |\alpha| - |\alpha_0| - 3(\mu-1)\right)! \leq 6^{\mu-1} 3^{|\alpha|-3\mu + 2} (|\alpha|-|\alpha_0|-3\mu)! .
		\]
		Here we have also used inequalities $(|\alpha|-3\mu + 1)(|\alpha|-3\mu + 3)<(|\alpha|-3\mu + 2)^2$ and $n^3\leq 3^n$, $n\in\NN$. We note by Proposition \ref{prop:partition multi-index counting} that $\# A_{\alpha,\mu} \leq \# P(\alpha,\mu+1)\leq 2^{|\alpha|+d\mu}$. We then conclude from \eqref{Leibniz estimate} and the above estimates that
		\begin{equation}
			\label{est:Taylor coeffs}
			|\partial^\alpha(g_{0}^\mu u)(0)|/\alpha! \leq A_u (2^d A_f)^\mu (2\times 3^{s-1}C_{u,f})^{|\alpha|} (|\alpha|-3\mu)!^{s-1} .
		\end{equation}
		Noting that $\#\{ \alpha\in\NN^d : |\alpha|=m \} \leq 2^{m + d -1}$ by \eqref{partition of integer}, we recall 
		\[
		T_{2k+2\mu}(g_{0}^\mu u)(x) = \sum_{m=3\mu}^{2k+2\mu-1} \sum_{|\alpha|=m} \frac{\partial^\alpha(g_{0}^\mu u)(0)}{\alpha!} x^\alpha,
		\]
		in view of \eqref{est:Taylor coeffs}, to obtain the following estimate, with $C = 4\times 3^{s-1}C_{u,f}$, 
		\begin{equation}
			\label{est:Taylor 2k+2mu th}
			\left|T_{2k+2\mu}(g_{0}^\mu u)(x)\right| \leq 2^{d-1} A_u (2^d A_f)^\mu \sum_{m=3\mu}^{2k+2\mu-1} C^m (m-3\mu)!^{s-1} |x|^m .
		\end{equation}
		Applying \eqref{basic inequality 1} with \eqref{Im f_sigma} we have
		\begin{equation}
			\label{est:(x-x_0)alpha}
			\left|e^{\frac{i}{2h} f_0(x)} |x|^{m} \right| \leq C_0^{-m/2} m!^{1/2}  h^{m/2} .
		\end{equation}
		Combing \eqref{est:Taylor 2k+2mu th} and \eqref{est:(x-x_0)alpha}, we obtain, with $\widetilde{A}_u = 2^d A_u$, $\widetilde{A}_f = 2^d A_f$,
		\[
		\frac{h^{-\mu}}{\mu!}\left|e^{\frac{i}{2h} f_0(x)}T_{2k+2\mu}(g_{0}^\mu u)(x)\right| \leq \widetilde{A}_u \widetilde{A}_f^\mu\sum_{m=3\mu}^{2k+2\mu-1} \frac{C^m}{C_0^{m/2}} h^{\frac{m}{2}-\mu} (m-3\mu)!^{s-1} \frac{m!^{1/2}}{\mu!}.
		\]
		Using the inequality $\displaystyle \frac{m!}{\mu!\mu!(m-2\mu)!} \leq 3^m$, and recalling $\mu<2k$, we have
		\begin{equation}
			\label{est:T 2k+2mu part 1}
			\begin{split}
				\frac{h^{-\mu}}{\mu!}\left|e^{\frac{i}{2h} f_0(x)}T_{2k+2\mu}(g_{0}^\mu u)(x)\right| &\leq \widetilde{A}_u \widetilde{A}_f^\mu \sum_{m=3\mu}^{2k+2\mu-1} \bigl(\frac{3C^2}{C_0}\bigr)^{\frac{m}{2}} h^{\frac{m}{2}-\mu} (m-2\mu)!^{s-\frac{1}{2}} \\
				&\leq (27 A^2 C^6 C_0^{-3} )^{k} \sum_{\ell =\mu}^{2k-1} \left(h^{\ell} \ell!^{2s-1}\right)^{1/2}.
			\end{split}
		\end{equation}
		Here we write $A=\max(\widetilde{A}_u,\widetilde{A}_f)$. It follows from \eqref{seq a_k decrease increase} that, for each $\mu\leq \ell <2k$,
		\begin{equation}
			\label{est:T 2k+2mu part 2}
			h^{\ell} \ell!^{2s-1}\leq \max(1,h^{2k}(2k)!^{2s-1}).
		\end{equation}
		In view of \eqref{est:T 2k+2mu part 1} and \eqref{est:T 2k+2mu part 2}, noting that $(2k)!\leq 2^{2k} k!^2$, we obtain
		\begin{equation}
			\label{est:T 2k+2mu final}
			\frac{h^{-\mu}}{\mu!}\left|e^{\frac{i}{2h} f_0(x)}T_{2k+2\mu}(g_{0}^\mu u)(x)\right| \leq 2k(27 A^2 C^6 C_0^{-3} )^{k} \max(1, 2^{k(2s-1)} k!^{2s-1} h^k).
		\end{equation}
		By \eqref{basic inequality 2}, we deduce from \eqref{est:1-rho and exp terms} that
		\begin{equation}
			\label{est:1-rho exp final}
			\left|e^{\frac{i}{3h} f_0(x)}(1-\rho(x)) \right|\leq \min(1,C_0'^k k!^{2s-1} h^k),\quad C_0' = \left((6s-3)C_0^{-1}\right)^{2s-1}.
		\end{equation}
		Furthermore, let us note that by \eqref{Im f_sigma},
		\[
		\int_{\RR^d} \left|e^{\frac{i}{6h} f_0(x)}\right| dx \leq \int_{\RR^d} e^{-\frac{C_0 |x|^2}{6h}} dx \leq C_1 h^{\frac{d}{2}} .
		\]
		Decomposing $ e^{\frac{i}{h}f_0(x)} = e^{\frac{i}{2h}f_0(x)} e^{\frac{i}{3h}f_0(x)} e^{\frac{i}{6h}f_0(x)}$, we combine the inequality above with \eqref{est:T 2k+2mu final} and \eqref{est:1-rho exp final} to conclude that, for each $0\leq \mu<2k$,
		\begin{equation}
			\label{est:drop cutoff}
			\left|\frac{i^{\mu} h^{-\mu}}{\mu!} \int e^{\frac{i}{h}f_0(x)}(1-\rho(x)) T_{2k+2\mu} (g_{0}^\mu u)(x)\,dx\right| \leq C^{k+1} k!^{2s-1} h^{k+\frac{d}{2}} ,
		\end{equation}
		where we renamed $ 54 A^2 C^6 C_0^{-3} C_1 \max\left(2^{2s-1},C_0'\right) $ to be $C$.
		
		$\bullet$ \emph{Integrate Taylor polynomials with quadratic phase: computation.}
		In view of \eqref{est:restrict to Taylor polynomials} and \eqref{est:drop cutoff}, we have shown that, for each $0\leq\mu<2k$,
		\begin{equation}
			\label{Taylor without cutoff}
			\biggl| \frac{I^{(\mu)}(0)}{\mu!} - \frac{i^{\mu} h^{-\mu}}{\mu!}\int_{\RR^d} e^{\frac{i}{h}f_0(x)} T_{2k+2\mu} (g_{0}^\mu u)(x) dx \biggr| \leq C^{k+1} (k!)^{2s-1} h^{k+\frac{d}{2}}.
		\end{equation}
		Noting that $f_0(x) = \frac{1}{2}f''(0)x\cdot x$ is quadratic, and that $T_{2k+2\mu}(g_{0}^\mu u)$ is a polynomial in $x$, the integral in \eqref{Taylor without cutoff} can be computed explicitly. To illustrate that, we shall give a brief proof of the following result, see also \cite[Lemma 2.2]{Sj82}:
		\begin{prop}
			\label{prop:exact integral}
			Let $A$ be a symmetric non-degenerate matrix with $\Im A$ positive definite. Let $ P_m (x) = \sum_{|\alpha|=m} c_\alpha x^\alpha$ be a homogeneous polynomial of degree $m$. Then 
			\[
			\int_{\RR^d} e^{\frac{i Ax\cdot x}{2}} P_m(x)\,dx = \frac{(\det( A/2\pi i))^{-\frac{1}{2}}}{2^\ell \ell!} (i A^{-1}\partial\cdot \partial)^{\ell} P_m ,\quad m=2\ell,\ \ell\in\NN .
			\]
			We remark that the above integral equals $0$ if $m\in\NN$ is odd.
		\end{prop}
		
		\begin{proof}
			Let us recall the Euler's homogeneous equation: $P_m = \frac{1}{m}\sum x_j \partial_{x_j} P_m$, then
			\[
			\int_{\RR^d} e^{\frac{i Ax\cdot x}{2}} P_m\,dx = \frac{1}{m}\int_{\RR^d} e^{\frac{i Ax\cdot x}{2}} \sum x_j \partial_{x_j} P_m \,dx= \frac{1}{m} \int_{\RR^d} \sum  -\partial_{x_j}\bigl(x_j e^{\frac{i Ax\cdot x}{2}}\bigr) P_m \,dx.
			\]
			We remark that the positive definiteness of $\Im A$ justifies all the integrations by parts throughout the proof. By a direct calculation, we get
			\[
			\sum  -\partial_{x_j}\left(x_j e^{\frac{i Ax\cdot x}{2}}\right) = (i A^{-1}\partial\cdot \partial) \left(e^{\frac{i Ax\cdot x}{2}}\right).
			\]
			Integrating by parts twice, we obtain that
			\[
			\int_{\RR^d} e^{\frac{i Ax\cdot x}{2}} P_m\,dx = \frac{1}{m}\int_{\RR^d} e^{\frac{i Ax\cdot x}{2}} (i A^{-1}\partial\cdot \partial) P_m \,dx.
			\] 
			Noting that $(i A^{-1}\partial\cdot \partial) P_m$ is a homogeneous polynomial of degree $m-2$, we can therefore iterate and conclude that
			\[
			\int_{\RR^d} e^{\frac{i Ax\cdot x}{2}} P_m(x)\,dx = \frac{(i A^{-1}\partial\cdot \partial)^{\ell} P_m}{m(m-2)\cdots 2} \int_{\RR^d} e^{\frac{i Ax\cdot x}{2}} dx,\quad m=2\ell,
			\]
			and that the integral equals $0$ when $m$ is odd since $(i A^{-1}\partial\cdot \partial) P_1 = 0$. The desired result follows by evaluating $\displaystyle \int_{\RR^d} e^{\frac{i Ax\cdot x}{2}} dx = (\det( A/2\pi i))^{-\frac{1}{2}}$.
		\end{proof}
		
		Let us return to compute the integral in \eqref{Taylor without cutoff}. Writing
		\[
		T_{2k+2\mu}(g_0^\mu u)(x) = \sum_{m=3\mu}^{2k+2\mu-1} \sum_{|\alpha|=m} \frac{\partial^\alpha (g_0^\mu u)(0)}{\alpha!} x^\alpha,
		\]
		using Proposition \ref{prop:exact integral} we have
		\[
		\begin{split}
			{ }&\quad\int_{\RR^d} e^{\frac{i}{h}f_0(x)} T_{2k+2\mu}(g_0^\mu u)(x) dx \\
			&= \left(\det\left(\frac{f''(0)}{2\pi i h}\right)\right)^{-\frac{1}{2}}\sum_{3\mu\leq 2\nu < 2k+2\mu} \frac{(ih f''(0)^{-1}\partial\cdot\partial)^\nu}{(2\nu)!!}  \sum_{|\alpha|=2\nu} \frac{\partial^\alpha (g_0^\mu u)(0)}{\alpha!} x^\alpha \\
			&= \left(\det\left(\frac{f''(0)}{2\pi i h}\right)\right)^{-\frac{1}{2}}\sum_{3\mu\leq 2\nu < 2k+2\mu} \frac{i^\nu h^\nu}{2^\nu \nu!} (f''(0)^{-1}\partial\cdot\partial)^\nu (g_0^\mu u)(0) .
		\end{split}
		\]
		It follows that, with operators $L_j$ given in \eqref{eqn:Lj}, 
		\[
		\sum_{\mu=0}^{2k-1}\frac{i^{\mu} h^{-\mu}}{\mu!}\int_{\RR^d} e^{\frac{i}{h}f_0(x)} T_{2k+2\mu} (g_{0}^\mu u)(x) dx = \det\left(\left(\frac{f''(0)}{2\pi i h}\right)\right)^{-\frac{1}{2}}\sum_{j<k} h^j L_j u,
		\]
		this together with \eqref{est:fix quadratic phase} and \eqref{Taylor without cutoff} implies that
		\[
		\biggl|\int e^{\frac{i}{h}f(x)} \rho(x) u(x) dx  - \det\left(\left(\frac{f''(0)}{2\pi i h}\right)\right)^{-\frac{1}{2}}\sum_{j<k} h^j L_j u\biggr| \leq C^{k+1} k!^{2s-1} h^{k+\frac{d}{2}} .
		\]
		In view of \eqref{eqn:outside cutoff} and \eqref{basic inequality 2}, we have
		\[
		\left|\int_U e^{\frac{i}{h}f(x)}u(x) dx - \int e^{\frac{i}{h}f(x)} \rho(x) u(x) dx\right| \leq C h^{\frac{d}{2}} e^{-\frac{C_0}{2}h^{-\frac{1}{2s-1}}} \leq C_1^{k+1} k!^{2s-1} h^{k+\frac{d}{2}},
		\]
		with $C_1 = C (C_0^{-1}(4s-2))^{(2s-1)}$, this completes the proof of \eqref{Gevrey stationary phase} under the technical condition $k<h^{-\frac{1}{2s-1}}$.
		
		$\bullet$ \emph{Estimates of operators $L_j$.} In order to remove the condition $k<h^{-\frac{1}{2s-1}}$, we need to derive estimates on $L_j u$, for $j\in \mathbb{N}$. To this end, we deduce from \eqref{est:Taylor coeffs} that for $|\alpha|=2\nu = 2j+2\mu$,
		\begin{equation}
			\label{estimate L_j u prep}
			|\partial^\alpha (g_0^\mu u) (0)| \leq A_u \big(2^{2s+2} 3^{2s} C_{u,f}^2\big)^j j!^{2s} \big(2^{d+4} 3^{2s} A_f C_{u,f}^2 \big)^\mu \mu!^{3-s} .
		\end{equation}
		Here we have used the inequalities $\alpha!\leq (2j+2\mu)! \leq 4^{2j+2\mu} j!^2 \mu!^2$ and $(2j-\mu)!\leq (2j)!/\mu!\leq 2^{2j} j!^2/\mu!$. Let us now recall the formula \eqref{eqn:Lj} for $L_j u$. Since $\Im f''(0)$ is positive definite, we get 
		$\| f''(0)^{-1} \|_{\rm max} = c_0 > 0$, where $\|A\|_{\rm max} = \max_{i j}a_{ij}$ is the max norm of a matrix $A$. It then follows from \eqref{eqn:Lj} and \eqref{estimate L_j u prep} that
		\[
		\begin{split}
			|L_j u| &\leq \sum_{\mu=0}^{2j} \frac{1}{2^{j+\mu} j! \mu!^2} (c_0 d^2)^{j+\mu} \max_{|\alpha|=2j+2\mu} \left|\partial^\alpha (g_0^\mu u) (0)\right| \\
			& \leq A_u \big(2^{2s+1} 3^{2s} d^2 c_0 C_{u,f}^2\big)^j j!^{2s-1} \sum_{\mu=0}^{2j} \big(3^{2s} 2^{d+3} d^2 c_0 A_f C_{u,f}^2 \big)^\mu \mu!^{1-s} \\
			&\leq 2 A_u \big(2^{2s+2d+7} 3^{6s} d^6 c_0^3 A_f^2 C_{u,f}^6 \big)^j j!^{2s-1} . 
		\end{split}
		\]
		Here we have used $s>1$ and the inequality $\sum_{\mu=0}^{2j} R^\mu \leq 2 R^{2j}$ as $R\gg 1$ in the last step. We summarize therefore the estimates for operators $L_j$, $j\in\NN$, as follows
		\begin{equation}
			\label{estimate L_j u}
			|L_j u| \leq A C^j j!^{2s-1},\quad A = 2 A_u ,\ \; C=2^{2s+2d+7} 3^{6s} d^6 c_0^3 A_f^2 C_{u,f}^6,
		\end{equation}
		where $A_f, C_f$ and $A_u, C_u$ are as in \eqref{f(x_0) Gevrey estimates} and \eqref{u(x_0) Gevrey estimates}, $C_{u,f}=\max(C_u, C_f)$.
		
		We have proved \eqref{Gevrey stationary phase} under the condition $k<h^{-\frac{1}{2s-1}}$:
		\begin{equation}
			\label{Gevrey stationary phase with condition}
			\biggl|\frac{1}{h^{d/2}}\int_U e^{\frac{i}{h} f(x)} u(x)\,dx - c_{f,0}\sum_{j=0}^{k-1} h^j L_j u\biggr| \leq C^{k+1} k!^{2s-1} h^k,\quad k < h^{-\frac{1}{2s-1}},
		\end{equation}
		where $c_{f,0}=\det\big(\frac{f''(0)}{2\pi i}\big)^{-1/2}$ and $C>0$ is some large constant. It remains to prove \eqref{Gevrey stationary phase} for $k\geq h^{-\frac{1}{2s-1}}$. Let us set $k_0 = [ (Ch)^{-\frac{1}{2s-1}} ]$, then $k_0 < h^{-\frac{1}{2s-1}}$. For any $k>k_0$, in view of \eqref{Gevrey stationary phase with condition} and \eqref{estimate L_j u} we have
		\[
		\begin{split}
			{ }&\quad \biggl|\frac{1}{h^{d/2}}\int_U e^{\frac{i}{h} f(x)} u(x)\,dx - c_{f,0}\sum_{j=0}^{k-1} h^j L_j u\biggr| \\
			&\leq \biggl|\frac{1}{h^{d/2}}\int_U e^{\frac{i}{h} f(x)} u(x)\,dx - c_{f,0}\sum_{j=0}^{k_0 -1} h^j L_j u\biggr| + |c_{f,0}|\sum_{j=k_0}^{k-1} h^j |L_j u| \\
			&\leq C^{k_0 + 1}k_0 !^{2s-1} h^{k_0} + \widetilde{A} \sum_{j=k_0}^{k-1} C^j j!^{2s-1} h^j .
		\end{split}
		\]
		In view of \eqref{seq a_k decrease increase}, we have $C^j j!^{2s-1} h^j \leq C^k k!^{2s-1} h^k$, for $k_0\leq j\leq k$, therefore,
		\begin{equation*}
			\biggl|\frac{1}{h^{d/2}}\int_U e^{\frac{i}{h} f(x)} u(x)\,dx - c_{f,0}\sum_{j=0}^{k-1} h^j L_j u\biggr| \leq (C+(k-k_0)\widetilde{A}) C^k k!^{2s-1} h^k \leq \widetilde{C}^{k+1} k!^{2s-1} h^k ,
		\end{equation*}
		for some constant $\widetilde{C}>C$. This completes the proof of \eqref{Gevrey stationary phase} for all $k\in\NN$.
	\end{proof}
	
	\section{The amplitude of approximate Bergman projection}\label{Sec construction of the amplitude}
	Let $\Omega\subset\CC^n$ be a pseudoconvex domain, and let $\Phi\in\CI(\Omega;\RR)$ be strictly plurisubharmonic in $\Omega$, namely
	\begin{equation}
		\label{Phi strict plurisubharmonic}
		\sum_{j, k=1}^{n} \frac{\partial^2 \Phi(x)}{\partial x_j \partial\xbar_k} \xi_j \overline{\xi}_k \geq c(x) |\xi|^2,\quad x\in\Omega,\ \; \xi\in\CC^n ,
	\end{equation}
	where $0<c\in C(\Omega)$.
	
	In this section, we aim to obtain a Gevrey symbol $a\in\GG^s$, holomorphic to $\infty$--order along the anti-diagonal, realizing a formal $\GG^{s,2s-1}$ symbol (see Definition \ref{defi:formal Gevrey symbol}), as a solution of \eqref{Aa=1 in main theorem}. To achieve this goal, the key analysis is a quantitative estimate of the growth rate (in $j$) of the Bergman coefficients $a_j$, see Section \ref{subsection:Bergman coefficients estimates}.
	
	\subsection{Two recursive formulas on the Bergman coefficients}
	Let us first briefly review the main result obtained in \cite{hitrik2022smooth}. For $x_0\in \Omega$, we consider the almost holomorphic function $\Psi\in \CI(\neigh((x_0,\overline{x_0}),\CC^{2n}))$ along the anti-diagonal $\{(x,\xbar): x\in\CC^n\}$, which satisfies
	\begin{equation*}
		\label{Psi as polarization}
		\Psi(x,\xbar) = \Phi(x),\ \; x\in\neigh(x_0,\CC^n);
	\end{equation*}
	\begin{equation}
		\label{Psi almost holomorphic}
		\partial_{\overline{x}}\Psi(x,y) = \OO(|y-\overline{x}|^\infty),\quad \partial_{\overline{y}}\Psi(x,y) = \OO(|y-\overline{x}|^\infty).
	\end{equation} 
	We shall recall from \cite[Theorem 1.1]{hitrik2022smooth} that there exists a classical elliptic symbol $a(x,y;h)\in S_{\textrm{cl}}^0(\textrm{neigh}((x_0,\overline{x_0}),\CC^{2n}))$ which allows an asymptotic expansion
	\begin{equation}
		\label{Bergman asymptotic}
		a(x,y;h) \sim \sum_{j=0}^\infty h^j a_j(x,y),\quad a_j\in \CI(\textrm{neigh}((x_0,\overline{x_0}),\CC^{2n})),
	\end{equation}
	with $a_j$ almost holomorphic along the anti-diagonal $\{ y=\overline{x}\}$ in the sense that
	\begin{equation}
		\label{a_j almost holomorphic}
		\partial_{\overline{x}}a_j (x,y) = \OO(|y-\overline{x}|^\infty),\quad \partial_{\overline{y}}a_j (x,y) = \OO(|y-\overline{x}|^\infty),\quad j = 0,1,2,\ldots,
	\end{equation}
	such that 
	\begin{equation}
		\label{eqn:Aa=1}
		(Aa)(y,\overline{y};h) = 1 + \OO(h^\infty),\quad y\in\neigh(x_0,\CC^n),
	\end{equation}
	where $A$ is an elliptic Fourier integral operator. Moreover, there exist small open neighborhoods $U\Subset V\Subset \Omega$ of $x_0$, with smooth boundaries, such that the operator
	\begin{equation}
		\label{Bergman projection}
		\widetilde{\Pi}_V u(x) := \frac{1}{h^n}\int_V e^{\frac{2}{h}(\Psi(x,\overline{y})-\Phi(y))} a(x,\overline{y};h) u(y)\,L(dy),\quad u\in H_\Phi(V),
	\end{equation}
	satisfies 
	\begin{equation}
		\label{reproducing property 1}
		\widetilde{\Pi}_V - 1 = \OO(h^\infty) : H_\Phi(V) \to L_\Phi^2(U).
	\end{equation}
	Let us now recall from \cite[Section 2]{hitrik2022smooth} more details about the Fourier integral operator $A$, thus a more explicit characterization of the amplitude $a(x,y;h)$. For that we follow \cite{deleporte2022analytic} to introduce $\phi\in\CI(\neigh((x_0,\overline{x_0};x_0,\overline{x_0}),\CC^{4n}))$ by
	\begin{equation}
		\label{phase phi defn}
		\phi(y,\widetilde{x};x,\widetilde{y}) := \Psi(x,\widetilde{y}) - \Psi(x,\widetilde{x}) - \Psi(y,\widetilde{y}) + \Psi(y,\widetilde{x}).
	\end{equation}
	We say that $\Gamma(y,\widetilde{x})\subset \CC_{x}^n \times \CC_{\widetilde{y}}^n$ is a good contour with respect to the phase function $(x,\widetilde{y}) \mapsto \phi(y,\widetilde{x};x,\widetilde{y})$, if it is a smooth $2n$--(real) dimensional contour of integration passing through the point $(x,\widetilde{y}) = (y,\widetilde{x})$, depending smoothly on the parameters $(y,\widetilde{x})\in\neigh((x_0,\overline{x_0}),\CC^{2n})$, such that
	\begin{equation}
		\label{good contour C 2n}
		\Re \phi(y,\widetilde{x};x,\widetilde{y}) \leq -C^{-1} \dist((x,\widetilde{y}),(y,\widetilde{x}))^2,\quad (x,\widetilde{y})\in\Gamma(y,\widetilde{x}),
	\end{equation}
	where $C>0$ is uniform for $(y,\widetilde{x})\in\neigh((x_0,\overline{x_0}),\CC^{2n})$. For future reference, we remark that, following \cite{hitrik2022smooth}, the affine contour
	\begin{equation}
		\label{good contour: affine}
		\Gamma_0(y,\widetilde{x}): \neigh(0,\CC^n)\owns z \mapsto (y+z,\widetilde{x}-\zbar) \in \CC_{x}^n \times \CC_{\widetilde{y}}^n,
	\end{equation}
	is good, provided $(y,\widetilde{x})\in\neigh((x_0,\overline{x_0}),\CC^{2n})$ for some small neighborhood. For a good contour $\Gamma(y,\widetilde{x})$ with respect to $(x,\widetilde{y}) \mapsto \phi(y,\widetilde{x};x,\widetilde{y})$ and an amplitude $a(x,\widetilde{y};h)$ satisfying \eqref{Bergman asymptotic} and \eqref{a_j almost holomorphic}, we define
	\begin{equation}
		\label{A_Gamma FIO}
		(A_\Gamma a)(y,\widetilde{x};h) = \frac{1}{(2ih)^n} \int_{\Gamma(y,\widetilde{x})} e^{\frac{2}{h} \phi(y,\widetilde{x};x,\widetilde{y})} a(x,\widetilde{y};h) \,dx d\widetilde{y} .
	\end{equation}
	It has been proved in \cite[Proposition 2.3]{hitrik2022smooth} that the restriction of $A_\Gamma a$ to the anti-diagonal is independent of the choice of a good contour, up to an $\OO(h^\infty)$--error:
	\begin{prop}
		\label{prop:independent of contour}
		There exists an open neighborhood $V_0\Subset \Omega$ of $x_0$ such that for each $y\in V_0$, any two good contours $\Gamma_j(y,\ybar)$, $j=1,2$ with respect to the phase $(x,\widetilde{y})\mapsto \phi(y,\ybar;x,\widetilde{y})$, and any amplitude $a(x,\widetilde{y};h)$ satisfying \eqref{Bergman asymptotic} and \eqref{a_j almost holomorphic}, we have
		\[
		(A_{\Gamma_1} a)(y,\ybar;h) - (A_{\Gamma_2} a)(y,\ybar;h) = \OO(h^\infty),\quad y\in V_0 .
		\] 
	\end{prop}
	\noindent
	We can therefore interpret \eqref{eqn:Aa=1} by, given that $\Gamma(y,\ybar)$ is good,
	\[
	(A_\Gamma a)(y,\ybar;h) = 1 + \OO(h^\infty),\quad y\in\neigh(x_0,\CC^n).
	\] 
	In view of Proposition \ref{prop:independent of contour}, one shall choose a particular good contour for convenience, following \cite[Section 2]{hitrik2022smooth}, we work with the good contour $\Gamma_0(y,\ybar)$ given in \eqref{good contour: affine}. Using the parametrization of $\Gamma_0(y,\ybar)$: $U\owns z\mapsto (y+z,\ybar-\zbar)$, for some $U=\neigh(0,\CC^n)$, it follows from \eqref{A_Gamma FIO} that the amplitude $a$ satisfies
	\begin{equation}
		\label{recursive HS integral form}
		(A_{\Gamma_0} a)(y,\ybar;h) = \frac{1}{h^n}\int_{U} e^{\frac{i}{h} f(y,z)} a(y+z,\ybar-\zbar;h)\,L(dz) = 1 + \OO(h^\infty),
	\end{equation}    
	\begin{equation}
		\label{eqn:f(y,z)}
		\text{with}\quad f(y,z) := -2i \phi(y,\ybar;y+z,\ybar-\zbar).
	\end{equation}
	Here $L(dz)$ denotes the Lebesgue measure on $\CC^n$, and we have used the identity $dz\,d\zbar = (2/i)^n L(dz)$ (we require that $(dx_1,\ldots,dx_n,dy_1,\ldots,dy_n)$ is positively oriented). It has been shown in \cite[Section 2]{hitrik2022smooth}, by Taylor expansions, that
	\begin{equation*}
		f(y,z) = 2i\Phi_{x\overline{x}}''(y)\zbar\cdot z + g(y,z),\quad g(y,z) = \OO(|z|^3).
	\end{equation*}
	We can then apply the complex stationary phase method in the form given in \cite[Theorem 7.7.5]{hormander1985analysis} to obtain a complete asymptotic expansion of the integral in \eqref{recursive HS integral form} as $h\to 0^+$, with the help of \eqref{Bergman asymptotic}, we conclude
	\[
	(A_{\Gamma_0} a)(y,\ybar;h) \sim \sum_{\ell=0}^\infty h^\ell \sum_{j+k = \ell}(L_{k,y} a_j)(y,\ybar),
	\]
	with the operators $L_{k,y}$, $k\in\NN$, given explicitly by
	\begin{equation}
		\label{eqn:L_k,y operators}
		\begin{split}
			(L_{k,y} a_j)(y,\ybar) = &\frac{\pi^n}{2^n \det(\Phi_{x\overline{x}}''(y))}\sum_{\nu-\mu=k}\sum_{2\nu\geq 3\mu}\frac{i^{\mu}}{2^\nu \mu! \nu!} \\ &\qquad\left(\big(\Phi_{x\overline{x}}''(y)\big)^{-1}\partial_z\cdot\partial_{\zbar}\right)^\nu \big(g(y,z)^\mu a_j(y+z,\ybar-\zbar)\big)\bigg\lvert_{z=0}
		\end{split}
	\end{equation}
	This with \eqref{recursive HS integral form} implies that the Bergman coefficients $a_j$ given in \eqref{Bergman asymptotic} satisfy the following recursive formula,
	\begin{equation}
		\label{recursive formula HS}
		\begin{gathered}
			a_m(y,\ybar) = -\sum_{j=1}^m \sum_{\substack{\nu-\mu=j \\ 2\nu\geq 3\mu}}\frac{i^{\mu}}{2^\nu \mu! \nu!}\left(\big(\Phi_{x\overline{x}}''(y)\big)^{-1}\partial_z\cdot\partial_{\zbar}\right)^\nu \big(g(y,z)^\mu a_{m-j}(y+z,\ybar-\zbar)\big)\bigg\lvert_{z=0}\\
			a_0(y,\ybar) = \frac{2^n \det(\Phi_{x\overline{x}}''(y))}{\pi^n}.
		\end{gathered}
	\end{equation}
	We remark that the terms in the sum appearing in \eqref{eqn:L_k,y operators} or \eqref{recursive formula HS} depend only on the restrictions of the $a_j$'s to the anti-diagonal. In fact, we have by \eqref{a_j almost holomorphic},
	\[
	\partial_z^\alpha \partial_{\zbar}^\beta \big(a_j(y+z,\ybar-\zbar)\big)\big\lvert_{z=0} = (-1)^{|\beta|} \partial_y^\alpha \partial_{\ybar}^\beta a_j(y,\ybar),\quad \alpha, \beta \in \NN^n .
	\]
	
	In the next subsection, we will prove the Gevrey-$s$ estimates for the Bergman coefficients $a_m(x,\xbar)$ given in \eqref{Bergman asymptotic}, for $x\in U\Subset\Omega$ a small open ball centered at $x_0\in\Omega$. We will also estimate the growth rate of $\|a_m(x,\overline{x})\|_{s,R,U}$ (with some constant $R>0$) as $m\to\infty$. For that we shall first derive another recursive formula on the Bergman coefficients $a_m$, which is essentially implied by \cite[Lemma 9 and Equation (10)]{charles2003berezin}. This recursive formula turns out to be more convenient than that in \eqref{recursive formula HS} when one estimates the growth of $a_m$ with respect to $m$. We will include a self-contained proof for readers' convenience. We remark that in the case of real analytic K{\"a}hler potentials $\Phi$ it was proved in \cite[Section 3]{hezari2021property}.
	
	Let $V\Subset \Omega$ be a small pseudoconvex neighborhood of $x_0$ as in \eqref{Bergman projection}, as we may even choose it to be a ball centered at $x_0$. For any $x,y\in V$, we introduce
	\begin{equation}
		\label{eqn:k_x(y)}
		k_x(y) := \exp\Bigl(\frac{2}{h} \Psi(y,\xbar) - \frac{1}{h} \Phi(x) \Bigr).
	\end{equation}
	Let us compute the following Taylor expansions:
	\begin{equation}
		\label{Phi Taylor expansion}
		\begin{split}
			\Phi(y) =\ &\Phi(x) + 2\Re\Bigl( \frac{\partial\Phi}{\partial x}(x)\cdot(y-x) + \frac{1}{2}\Phi_{xx}''(x) (y-x)\cdot (y-x) \Bigr) \\
			& + \Phi_{x\overline{x}}''(x)\overline{(y-x)}\cdot (y-x) + \OO(|y-x|^3), 
		\end{split}
	\end{equation}
	\begin{equation}
		\label{Psi Taylor expansion}
		\Psi(x,\overline{y}) = \Phi(x) + \frac{\partial\Phi}{\partial \overline{x}}(x)\cdot\overline{(y-x)} + \frac{1}{2}\Phi_{\overline{x}\overline{x}}''(x) \overline{(y-x)}\cdot \overline{(y-x)} + \OO(|y-x|^3),
	\end{equation}
	\begin{equation}
		\label{Psi Taylor expansion y,xbar}
		\Psi(y,\xbar) = \Phi(x) + \frac{\partial\Phi}{\partial x}(x)\cdot(y-x) + \frac{1}{2}\Phi_{x x}''(x) (y-x)\cdot (y-x) + \OO(|y-x|^3),
	\end{equation}
	where we used \eqref{Psi almost holomorphic} when computing the Taylor expansions for $\Psi(x,\overline{y})$ and $\Psi(y,\xbar)$. In view of \eqref{Phi Taylor expansion} and \eqref{Psi Taylor expansion}, we can assume $V$ to be small enough so that
	\begin{equation}
		\label{Psi basic estimate}
		\Phi(x) + \Phi(y) - 2\Re\Psi(x,\ybar) \asymp |x-y|^2,\quad x, y \in V .
	\end{equation}
	It then follows from \eqref{eqn:k_x(y)}, \eqref{Psi basic estimate} (switching $x$ and $y$) and the strict plurisubharmonicity of $\Phi$ that there exists $c_0=c_0(V)>0$ such that
	\begin{equation}
		\label{k_x pointwise bound}
		|k_x(y)|e^{-\frac{1}{h}\Phi(y)} \leq e^{-\frac{c_0}{h} |y-x|^2}.
	\end{equation} 
	As a consequence, we can deduce that
	\begin{equation}
		\label{k_x H_Phi norm}
		k_x\in L_\Phi^2 (V),\quad \|k_x\|_{L_\Phi^2 (V)} \leq c_n c_0^{-n/2} h^{n/2},\quad c_n^2 = \int_{\CC^n} e^{-2|z|^2} L(dz).
	\end{equation}
	We note that $k_x$ may not be in $H_\Phi(V)$ as we only know $k_x$ is almost holomorphic in $y\in V$, in view of \eqref{eqn:k_x(y)} and \eqref{Psi almost holomorphic}. Let us now estimate $h\partial_{\ybar} k_x(y)$. For that we compute
	\[
	h\partial_{\ybar} k_x(y) = 2\partial_{\ybar}\Psi(y,\xbar) k_x(y).
	\] 
	This, together with \eqref{Psi almost holomorphic} and \eqref{k_x pointwise bound} implies that
	\begin{equation}
		\label{dbar k_x pointwise}
		|h\partial_{\ybar} k_x(y)|e^{-\frac{1}{h}\Phi(y)} \leq \OO(|y-x|^\infty) e^{-\frac{c_0}{h} |y-x|^2} .
	\end{equation}
	By \eqref{basic inequality 1} we get, for each $N\in\NN$,
	\[
	|y-x|^N e^{-\frac{c_0}{2h} |y-x|^2} \leq c_0^{-N/2} N!^{1/2} h^{N/2},
	\]
	we can therefore deduce from \eqref{dbar k_x pointwise}
	\begin{equation} 
		\label{dbar k_x estimates}
		\left\| e^{-\frac{1}{h}\Phi(y)} h\partial_{\ybar} k_x(y) \right\|_{L^\infty(V)} = \OO(h^\infty);\quad \|h\partial_{\ybar} k_x\|_{L_\Phi^2 (V)} = \OO(h^\infty).
	\end{equation}
	We claim that $k_x\in L_\Phi^2 (V)$ is $\OO(h^\infty)$--close to $H_\Phi(V)$. To see that, let 
	\[
	\Pi_\Phi : L_\Phi^2 (V) \to H_\Phi(V) 
	\]
	be the orthogonal projection. We note that the solution of the $\overline{\partial}$--equation
	\[
	\partial_{\ybar} u(y) = \partial_{\ybar} k_x(y),\quad u\in L_\Phi^2 (V),
	\]
	with the minimal $L_\Phi^2 (V)$--norm is given by $(1-\Pi_\Phi)k_x$. Applying H\"ormander's $L^2$--estimates for the $\overline{\partial}$--equation on the open pseudoconvex $V$ and the strictly plurisubharmonic function $\Phi$, see for instance \cite[Proposition 4.2.5]{hormander1994convexity}, we obtain with \eqref{dbar k_x estimates},
	\begin{equation}
		\label{k_x-PiPhik_x L2 norm}
		\| (1-\Pi_\Phi)k_x \|_{L_\Phi^2 (V)}\leq Ch^{1/2} \|\partial_{\ybar} k_x\|_{L_\Phi^2 (V)} \leq \OO(h^\infty) .
	\end{equation}
	Applying \eqref{reproducing property 1} to $\Pi_\Phi k_x\in H_\Phi(V)$, with a smaller neighborhood $U\Subset V$ of $x_0$, we get
	\begin{equation}
		\label{L2 reproducing for PiPhik_x}
		\| (\widetilde{\Pi}_V - 1)\Pi_\Phi k_x \|_{L_{\Phi}^2(U)} = \OO(h^\infty) \|\Pi_\Phi k_x\|_{H_\Phi(V)} \leq \OO(h^\infty),
	\end{equation}
	where we used \eqref{k_x H_Phi norm} and the fact that $\Pi_\Phi = \OO(1): L_\Phi^2(V) \to H_\Phi(V)$ in the second inequality. Let us recall that 
	\[
	\widetilde{\Pi}_V = \OO(1) : L_\Phi^2(V) \to L_\Phi^2(V), 
	\]
	which is implied by the following estimate
	together with the Schur test. We can therefore combine \eqref{k_x-PiPhik_x L2 norm} and \eqref{L2 reproducing for PiPhik_x} to obtain
	\begin{equation}
		\label{Pi_V k_x - k_x L2 norm}
		\begin{split}
			\|(\widetilde{\Pi}_V - 1) k_x \|_{L_{\Phi}^2(U)} &\leq \|(\widetilde{\Pi}_V - 1)\Pi_\Phi k_x \|_{L_{\Phi}^2(U)} + \| (\widetilde{\Pi}_V - 1)((1-\Pi_\Phi) k_x) \|_{L_{\Phi}^2(V)} \\
			&\leq \OO(h^\infty) + \OO(1)\OO(h^\infty) = \OO(h^\infty).	
		\end{split}
	\end{equation}
	
	We shall derive a pointwise estimate for $h\partial_{\ybar}(\widetilde{\Pi}_V k_x)(y)$ in $y\in U$. In view of \eqref{Psi almost holomorphic}, \eqref{Bergman asymptotic} and \eqref{a_j almost holomorphic}, using \eqref{Bergman projection} we have
	\[
	\begin{split}
		{ }&\qquad |h\partial_{\ybar}(\widetilde{\Pi}_V k_x)(y)|e^{-\frac{1}{h}\Phi(y)} \\
		&\leq \frac{1}{h^n} \int_V e^{\frac{1}{h}(2\Re\Psi(y,\wbar) - \Phi(y) - \Phi(w))} (|\partial_{\ybar} \Psi(y,\wbar)| + h|\partial_{\ybar} a(y,\wbar)|) |k_x(w)| e^{-\frac{1}{h}\Phi(w)} L(dw) \\
		&\leq \frac{1}{h^n}\int_V e^{-\frac{|y-w|^2}{Ch}} \OO(|y-w|^\infty) |k_x(w)| e^{-\frac{1}{h}\Phi(w)} L(dw). 
	\end{split}
	\]
	It then follows from \eqref{basic inequality 1} and \eqref{k_x H_Phi norm}, combined with the Cauchy--Schwarz inequality, that
	\begin{equation}
		\label{dbar Pi_V k_x estimate}
		\left\|e^{-\frac{1}{h}\Phi(y)} h\partial_{\ybar}(\widetilde{\Pi}_V k_x)(y) \right\|_{L^\infty(U)} = \OO(h^\infty) .
	\end{equation}
	
	Let us now recall a well-known result, which allows one to deduce pointwise estimates from the weighted $L^2$--estimates like \eqref{Pi_V k_x - k_x L2 norm}, as follows:
	\begin{prop}\label{prop:L2 to pointwise}
		Suppose that $\widetilde{U}\Subset U \Subset \Omega$ are open. Then there exists $C>0$ such that for any $f\in L_\Phi^2 (U)$ with $h\overline{\partial} f\in L^\infty(U)$ and $h>0$ sufficiently small, we have
		\begin{equation}
			\label{L2 and dbar to pointwise}
			|f(z)| e^{-\frac{1}{h}\Phi(z)} \leq C \left(	\left\|e^{-\frac{1}{h}\Phi(y)} h\partial_{\ybar}f(y)\right\|_{L^\infty(U)} + h^{-n}\|f\|_{L_\Phi^2 (U)} \right),\quad z\in\widetilde{U}.
		\end{equation} 
	\end{prop} 
	\noindent
	We omit the proof here and refer the reader to \cite[Proposition 5.1]{hitrik2022smooth}. Let us proceed to apply \eqref{L2 and dbar to pointwise} to $f = k_x - \widetilde{\Pi}_V k_x$, recalling \eqref{Pi_V k_x - k_x L2 norm}, \eqref{dbar k_x estimates} and \eqref{dbar Pi_V k_x estimate}, we conclude that
	\[
	\forall z\in\widetilde{U},\quad |k_x(z) - \widetilde{\Pi}_V k_x (z)| e^{-\frac{1}{h}\Phi(z)} = \OO(h^\infty)	.
	\]
	Recalling \eqref{Bergman projection} and \eqref{eqn:k_x(y)}, noting that $k_x(x) = e^{\frac{1}{h}\Phi(x)}$, we therefore obtain 
	\begin{equation}
		\label{Charles recursive: the integral}
		\frac{1}{h^n}\int_V e^{\frac{i}{h}\widetilde{f}_x(y)} a(x,\overline{y};h) \,L(dy) = 1 + \OO(h^\infty),\quad x\in\widetilde{U},
	\end{equation}
	where
	\begin{equation}
		\label{eqn:tilde f(x,y)}
		\widetilde{f}_x(y) := 2i\left( \Phi(y) + \Phi(x) - \Psi(x,\overline{y}) - \Psi(y,\xbar) \right) .
	\end{equation}
	In view of the Taylor expansions \eqref{Phi Taylor expansion}-\eqref{Psi Taylor expansion y,xbar}, we have
	\begin{equation}
		\label{eqn:tilde f(x,y) expansion}
		\widetilde{f}_x(y) = 2i\Phi_{x\overline{x}}''(x)\overline{(y-x)}\cdot (y-x) + \OO(|y-x|^3).
	\end{equation}
	Recalling the strict plurisubharmonicity of $\Phi$, we may therefore assume that $V\Subset\Omega$ is small enough so that for some $C>0$, we have
	\begin{equation}
		\label{Im f_x positive definite}
		\Im \widetilde{f}_x(y)\geq C^{-1} |y-x|^2,\quad x, y\in V.
	\end{equation}
	
	Let us denote by $\Hess_\RR (f)$ the Hessian matrix of $f\in\CI(V)$ with respect to real coordinates $(\Re y,\Im y)\in\RR^{2n}\simeq \CC^n$. By a direct computation we have
	\begin{equation}
		\label{real Hessian tilde fx}
		(\Hess_\RR(\widetilde{f}_x)|_x)^{-1}\partial_{\Re y,\Im y}\cdot \partial_{\Re y,\Im y} = -i(\Phi_{x\overline{x}}''(x))^{-1}\partial_y\cdot \partial_{\overline{y}}.
	\end{equation}
	Using \eqref{eqn:tilde f(x,y) expansion}, \eqref{Im f_x positive definite} and \eqref{real Hessian tilde fx}, we can apply the complex stationary phase lemma given in \cite[Theorem 7.7.5]{hormander1985analysis} to achieve a full asymptotic expansion for the integral in \eqref{Charles recursive: the integral}, as $h\to 0^+$. Therefore, we obtain differential operators $L_{j,x}$ of order $2j$ such that,
	\begin{equation}
		\label{Lj,x expansion}
		\sum_{j=0}^\infty h^j (L_{j,x} a(x,\overline{y};h))(x) = 1 + \OO(h^\infty).
	\end{equation}
	Here $L_{j,x}$ are differential operators in $y$ with the following expressions,
	\begin{equation}
		\label{eqn:L_j Charles}
		L_{j,x}v(x) = \frac{\pi^n}{2^n \det\Phi_{x\overline{x}}''(x)} \sum_{\nu-\mu = j}\sum_{2\nu\geq 3\mu}\frac{(-1)^\mu}{2^j \mu! \nu!}((\Phi_{x\overline{x}}''(x))^{-1}\partial_y\cdot \partial_{\overline{y}})^\nu (\widetilde{g}_x^\mu v)\bigg\lvert_{y=x},
	\end{equation}
	where the function $\widetilde{g}_x(y)=\OO(|y-x|^3)$ is given explicitly by
	\begin{equation*}
		\label{eqn:g_x Charles}
		\widetilde{g}_x(y) := (2i)^{-1}\widetilde{f}_x(y) - \Phi_{x\overline{x}}''(x)\overline{(y-x)}\cdot (y-x).
	\end{equation*}
	For future reference, we shall also compute the derivatives of $\widetilde{g}_x$ at $y=x$ by \eqref{eqn:tilde f(x,y)}. We remark that for any $\alpha$, $\beta\in\NN^n$, 
	\begin{equation}
		\label{derivs tilde g_x}
		\begin{split}
			\partial_y^\alpha \partial_{\overline{y}}^\beta \widetilde{g}_x(y) \big\lvert_{y=x} = \begin{cases}
				\partial_x^\alpha \partial_{\overline{x}}^\beta \Phi(x),\,\textrm{ if }|\alpha|, |\beta|\geq 1\textrm{ and }|\alpha+\beta|\geq 3; \\
				0,\quad\textrm{otherwise.}
			\end{cases}
		\end{split}
	\end{equation}
	Recalling that $a(x,y;h)$ has the asymptotic expansion \eqref{Bergman asymptotic}, we get from \eqref{Lj,x expansion},
	\begin{equation}
		\label{pre recursive}
		\sum_{m=0}^\infty h^m \sum_{j+k=m} (L_{j,x}a_k(x,\overline{y}))(x) = 1 + \OO(h^\infty).
	\end{equation}
	It follows that
	\begin{equation*}
		L_{0,x} a_0(x,\overline{x}) = 1,\quad \sum_{j+k=m} (L_{j,x}a_k(x,\overline{y}))(x) = 0,\quad m\geq 1.
	\end{equation*}
	Using \eqref{eqn:L_j Charles}, we conclude from \eqref{pre recursive} the following recursive formula on $a_j$.
	\begin{prop}
		For each $x\in\widetilde{U}\Subset\Omega$, the Bergman coefficients $a_j$ in \eqref{Bergman asymptotic} satisfy
		\begin{equation}
			\label{recursive formula eq}
			\begin{gathered}
				a_m(x,\overline{x}) = -\sum_{j=1}^m \sum_{\nu-\mu = j}\sum_{2\nu\geq 3\mu}\frac{(-1)^\mu}{2^j \mu! \nu!}\bigl((\Phi_{x\overline{x}}''(x))^{-1}\partial_y\cdot \partial_{\overline{y}}\bigr)^\nu \bigl(\widetilde{g}_x^\mu(y) a_{m-j}(x,\overline{y})\bigr)\Big\lvert_{y=x}, \\
				\textrm{with}\quad a_0(x,\overline{x}) = \frac{2^n}{\pi^n} \det\Phi_{x\overline{x}}''(x) .
			\end{gathered}
		\end{equation}
	\end{prop}

	\subsection{Estimating the Bergman coefficients}
	\label{subsection:Bergman coefficients estimates}
	
	In this subsection, our goal is to prove the following quantitative estimates on the growth rate (in $j$) of the Bergman coefficients.
	\begin{prop}\label{BK upper bounds prop}
		Let $s>1$. Let $\Omega\subset\CC^n$ and $\Phi\in\GG^s(\Omega)$ be as in Theorem \ref{main thm}. For any $x_0\in \Omega$, there exist a small neighborhood $U\Subset \Omega$ of $x_0$ and constant $C>0$ such that for any $\alpha, \beta\in\NN^n$ and any integer $m\geq 0$ we have
		\begin{equation}\label{BK upper bounds eq}
			\bigl| \partial^{\alpha}_x \partial_{\xbar}^{\beta} a_m(x, \xbar) \bigr|\leq C^{|\alpha|+|\beta|+m+1}\alpha!^s\beta!^s m!^{2s-1},\quad x\in U. 
		\end{equation}
	\end{prop}
	
	\begin{rem}\label{optimality rmk}
		The estimates in \eqref{BK upper bounds eq} is optimal in the sense that the exponent $2s-1$ cannot be reduced. By a result of Carleson \cite[Theorem 2 and Example 2]{carleson1961universal}, we can construct a function $\phi(r)\in \mathcal{G}^s(\mathbb{R})$ with the prescribed Taylor series $\sum_{j=2}^{\infty} j!^{2s-2}r^{2j}$ at $r=0$. Take $\Omega$ as a small disk $D(0, \delta)$ in $\mathbb{C}$ and $\Phi(x)=\frac{1}{2}|x|^2-\frac{1}{2}\varphi(|x|)$ for $x\in \Omega$. Clearly, $\Phi$ is strictly plurisubharmonic on $\Omega$ when $\delta>0$ is small enough. Since $\Phi$ is radial, $\{x^j\}_{j=0}^{\infty}$ forms an orthogonal basis of the Bergman space $H_{\Phi}(\Omega)$. Let $I_j=\int_{\Omega} |x|^{2j} e^{-2\Phi(x)/h}L(dx)$, then the Bergman kernel is given by $K(x, \xbar)=\sum_{j=0}^{\infty} |x|^{2j}/I_j$. In particular, $K(0, 0)=1/I_0$. A straightforward computation gives 
		\begin{equation*}
			I_0=\sum_{l=0}^{\infty}\int_{\Omega} e^{-\frac{1}{h}|x|^2}\frac{1}{h^ll!} \Bigl(\sum_{j=2}^{\infty} j!^{2s-2}|x|^{2j}\Bigr)^l L(dx)+\OO(h^{\infty}), \quad \mbox{ as } h\rightarrow 0. 
		\end{equation*}
		Recall the notation of majorant: for two formal series $\sum_{j=0}^{\infty} b_j h^j$ and $\sum_{j=0}^{\infty} c_j h^j$, we write $\sum_{j=0}^{\infty} b_j h^j >>_h \sum_{j=0}^{\infty} c_j h^j$ if $b_j\geq |c_j|$ for all $j$. Since all coefficients in $\sum_{j=2}^{\infty} j!^{2s-2}r^{2j}$ are nonnegative, we have
		\begin{align*}
			I_0>>_h &\sum_{l=0}^{1}\int_{\Omega} e^{-\frac{1}{h}|x|^2}\frac{1}{h^ll!} \Bigl(\sum_{j=2}^{\infty} j!^{2s-2}|x|^{2j}\Bigr)^l L(dx)
			\\=&\int_{\mathbb{C}} e^{-\frac{1}{h}|x|^2}L(dx)+\sum_{j=2}^{\infty}\frac{j!^{2s-2}}{h}\int_{\mathbb{C}} e^{-\frac{1}{h}|x|^2} |x|^{2j}L(dx)+\OO(h^{\infty})
			\\=&\pi h\Bigl(1+\sum_{j=2}^{\infty} j!^{2s-1} h^{j-1}\Bigr)+\OO(h^{\infty}). 
		\end{align*}
		Note taking the reciprocal will not change the Gevrey order (see \cite[Proposition 9.1]{HLX20}), we obtain that the Bergman coefficients $a_j(0, 0)$ cannot be bounded by $C^{j+1}j!^{2s'-1}$ for any $s'<s$. 
	\end{rem}
	
	\begin{proof}[Proof of Proposition \ref{BK upper bounds prop}]
		\textbf{Step 1.} We shall first simplify the recursive formula \eqref{recursive formula eq}. 
		
		Let $U\Subset \Omega$ be a small open neighborhood of $x_0\in\Omega$ where \eqref{recursive formula eq} holds. By denoting 
		\[
		g^{i\overline{j}} := \big((\Phi_{x\xbar}''(x))^{-1}\big)_{ji},\quad 1\leq i, j\leq n, 
		\]
		we have
		\begin{align*}
			((\Phi_{x\overline{x}}''(x))^{-1}\partial_y\cdot \partial_{\overline{y}})^\nu &= \Bigl( \sum_{i, j=1}^n g^{i\overline{j}}(x) \partial_{y_i}\partial_{\overline{y_j}} \Bigr)^{\nu} \\
			&=\sum_{1\leq i_1, j_1, \cdots i_{\nu}, j_{\nu}\leq n} g^{i_1\overline{j_1}}\cdots g^{i_{\nu}\overline{j_{\nu}}} \partial_{y_1}\cdots \partial_{y_{\nu}} \partial_{\overline{y_{j_1}}}\cdots \partial_{\overline{y_{j_{\nu}}}}.
		\end{align*}
		Denote $I=(i_1, i_2, \cdots, i_{\nu}), J=(j_1, j_2, \cdots, j_{\nu})$
		and $g^{I\overline{J}}=g^{i_1\overline{j_1}}g^{i_2\overline{j_2}}\cdots g^{i_{\nu}\overline{j_{\nu}}}$. Let $\{e_j\}_{j=1}^n$ be the standard basis vectors in $\mathbb{R}^n$. Consider the map 
		\begin{equation*}
			\Gamma: \{1, \cdots, n\}^{\nu} \times \{1, \cdots, n\}^{\nu} \rightarrow \{(\alpha, \beta)\in \NN^n\times \NN^n: |\alpha|=|\beta|=\nu\}
		\end{equation*}
		defined by $\Gamma(I, J)=(e_{i_1}+\cdots e_{i_{\nu}}, e_{j_1}+\cdots e_{j_{\nu}})$. The map $\Gamma$ is clearly surjective. For any $(\alpha, \beta)$ and $\nu$, set 
		\begin{equation}\label{Aab eq}
			A_{\alpha\beta\nu}:=\Gamma^{-1}(\alpha, \beta)=\left\{(I,J):
			\begin{array}{ll}
				e_{i_1}+e_{i_2}+\cdots+e_{i_v}=\alpha\\ e_{j_1}+e_{j_2}+\cdots+e_{j_v}=\beta
			\end{array} \right\}.
		\end{equation}
		Then we can write 
		\begin{equation*}       
			((\Phi_{x\overline{x}}''(x))^{-1}\partial_y\cdot\partial_{\overline{y}})^\nu=\sum_{|\alpha|=|\beta|=\nu}\sum_{A_{\alpha\beta\nu}} g^{I\overline{J}} \partial_y^{\alpha} \partial_{\overline{y}}^{\beta}.
		\end{equation*}
		In view of \eqref{a_j almost holomorphic} we have
		\begin{align*}
			{ }&\quad((\Phi_{x\overline{x}}''(x))^{-1}\partial_y\cdot\partial_{\overline{y}})^\nu\bigl(\widetilde{g}_x^\mu(y)\,a_{m-j}(x,\overline{y})\bigr)\Big|_{y=x} \\
			&=\sum_{|\alpha|=|\beta|=\nu}\sum_{A_{\alpha\beta\nu}} g^{I\overline{J}} \sum_{\beta_1+\beta_2=\beta}\binom{\beta}{\beta_1,\beta_2}\partial_y^{\alpha}\partial_{\ybar}^{\beta_1} \widetilde{g}_x^\mu\big|_{y=x} \partial_{\xbar}^{\beta_2}a_{m-j}(x,\xbar).
		\end{align*}
		Note that $\partial^{\alpha}_{y} \widetilde{g}_x\big|_{y=x}=0$ for any $\alpha$ by \eqref{derivs tilde g_x}, thus, among all the terms on the right side of the above equation, we only need to consider those with $|\beta_1|\geq \mu$, as $\partial_y^{\alpha}\partial_{\ybar}^{\beta_1}\widetilde{g}_x^\mu\big|_{y=x}$ vanishes otherwise. Hence, we can assume $|\beta_2|=|\beta|-|\beta_1|\leq \nu-\mu=j$. 
		
		By the above computations, \eqref{recursive formula eq} implies
		\begin{align}\label{recursive formula 2 eq}
			a_m=-\sum_{j=1}^m\sum_{\substack{\nu-\mu=j \\ 2\nu\geq 3\mu}}\frac{(-1)^{\mu}}{2^j\mu!\nu!}\sum_{|\alpha|=|\beta|=\nu}\sum_{A_{\alpha\beta\nu}} g^{I\overline{J}} \sum_{\substack{\beta_1+\beta_2= \beta\\ |\beta_2|\leq j}}\binom{\beta}{\beta_1,\beta_2}\partial_y^{\alpha}\partial_{\ybar}^{\beta_1} \widetilde{g}_x^\mu\big|_{y=x} \, \partial_{\xbar}^{\beta_2}a_{m-j}.
		\end{align} 
		Since in \eqref{recursive formula 2 eq} the Bergman coefficient $a_m$ relies on not only the previous ones but also their derivatives, in order to estimate $a_m$ recursively, we need to estimate their derivatives at the same time. In this regard, we apply $\partial_x^{\gamma}\partial^{\delta}_{\xbar}$ to both sides of \eqref{recursive formula 2 eq} and obtain a recursive formula for the derivatives of $a_m(x,\xbar)$ as follows. 
		\begin{align}\label{recursive formula on derivatives eq}
			\begin{split}
				\partial_x^{\gamma}\partial_{\xbar}^{\delta}a_m
				=-&\sum_{j=1}^m\sum_{\substack{\nu-\mu=j\\ 2\nu\geq 3\mu}}\frac{(-1)^{\mu}}{2^j\mu!\nu!}
				\sum_{\substack{\gamma_1+\gamma_2+\gamma_3=\gamma \\ \delta_1+\delta_2+\delta_3=\delta}}\binom{\gamma}{\gamma_1,\gamma_2,\gamma_3}\binom{\delta}{\delta_1,\delta_2,\delta_3}\sum_{|\alpha|=|\beta|=\nu}\sum_{A_{\alpha\beta\nu}} 
				\\
				&\partial_x^{\gamma_1}\partial_{\xbar}^{\delta_1}g^{I\overline{J}}\cdot\sum_{\substack{\beta_1+\beta_2=\beta\\|\beta_2|\leq j}}\binom{\beta}{\beta_1,\beta_2} \partial_x^{\gamma_2}\partial_{\xbar}^{\delta_2}\left(\partial_y^{\alpha}\partial_{\ybar}^{\beta_1} \widetilde{g}_x^\mu \big|_{y=x} \right) \partial_x^{\gamma_3}\partial_{\xbar}^{\beta_2+\delta_3}a_{m-j}.
			\end{split}
		\end{align}
		
		\noindent
		\textbf{Step 2.} We shall estimate the factors $\partial_x^{\gamma_1}\partial_{\xbar}^{\delta_1}g^{I\overline{J}}$ and $\partial_x^{\gamma_2}\partial_{\xbar}^{\delta_2}\left(\partial_y^{\alpha}\partial_{\ybar}^{\beta_1} \widetilde{g}_x^\mu \big|_{y=x} \right)$ in \eqref{recursive formula on derivatives eq}, so that we can further simplify \eqref{recursive formula on derivatives eq} to derive a recursive inequality on the pointwise norms of derivatives of the Bergman coefficients $a_j(x,\xbar)$. We separate the estimates of $\partial_x^{\gamma_1}\partial_{\xbar}^{\delta_1}g^{I\overline{J}}$ and $\partial_x^{\gamma_2}\partial_{\xbar}^{\delta_2}\left(\partial_y^{\alpha}\partial_{\ybar}^{\beta_1} \widetilde{g}_x^\mu \big|_{y=x} \right)$ into the following two lemmas.
		
		\begin{lem}\label{factor 1 lem}
			Let $\alpha, \beta\in \NN^n$ such that $|\alpha|=|\beta|=\nu$, and let $(I,J)\in A_{\alpha\beta\nu}$ given in \eqref{Aab eq}. For $U\Subset \Omega$, there exists a positive constant $C$ such that for any multi-indices $\gamma,\delta\in\NN^n$ and any $x\in U$, we have
			\begin{align*}
				\left|\partial_x^{\gamma}\partial_{\xbar}^{\delta}\, g^{I\overline{J}}(x)\right|\leq C^{|\gamma+\delta|+\nu}\gamma!^s\delta!^s.
			\end{align*}
		\end{lem}
		
		\begin{proof}
			Since $\Phi\in \GG^s(\Omega)$, each entry of the matrix $\Phi_{x\xbar}''(x)$ is also in $\GG^s(\Omega)$. In addition, since $\Phi$ is strictly plurisubharmonic, $\Phi_{x\xbar}''(x)$ is positive definite, and it follows that each entry of the inverse matrix $(\Phi_{x\xbar}''(x))^{-1}$ also lies in $\GG^s(\Omega)$. Thus, for $U\Subset \Omega$ there exists $C>0$ such that for any $1\leq i, j \leq n$, and any $\gamma, \delta \geq 0$, we have
			\begin{equation*}
				\left|\partial_x^{\gamma}\partial_{\xbar}^{\delta}\, g^{i\overline{j}}(x)\right|\leq C^{|\gamma+\delta|+1}\gamma!^s \delta!^s,\quad x\in U.
			\end{equation*} 
			Since 
			\begin{equation*}
				\partial_x^{\gamma}\partial_{\xbar}^{\delta}\, g^{I\overline{J}}=\sum_{\substack{\gamma_1+\cdots+\gamma_{\nu}=\gamma \\ \delta_1+\cdots +\delta_{\nu}=\delta}} \binom{\gamma}{\gamma_1, \cdots, \gamma_{\nu}} \binom{\delta}{\delta_1, \cdots, \delta_{\nu}}\partial_x^{\gamma_1} \partial_{\xbar}^{\delta_1} g^{i_1\overline{j_1}}\cdots \partial_x^{\gamma_{\nu}}\partial_{\xbar}^{\delta_{\nu}}g^{i_{\nu}\overline{j_\nu}},
			\end{equation*}
			we have
			\begin{align*}
				\left| \partial_x^{\gamma}\partial_{\xbar}^{\delta}\, g^{I\overline{J}}\right|\leq &\sum_{\substack{\gamma_1+\cdots+\gamma_{\nu}=\gamma \\ \delta_1+\cdots +\delta_{\nu}=\delta}} \binom{\gamma}{\gamma_1, \cdots, \gamma_{\nu}} \binom{\delta}{\delta_1, \cdots, \delta_{\nu}} C^{|\gamma+\delta|+\nu} \gamma_1!^s\delta_1!^s\cdots \gamma_{\nu}!^s\delta_{\nu}!^s\\
				=& C^{|\gamma+\delta|+\nu}\gamma!\delta! \sum_{\substack{\gamma_1+\cdots+\gamma_{\nu}=\gamma \\ \delta_1+\cdots +\delta_{\nu}=\delta}} \gamma_1!^{s-1}\delta_1!^{s-1}\cdots \gamma_{\nu}!^{s-1}\delta_{\nu}!^{s-1}. 
			\end{align*}
			By the facts $\gamma_1!\cdots\gamma_{\nu}!\leq \gamma!$, $\delta_1!\cdots\delta_{\nu}!\leq \delta!$ and Proposition \ref{prop:partition multi-index counting}, we conclude
			\begin{equation*}
				\left| \partial_x^{\gamma}\partial_{\xbar}^{\delta}\, g^{I\overline{J}}\right|\leq 2^{|\gamma|+|\delta|+2n(\nu-1)}C^{|\gamma|+|\delta|+\nu}\gamma!^s\delta!^s.
			\end{equation*}
			Hence the result follows by renaming $2^{2n}C$ to be $C$.
		\end{proof}
		
		\begin{lem}\label{factor 2 lem}
			For $U\Subset \Omega$, there exists a positive constant $C$ such that for any multi-indices $\alpha,\beta, \gamma, \delta\in \NN^n$, any integer $\mu\geq 1$ and any $x\in U$, we have
			\begin{align*}
				\left|\partial_x^{\gamma}\partial_{\xbar}^{\delta}\bigl(\partial_y^{\alpha}\partial_{\ybar}^{\beta} \widetilde{g}_x^\mu\big|_{y=x}\bigr)\right|\leq C^{|\alpha+\beta+\gamma+\delta|+\mu+1}\alpha!^s\beta!^s\gamma!^s\delta!^s\slash \mu!^{2s-2}.
			\end{align*}
		\end{lem}
		
		\begin{proof}
			By a straightforward computation, we have
			\begin{align*}
				\partial_y^{\alpha}\partial_{\overline{y}}^{\beta}\widetilde{g}_x^\mu\big|_{y=x}
				=\sum_{\substack{\alpha_1+\cdots+\alpha_{\mu}=\alpha\\
						\beta_1+\cdots+\beta_{\mu}=\beta}}\binom{\alpha}{\alpha_1,\cdots,\alpha_{\mu}}\binom{\beta}{\beta_1,\cdots,\beta_{\mu}}\prod_{k=1}^{\mu}\partial_y^{\alpha_k}\partial_{\overline{y}}^{\beta_k}\widetilde{g}_x\big|_{y=x}.
			\end{align*}
			Let us recall \eqref{derivs tilde g_x} that
			\begin{equation*}
				\partial_{y}^{\alpha}\partial_{\ybar}^{\beta}\widetilde{g}_x \big|_{y=x}=\begin{cases}
					\partial_x^{\alpha}\partial_{\xbar}^{\beta}\Phi(x) & \mbox{when } |\alpha|\geq 1, |\beta|\geq 1 \mbox{ and } |\alpha+\beta|\geq 3,\\
					0 & \mbox{otherwise}.
				\end{cases}
			\end{equation*}
			If we define 
			\begin{align*}
				B_{\alpha\beta\mu}=\left\{\{\alpha_k\}_{k=1}^{\mu},\{\beta_k\}_{k=1}^{\mu}: \begin{array}{ll}
					\alpha_1+\cdots+\alpha_{\mu}=\alpha,\\
					\beta_1+\cdots+\beta_{\mu}=\beta,\\
					|\alpha_k|\geq 1, |\beta_k|\geq 1,
					|\alpha_k+\beta_k|\geq 3
				\end{array}\right\},
			\end{align*} 
			then
			\begin{align*}
				\partial_y^{\alpha}\partial_{\ybar}^{\beta}\widetilde{g}_x^\mu\big|_{y=x}
				=\sum_{B_{\alpha\beta\mu}}\binom{\alpha}{\alpha_1,\cdots,\alpha_{\mu}}\binom{\beta}{\beta_1,\cdots,\beta_{\mu}}\prod_{k=1}^{\mu}\partial_x^{\alpha_k}\partial_{\xbar}^{\beta_k}\Phi(x).
			\end{align*}
			By taking more derivatives, we have
			\begin{align*}
				{}&\quad\quad\partial_x^{\gamma}\partial_{\xbar}^{\delta}\bigl(\partial_y^{\alpha}\partial_{\ybar}^{\beta}\widetilde{g}_x^\mu\big|_{y=x}\bigr)
				\\
				&=\sum_{C_{\gamma\delta\mu}}\sum_{B_{\alpha\beta\mu}}\binom{\gamma}{\gamma_1 \cdots \gamma_{\mu}}\binom{\delta}{\delta_1 \cdots \delta_{\mu}}\binom{\alpha}{\alpha_1 \cdots \alpha_{\mu}}\binom{\beta}{\beta_1 \cdots \beta_{\mu}}\prod_{k=1}^{\mu}\partial_x^{\alpha_k+\gamma_k}\partial_{\xbar}^{\beta_k+\delta_k}\Phi(x),
			\end{align*}
			where
			\begin{align*}
				C_{\gamma\delta\mu}=\left\{\{\gamma_j\}_{j=1}^{\mu},\{\delta_j\}_{j=1}^{\mu}: \begin{array}{ll}
					\gamma_1+\gamma_2+\cdots+\gamma_{\mu}=\gamma\\
					\delta_1+\delta_2+\cdots+\delta_{\mu}=\delta	\end{array}\right\}.
			\end{align*} 
			Since $\Phi\in \GG^s(\Omega)$, for $U\Subset \Omega$ there exists $C>0$ such that for any multi-indices $\alpha,\beta,\gamma,\delta$, any integer $\mu \geq 1$ and any $x\in K$, we have
			\begin{align*}
				{}&\quad\left|\partial_x^{\gamma}\partial_{\xbar}^{\delta}\bigl(\partial_y^{\alpha}\partial_{\ybar}^{\beta}\widetilde{g}_x^\mu\big|_{y=x}\bigr)\right|
				\\
				&\leq C^{|\alpha+\beta+\gamma+\delta|+\mu}\sum_{C_{\gamma\delta\mu}}\sum_{B_{\alpha\beta\mu}}\binom{\gamma}{\gamma_1 \cdots \gamma_{\mu}}\binom{\delta}{\delta_1 \cdots \delta_{\mu}}\binom{\alpha}{\alpha_1 \cdots \alpha_{\mu}}\binom{\beta}{\beta_1 \cdots \beta_{\mu}} \\
				&\qquad\qquad\qquad\qquad\qquad\qquad\cdot\prod_{k=1}^{\mu}(\alpha_k+\gamma_k)!^s(\beta_k+\delta_k)!^s
				\\
				&\leq C^{|\alpha+\beta+\gamma+\delta|+\mu}\alpha!\beta!\gamma!\delta!\sum_{C_{\gamma\delta\mu}}\sum_{B_{\alpha\beta\mu}}2^{s|\alpha+\beta+\gamma+\delta|}\prod_{k=1}^{\mu}\bigl(\alpha_k!\beta_k!\gamma_k!\delta_k!\bigr)^{s-1}.
			\end{align*}
			The above second inequality follows from the following elementary facts
			\begin{align*}
				\binom{\alpha_k+\gamma_k}{\alpha_k , \gamma_k}\leq 2^{|\alpha_k+\gamma_k|}, \quad \binom{\beta_k+\delta_k}{\beta_k , \delta_k}\leq 2^{|\beta_k+\delta_k|}.
			\end{align*}
			\textbf{Claim.} We have $\alpha_1!\cdots \alpha_{\mu}!\leq n^{\mu n}\alpha!\slash \mu!$ and $\beta_1!\cdots \beta_{\mu}!\leq n^{\mu n}\beta!\slash \mu!$.
			
			For $1\leq k\leq \mu$, we denote $\alpha_k=(\alpha_k^1, \cdots, \alpha_k^n)$. Then
			\begin{equation*}
				\prod_{k=1}^\mu \alpha_k! = \prod_{k=1}^\mu \prod_{m=1}^n \alpha_k^m ! \leq \prod_{m=1}^n \frac{(\alpha_1^m+\cdots+\alpha_{\mu}^m)!}{(\#\{1\leq k\leq \mu: \alpha_k^m\neq 0\})!} = \frac{\alpha!}{a_1\cdots a_n},	
			\end{equation*}
			with $a_m:=\#\bigl\{1\leq k\leq \mu: \alpha_k^m\neq 0\bigr\}$ for $1\leq m\leq n$. Here we have used the following combinatorial inequality:
			\[
			\binom{n_1+\cdots+n_m}{n_1,\ldots,n_m} \geq m!,\quad\text{if }n_k\geq 1,\ 1\leq k\leq m.
			\]
			We can prove the above inequality directly as follows:
			\[
			\begin{split}
				\binom{n_1+\cdots+n_m}{n_1,\ldots,n_m} = &\binom{n_1}{n_1}\binom{n_1+n_2}{n_2}\cdots\binom{n_1+n_2+\cdots+n_m}{n_m} \\
				\geq & (n_1+n_2)\cdots(n_1+n_2+\cdots+n_m) \geq 2\times\cdots\times m = m! .
			\end{split}
			\]
			Since $|\alpha_k|\geq 1$ for any $1\leq k\leq \mu$, we get $a_1+\cdots +a_n \geq \mu$. By the elementary fact $\binom{a_1+\cdots+a_n}{a_1, \cdots, a_n}\leq n^{a_1+\cdots+a_n}$, and noting the trivial bound $a_1+\cdots+a_n \leq n\mu$, we obtain
			\begin{align*}
				a_1!\cdots a_n!\geq n^{-(a_1+\cdots+a_n)} (a_1+\cdots+a_n)!\geq n^{-n\mu} \mu!.
			\end{align*}  
			It follows that $\alpha_1!\cdots \alpha_{\mu}!\leq \alpha!\slash(n^{-n\mu} \mu!) = n^{\mu n}\alpha!\slash \mu!$. Similarly, we can also verify the second inequality in the claim.
			
			By the above claim we further have
			\begin{align*}
				\left|\partial_x^{\gamma}\partial_{\xbar}^{\delta}\bigl(\partial_y^{\alpha}\partial_{\ybar}^{\beta}\widetilde{g}_x^\mu\big|_{y=x}\bigr)\right|
				&\leq
				C^{|\alpha+\beta+\gamma+\delta|+\mu}\alpha!\beta!\gamma!\delta!\sum_{C_{\gamma\delta\mu}}\sum_{B_{\alpha\beta\mu}}2^{s|\alpha+\beta+\gamma+\delta|}n^{2n\mu}\frac{(\alpha!\beta!\gamma!\delta!)^{s-1}}{\mu!^{2s-2}}
				\\
				&=(2^sC)^{|\alpha+\beta+\gamma+\delta|} (n^{2n}C)^{\mu}\frac{(\alpha!\beta!\gamma!\delta!)^s}{\mu!^{2s-2}} (\#B_{\alpha\beta\mu})(\#C_{\gamma\delta\mu}).
			\end{align*}
			By Proposition \ref{prop:partition multi-index counting}, we have $\# B_{\alpha\beta\mu}\leq 2^{|\alpha|+|\beta|+2n(\mu-1)}$, $\# C_{\gamma\delta\mu}\leq 2^{|\gamma|+|\delta|+2n(\mu-1)}$. Therefore,
			\begin{align*}
				\left|\partial_x^{\gamma}\partial_{\xbar}^{\delta}\bigl(\partial_y^{\alpha}\partial_{\ybar}^{\beta}\widetilde{g}_x^\mu\big|_{y=x}\bigr)\right|\leq (2^{s+1}C)^{|\alpha+\beta+\gamma+\delta|} (2^{4n}n^{2n}C)^{\mu}\frac{(\alpha!\beta!\gamma!\delta!)^s}{\mu!^{2s-2}}.
			\end{align*}
			So the desired result follows by renaming $2^{4n+s+1}n^{2n}C$ to be $C$.
		\end{proof}
		
		\noindent
		\textbf{Step 3.} Let us now come back to \eqref{recursive formula on derivatives eq}. By the above two lemmas, there exists $C>0$ such that for any multi-indices $\gamma, \delta\geq 0$ and any integer $m\geq 0$ we have on $U$
		\begin{align*}
			\left| \partial_x^{\gamma}\partial_{\xbar}^{\delta}a_m\right|
			\leq &\sum_{j=1}^m\sum_{\substack{\nu-\mu=j\\ 2\nu\geq 3\mu}}\frac{1}{2^j \mu!\nu!}
			\sum_{\substack{\gamma_1+\gamma_2+\gamma_3=\gamma \\ \delta_1+\delta_2+\delta_3=\delta}}\binom{\gamma}{\gamma_1,\gamma_2,\gamma_3}\binom{\delta}{\delta_1,\delta_2,\delta_3}\sum_{|\alpha|=|\beta|=\nu}\sum_{A_{\alpha\beta\nu}} \sum_{\substack{\beta_1+\beta_2=\beta\\|\beta_2|\leq j}}
			\\
			&\binom{\beta}{\beta_1,\beta_2}C^{|\alpha+\beta_1+\gamma_1+\gamma_2+\delta_1+\delta_2|+\nu+\mu+1}\frac{(\alpha!\beta_1!\gamma_1!\gamma_2!\delta_1!\delta_2!)^s}{\mu!^{2s-2}}  \left| \partial_x^{\gamma_3}\partial_{\xbar}^{\beta_2+\delta_3}a_{m-j}\right|.
		\end{align*}
		After simplification, using $\alpha!, \beta!\leq \nu!$, it writes into
		\begin{align*}
			\frac{\left| \partial_x^{\gamma}\partial_{\xbar}^{\delta}a_m(x,\xbar)\right|}{\gamma!^s\delta!^s}
			\leq \sum_{j=1}^m 2^{-j} \sum_{\substack{\nu-\mu=j\\ 2\nu\geq 3\mu}}\frac{\nu!^{2s-1}}{\mu!^{2s-1}}
			&\sum_{\substack{\gamma_1+\gamma_2+\gamma_3=\gamma \\ \delta_1+\delta_2+\delta_3=\delta}}\sum_{|\alpha|=|\beta|=\nu}\sum_{A_{\alpha\beta\nu}} \sum_{\substack{\beta_1+\beta_2=\beta\\|\beta_2|\leq j}} \\ 
			&\quad C^{|\gamma-\gamma_3|+|\delta-\delta_3|+4\nu} \frac{\left| \partial_x^{\gamma_3}\partial_{\xbar}^{\beta_2+\delta_3}a_{m-j}\right|}{\gamma_3!^s\beta_2!^s\delta_3!^s}.
		\end{align*}
		Note that $\#A_{\alpha\beta\nu}\leq n^{2\nu}$ by \eqref{Aab eq} and that $\nu! = (\mu+j)! \leq 2^{\nu}\mu!j!$. Moreover, we observe that for any $0\leq \gamma_3\leq \gamma$, $\#\left\{(\gamma_1, \gamma_2): \gamma_1+\gamma_2=\gamma-\gamma_3 \right\}\leq 2^{|\gamma-\gamma_3|}$. Using these facts we deduce that
		\begin{align*}
			\frac{\left| \partial_x^{\gamma}\partial_{\xbar}^{\delta}a_m\right|}{\gamma!^s\delta!^s}
			\leq \sum_{j=1}^m \frac{j!^{2s-1}}{2^j}\sum_{\nu=j}^{3j}\sum_{|\alpha|=|\beta|=\nu}\sum_{\substack{\gamma_3\leq\gamma\\ \delta_3\leq \delta}} \sum_{\substack{\beta_2\leq \beta\\|\beta_2|\leq j}} (2^{2s-1}n^2C^4)^{\nu}(2C)^{|\gamma-\gamma_3|+|\delta-\delta_3|} 
			\frac{\left| \partial_x^{\gamma_3}\partial_{\xbar}^{\beta_2+\delta_3}a_{m-j}\right|}{\gamma_3!^s\beta_2!^s\delta_3!^s}.
		\end{align*}
		By renaming $(2^{2s-1}n^2C^4)^3$ to be $C$, we obtain a recursive inequality on derivatives of the Bergman coefficients: 
		\begin{equation}\label{recursive inequality eq}
			\frac{\left| \partial_x^{\gamma}\partial_{\xbar}^{\delta}a_m\right|}{\gamma!^s\delta!^s}
			\leq \sum_{j=1}^m \frac{j!^{2s-1}}{2^j}\sum_{\nu=j}^{3j}\sum_{|\alpha|=|\beta|=\nu}\sum_{\substack{\gamma_3\leq\gamma\\ \delta_3\leq \delta}} \sum_{\substack{\beta_2\leq \beta\\|\beta_2|\leq j}} C^{|\gamma-\gamma_3|+|\delta-\delta_3|+j}
			\frac{\left| \partial_x^{\gamma_3}\partial_{\xbar}^{\beta_2+\delta_3}a_{m-j}\right|}{\gamma_3!^s\beta_2!^s\delta_3!^s}.
		\end{equation}
		
		\noindent
		\textbf{Step 4.} We will use \eqref{recursive inequality eq} to prove the desired estimates \eqref{BK upper bounds eq} on the Bergman coefficients $a_m(x,\xbar)$, $m\in \mathbb{N}$.
		
		\noindent
		\textbf{Claim.} There exists $M>0$  such that for any multi-indices $\gamma, \delta\geq 0$, any integer $m\geq 0$ and any $x\in U$ we have
		\begin{equation}\label{BK upper bounds M eq}
			\left|\partial_x^{\gamma}\partial_{\xbar}^{\delta}a_m(x, \xbar)\right|\leq \binom{m+|\delta|}{m}^s M^{m+1} (2C)^{|\gamma+\delta|} \gamma!^s\delta!^s m!^{2s-1},
		\end{equation}
		where $C$ is the same constant as in \eqref{recursive inequality eq}.
		
		Suppose that we have proved the claim, it follows that 
		\begin{align*}
			\left|\partial_x^{\gamma}\partial_{\xbar}^{\delta}a_m\right|\leq 2^{(m+|\delta|)s} M^{m+1} (2C)^{|\gamma+\delta|} \gamma!^s\delta!^s m!^{2s-1}\leq (2^{s+1}CM)^{|\gamma+\delta|+m+1} \gamma!^s\delta!^sm!^{2s-1}.
		\end{align*}
		So the estimates \eqref{BK upper bounds eq} follows by renaming $2^{s+1}CM$ to be $C$.
		
		In the rest, we will verify the above claim by induction. For $m=0$ and any multi-indices $\gamma, \delta\in\NN^n$, since $a_0(x,\xbar)=\frac{2^n}{\pi^n}\det \Phi_{x\xbar}''(x)\in \GG^s(U)$ by \eqref{recursive formula eq}, \eqref{BK upper bounds M eq} follows immediately. Now we assume \eqref{BK upper bounds M eq} holds up to $m-1$ and proceed to $m$. By \eqref{recursive inequality eq}, we deduce that
		\begin{align}\label{bm 1 eq}
			\begin{split}
				\frac{\left| \partial_x^{\gamma}\partial_{\xbar}^{\delta}a_m\right|}{(2C)^{|\gamma+\delta|}M^{m+1}\gamma!^s\delta!^s}
				\leq &\sum_{j=1}^m\sum_{\nu=j}^{3j}j!^{2s-1}\sum_{\substack{\gamma_3\leq\gamma\\ \delta_3\leq \delta}}  2^{-|\gamma-\gamma_3|-|\delta-\delta_3|}(2C^2\slash M)^j \sum_{|\alpha|=|\beta|=\nu}\sum_{\substack{\beta_2\leq \beta\\|\beta_2|\leq j}}
				\\&\quad \binom{\beta_2+\delta_3}{\beta_2}^s\binom{m-j+|\beta_2+\delta_3|}{m-j}^s(m-j)!^{2s-1}.
			\end{split}
		\end{align}
		Let us recall the combinatorial inequality $\binom{\beta_2+\delta_3}{\beta_2}\leq \binom{|\beta_2+\delta_3|}{|\beta_2|}$ and the identity
		\begin{equation*}
			\binom{m-j+|\beta_2+\delta_3|}{|\beta_2+\delta_3|}\binom{|\beta_2+\delta_3|}{|\beta_2|}
			=\binom{m-j+|\beta_2+\delta_3|}{|\delta_3|}\binom{m-j+|\beta_2|}{|\beta_2|}.
		\end{equation*}
		Since $|\beta_2|\leq j$ and $|\delta_3|\leq |\delta|$, we have
		\begin{equation*}
			\binom{m-j+|\beta_2+\delta_3|}{|\beta_2+\delta_3|}\binom{|\beta_2+\delta_3|}{|\beta_2|}
			\leq \binom{m+|\delta|}{|\delta|}\binom{m}{j}.
		\end{equation*}
		Plugging this into \eqref{bm 1 eq}, after simplification we get
		\begin{align*}
			\frac{\left| \partial_x^{\gamma}\partial_{\xbar}^{\delta}a_m\right|}{(2C)^{|\gamma+\delta|}M^{m+1}\gamma!^s\delta!^s}
			\leq &\binom{m+|\delta|}{|\delta|}^sm!^{2s-1} \sum_{j=1}^m\sum_{\nu=j}^{3j}\sum_{\substack{\gamma_3\leq\gamma\\ \delta_3\leq \delta}}  2^{-|\gamma_3|-|\delta_3|}\bigl(\frac{2C^2}{M}\bigr)^j \\
			& \quad\cdot\sum_{|\alpha|=|\beta|=\nu}\#\{\beta_2: \beta_2\leq \beta\}.
		\end{align*}
		Clearly,
		\begin{align*}
			\sum_{|\alpha|=|\beta|=\nu} \#\{\beta_2:\beta_2\leq \beta\}
			\leq \sum_{|\alpha|=|\beta|=\nu}2^{|\beta|}=\binom{\nu+n-1}{n-1}^22^{\nu}\leq 2^{3\nu+2n-2}.
		\end{align*}
		It follows that
		\begin{align*}
			\frac{\left| \partial_x^{\gamma}\partial_{\xbar}^{\delta} a_m (x,\xbar) \right|}{(2C)^{|\gamma+\delta|}M^{m+1}\gamma!^s\delta!^s}
			\leq& \binom{m+|\delta|}{|\delta|}^sm!^{2s-1} \sum_{j=1}^m\bigl(\frac{2C^2}{M}\bigr)^j \sum_{\nu=j}^{3j} 2^{3\nu+2n-2} \sum_{\gamma_3\leq \gamma}2^{-|\gamma_3|}\sum_{\delta_3\leq \delta}2^{-|\delta_3|}\\
			\leq& \binom{m+|\delta|}{|\delta|}^sm!^{2s-1} \sum_{j=1}^m\bigl(\frac{2C^2}{M}\bigr)^j \sum_{\nu=j}^{3j} 2^{3\nu+4n-2}\\
			\leq & \binom{m+|\delta|}{|\delta|}^sm!^{2s-1}2^{4n-1} \sum_{j=1}^m\bigl(\frac{2^{10}C^2}{M}\bigr)^j
		\end{align*}
		By choosing $M=2^{4n+10}C^2$, we have $2^{4n-1} \sum_{j=1}^m(2^{10}C^2\slash M)^j\leq 1$ and thus \eqref{BK upper bounds M eq} holds for $m$. Therefore, the induction is concluded and the proof is completed.			
	\end{proof}
	
	
	\subsection{The amplitude as an inversion of Fourier integral operator $A$}
	We have proved the estimates \eqref{BK upper bounds eq} on the Bergman coefficients $a_j(x,\xbar)$, $x\in U\Subset\Omega$, with $U$ being a small neighborhood of $x_0\in\Omega$. We shall now use Proposition \ref{prop:Gevrey almost holo extension} to extend $a_j$ off-diagonally, almost holomorphically to a neighborhood of $\{(x,\xbar) : x\in U\}$ in $\CC^{2n}$ given by
	\begin{equation}
		\label{complex neighborhood tilde U}
		\widetilde{U} = \{(x+z , \xbar-\zbar) : x\in U,\ z\in B_{\CC^n}(0,\rho) \} \subset \CC^{2n},
	\end{equation} 
	where $B_{\CC^n}(0,\rho)$ denotes the open ball in $\CC^n$, centered at $0$ with radius $\rho>0$. Moreover, we aim to find an almost holomorphic extension in $\GG_{\rm b}^s(\widetilde{U})$.
	
	To achieve that, let us identify the anti-diagonal subspace $\Lambda\subset \CC^{2n}$ with $\RR^{2n}$ using the holomorphic linear change of variables:
	\begin{equation}
		\label{eqn:kappa}
		\kappa : \CC^{2n}\owns (x,y) \mapsto \left(\frac{x+y}{2}, \frac{x-y}{2i}\right) \in \CC^{2n},
	\end{equation}
	then $\kappa(\{(x,\xbar) : x\in U\}) = U$. Here we shall view $U$ as a subset of $\RR^{2n}\cong \CC^n$, equalizing $x\in \CC^n$ and $(\Re x, \Im x)\in\RR^{2n}$. Let $U_1$ be a complex neighborhood of $U$ in $\CC^{2n}$,
	\[
	U_1 := \{x+iy : x\in U\subset\RR^{2n},\ y\in\RR^{2n},\,|y|<\rho\} \subset \CC^{2n}
	\]
	for some $\rho>0$, then we see that $U_1\cap\RR^{2n} = U$, $\kappa^{-1}(U_1) = \widetilde{U}$ given in \eqref{complex neighborhood tilde U}. The estimates \eqref{BK upper bounds eq} show that $(\kappa^{-1})^* a_j \in \GG_{\rm b}^s(U)$, and that there exist $C, R>0$ such that 
	\begin{equation}
		\label{a_j Gevrey semi-norm}
		\|(\kappa^{-1})^* a_j\|_{s,R,U} \leq C^{1+j} j!^{2s-1},\quad j\in\NN .
	\end{equation}
	By Proposition \ref{prop:Gevrey almost holo extension}, each $(\kappa^{-1})^* a_j$ admits an almost holomorphic extension $b_j \in \GG_{\rm b}^s(U_1)$. Let us now set
	\begin{equation*}
		\label{defn:a_j(x,y) Gevrey}	
		a_j(x,y) := b_j(\kappa(x,y)),\quad (x,y)\in\widetilde{U}\subset\CC^{2n},\ j\in\NN ,
	\end{equation*}
	this then extends $a_j(x,\xbar)$, $x\in U$ almost holomorphically to $\widetilde{U}$, in the sense that $\overline{\partial} a_j$ is flat along the anti-diagonal. Moreover, it follows from \eqref{estimate almost holo ext Gevrey semi-norm} and \eqref{a_j Gevrey semi-norm} that there exist $C_1, R_1>0$ such that
	\begin{equation*}
		\label{a_j(x,y) Gevrey semi-norm}
		\|a_j\|_{s,R_1,\widetilde{U}} \leq C_1^{1+j} j!^{2s-1},\quad j\in\NN .
	\end{equation*}
	Renaming $\max(C_1,R_1)$ to be $C$, we conclude that $\{a_j\}_{j=0}^\infty$ forms a formal $\GG^{s,2s-1}$ symbol (see Definition \ref{defi:formal Gevrey symbol}) on $\widetilde{U}$ satisfying
	\begin{equation}
		\label{a_j(x,y) as formal Gevrey symbol}
		\|\partial^\alpha a_j\|_{L^\infty(\widetilde{U})} \leq C^{1+|\alpha|+j}\alpha!^s j!^{2s-1},\quad \alpha\in\NN^{4n},\ j\in\NN .
	\end{equation}
	Here we identify $\CC^{2n}$ with $\RR^{4n}$ so $\partial$ includes $\partial_{\Re x_i}$, $\partial_{\Im x_i}$, $\partial_{\Re y_i}$, $\partial_{\Im y_i}$,  $i=1,2,\ldots,n$. 
	
	Proposition \ref{prop:Borel lemma Gevrey} shows that we can realize $\{a_j\}_{j=0}^\infty$ as a Gevrey function, i.e., there exist $a = a(x,y;h)\in \GG_{\rm b}^s(\widetilde{U})$ and $\widetilde{C}>0$ such that
	\begin{equation}
		\label{amplitude a(x,y;h) Gevrey asymp}
		\biggl\| \partial^\alpha \big(a - \sum_{j=0}^{N-1} a_j h^j\big) \biggr\|_{L^\infty(\widetilde{U})} \leq \widetilde{C}^{1+|\alpha|+N} \alpha!^s N!^{2s-1} h^N ,\quad\alpha\in\NN^{4n},\ N\in \NN .
	\end{equation}
	We note that $\partial_{\xbar,\ybar} a(x,y;h)$ is flat along the anti-diagonal $\{y=\xbar\}$, in view of \eqref{eqn:a realizing a_j} and the flatness of $\overline{\partial} a_j$, $j\in\NN$. It then follows from Proposition \ref{prop:flatness estimtate} and the estimate \eqref{amplitude a(x,y;h) Gevrey asymp} (with $N=0$) that there exists $C>0$ such that
	\begin{equation}
		\label{amplitude dbar estimate}
		|\partial_{\xbar,\ybar} a(x,y;h)| \leq C \exp \left(-C^{-1} |x - \ybar|^{-\frac{1}{s-1}} \right) .
	\end{equation}
	
	Let us recall that $\Omega\subset\CC^n$ is a pseudoconvex domain, $\Phi\in\GG^s(\Omega;\RR)$ is strictly plurisubharmonic on $\Omega$, and $U\Subset\Omega$ is a neighborhood of $x_0\in\Omega$. Noting that $\Phi\in\GG_{\rm b}^s(U)$, we obtain again by Propositions \ref{prop:Gevrey almost holo extension} and \ref{prop:flatness estimtate}, and the holomorphic linear map $\kappa$ given in \eqref{eqn:kappa}, a polarization $\Psi\in\GG_{\rm b}^s(\widetilde{U})$ (with $\widetilde{U}$ of the form \eqref{complex neighborhood tilde U}) satisfying   
	\begin{equation}
		\label{Psi as polarization Gevrey}
		\Psi(x,\xbar) = \Phi(x),\quad x\in U;
	\end{equation}
	\begin{equation}
		\label{Psi almost holomorphic Gevrey}
		|\partial_{\overline{x},\ybar}\Psi(x,y)| \leq C \exp\big(-C^{-1}|x-\ybar|^{-\frac{1}{s-1}}\big),\quad (x,y)\in\widetilde{U} .
	\end{equation} 
	
	In view of the quantitative flatness of $\overline{\partial} a$ and $\overline{\partial} \Psi$ given in \eqref{amplitude dbar estimate} and \eqref{Psi almost holomorphic Gevrey} respectively, we can prove a Gevrey version of Proposition \ref{prop:independent of contour} with an improved error term estimate.  
	\begin{prop}\label{prop: independent of contour Gevrey}
		There exists an open neighborhood $V_0\Subset \Omega$ of $x_0$ such that for any two good contours $\Gamma_j(y,\ybar)$ for $y\in V_0$, $j=0,1$, with respect to the phase $(x,\widetilde{y})\mapsto \phi(y,\ybar;x,\widetilde{y})$ given in \eqref{phase phi defn}, with the amplitude $a(x,\widetilde{y};h)$ satisfying \eqref{amplitude a(x,y;h) Gevrey asymp} and \eqref{amplitude dbar estimate}, we have uniformly for any $y\in V_0$
		\begin{equation*}
			(A_{\Gamma_0} a)(y,\ybar;h) - (A_{\Gamma_1} a)(y,\ybar;h) = \OO(1) \exp\left(-C^{-1}h^{-\frac{1}{2s-1}}\right),
		\end{equation*}
		where $C$ is a positive constant depending only on $a$, $\varphi$ and $s$.
	\end{prop}
	
	\begin{proof}
		According to the proof of \cite[Proposition 2.3]{hitrik2022smooth} (see also \cite[Chapter 3]{Sj82} and \cite[Chapter 1]{De92}), good contours are homotopic to each other through a family of uniformly good contours. To be precise, we recall that, by taking a sufficiently small neighborhood $V_0\Subset \Omega$ of $x_0$, there exist the Morse coordinates $(t,s)\in\neigh(0,\RR^{2n}\times\RR^{2n})$ with the smooth mapping 
		\[
		\gamma_y : (t,s)\mapsto (x(t,s),\widetilde{y}(t,s))\in \CC^{n}\times\CC^n,
		\] 
		depending smoothly on $y\in V_0$ as well, such that
		\[
		\Re\phi(y,\ybar;x(t,s),\widetilde{y}(t,s)) = \frac{1}{2}(t^2 - s^2).
		\]
		Furthermore, in the Morse coordinates $(t,s)$, every good contour $\Gamma(y,\ybar)$ takes the form
		\begin{equation}
			\label{good contour condition Morse coordinates}
			t = g(y;s),\quad s\in\neigh(0,\RR^{2n});\quad |g(y;s)|\leq \alpha |s|,\quad 0<\alpha<1.
		\end{equation}
		We therefore assume that good contours $\Gamma_j(y,\ybar)$, $j=0,1$ are parameterized by
		\[
		\Gamma_j(y,\ybar) : \neigh(0,\RR^{2n})\owns s \mapsto \gamma_y(g_j(y;s),s) \in\CC_x^n\times\CC_{\widetilde{y}}^n ,
		\]
		where the smooth mappings $g_j(y;\cdot)$, $j=0,1$ satisfy \eqref{good contour condition Morse coordinates}. It follows that $\Gamma_0(y,\ybar)$ and $\Gamma_1(y,\ybar)$ are homotopic via a family of uniformly good contours $\Gamma_{\theta}(y, \ybar)$, $\theta\in [0, 1]$, with the following parametrization
		\[
		\Gamma_\theta(y,\ybar) : \neigh(0,\RR^{2n})\owns s \mapsto \gamma_y(\theta g_1(y;s) + (1-\theta)g_0(y;s),s) \in\CC_x^n\times\CC_{\widetilde{y}}^n ,
		\] 
		and we have
		\begin{equation}\label{Gaussian phase along good contours}
			\Re \phi(y, \ybar;x,\widetilde{y}) \leq -C^{-1} \dist((x,\widetilde{y}),(y,\ybar))^2,\quad\forall (x,\widetilde{y})\in\Gamma_{\theta}(y,\ybar),\ \theta\in[0,1],
		\end{equation}
		where $C$ is a positive constant uniform for $\theta\in [0, 1]$ and $y\in V_0$. 
		
		In the rest of the proof, we will use $C$ to denote a positive constant independent from $\theta\in [0, 1]$ and $y\in V_0$, which could change from line to line. Letting 
		\[
		G_{[0, 1]}(y, \ybar) = \{\gamma_y(\theta g_1(y;s) + (1-\theta)g_0(y;s),s) : s\in\neigh(0,\RR^{2n}),\ 0\leq\theta\leq 1\},
		\]
		with the orientation given by the map $(s,\theta)\mapsto \gamma_y(\theta g_1(y;s) + (1-\theta)g_0(y;s),s)$, there exists $C>0$ such that for any $y\in V_0$, 
		\begin{equation*}
			\partial G_{[0, 1]}(y, \ybar)\setminus(\Gamma_1(y, \ybar)- \Gamma_0(y, \ybar))\subset \{(x, \widetilde{y}): \Re \phi(y, \ybar;x,\widetilde{y}) \leq -C^{-1}\}.
		\end{equation*}
		Therefore, in view of Proposition \ref{prop:Borel lemma Gevrey}, by setting $\omega:=e^{\frac{2}{h} \phi(y,\ybar;x,\widetilde{y})} a(x,\widetilde{y};h) \,dx \wedge d\widetilde{y}$ we have uniformly for $y\in V_0$,
		\begin{equation}\label{good contour eq}
			\int_{\partial G_{[0,1]}(y,\ybar)} \omega=\int_{\Gamma_1(y,\ybar)} \omega - \int_{\Gamma_0(y, \ybar)} \omega + \OO(e^{-\frac{1}{Ch}}).
		\end{equation}
		On the other hand, by Stokes' formula, 
		\begin{equation}\label{Stoke thm eq}
			\int_{\partial G_{[0,1]}(y,\ybar)} \omega=\int_{G_{[0,1]}(y,\ybar)} d\omega=\int_{G_{[0,1]}(y,\ybar)} \oo{\partial}\bigl(e^{\frac{2}{h} \phi(y,\ybar;x,\widetilde{y})} a(x,\widetilde{y};h)\bigr) \wedge dx \wedge d\widetilde{y}.
		\end{equation}
		(Rigorously speaking, we should apply Stokes' formula to the pull-back of $\omega$ under the map $(s,\theta)\mapsto \gamma_y(\theta g_1(y;s) + (1-\theta)g_0(y;s),s)$ on $\neigh_s(0,\RR^{2n})\times [0,1]_\theta$, whose image is $G_{[0,1]}(y,\ybar)$, but we omit the details here for simplicity.)
		
		\noindent 
		In view of \eqref{phase phi defn} and \eqref{Psi almost holomorphic Gevrey}, for any $1\leq j \leq n$, we have
		\begin{align*}
			\partial_{\xbar_j}\varphi(y, \ybar; x, \widetilde{y})&= \partial_{\xbar_j}\Psi(x,\widetilde{y}) - \partial_{\xbar_j}\Psi(x,\ybar)
			\\&= \OO(1) \Bigl(\exp\left(C^{-1}|x-\oo{\widetilde{y}}|^{-\frac{1}{s-1}}\right) + \exp\left(C^{-1}|x-y|^{-\frac{1}{s-1}}\right)\Bigr).
		\end{align*}
		Note $|x-\oo{\widetilde{y}}|\leq |x-y|+|y-\oo{\widetilde{y}}|\leq \sqrt{2}\,\dist((x,\widetilde{y}),(y,\ybar))$. We thus obtain
		\begin{equation*}
			\partial_{\xbar_j}\varphi(y, \ybar; x, \widetilde{y})=\OO(1) \exp\left(C^{-1}\dist((x,\widetilde{y}),(y,\ybar))^{-\frac{1}{s-1}}\right).
		\end{equation*}
		Similarly, we also have
		\begin{equation*}
			\partial_{\oo{\widetilde{y}}_j}\varphi(y, \ybar; x, \widetilde{y}) = \OO(1) \exp\left(C^{-1}\dist((x,\widetilde{y}),(y,\ybar))^{-\frac{1}{s-1}}\right) .
		\end{equation*}
		Following the same steps as above, we can deduce from \eqref{amplitude dbar estimate} that for $1\leq j\leq n$,
		\begin{equation*}
			\partial_{\xbar_j, \oo{\widetilde{y}}_j} a(x, \widetilde{y}; h) = \OO(1) \exp\left(C^{-1}\dist((x,\widetilde{y}),(y,\ybar))^{-\frac{1}{s-1}}\right) .
		\end{equation*}
		Combining the above estimates and \eqref{Gaussian phase along good contours}, we obtain uniformly along $G_{[0, 1]}(y, \ybar)$,
		\begin{equation*}
			\overline{\partial}\bigl(e^{\frac{2}{h} \phi(y,\ybar;x,\widetilde{y})} a(x,\widetilde{y};h)\bigr)=\OO(1) \exp\Bigl(-\frac{1}{Ch} \bigl(t^2 + h t^{-\frac{1}{s-1}}\bigr) \Bigr),\quad t=\dist((x,\widetilde{y}),(y,\ybar)) . 
		\end{equation*}
		We note by a straightforward computation that
		\begin{equation}
			\label{the minimum}
			\min_{t>0} \big(t^2 + h t^{-\frac{1}{s-1}}\big) = \frac{2s-1}{2s-2} (2s-2)^{\frac{1}{2s-1}} h^{1-\frac{1}{2s-1}}.
		\end{equation}
		We therefore conclude that there exists $C>0$ such that 
		\begin{equation*}
			\overline{\partial}\bigl(e^{\frac{2}{h} \phi(y,\ybar;x,\widetilde{y})} a(x,\widetilde{y};h)\bigr) = \OO(1) \exp\bigl(-C^{-1}h^{-\frac{1}{2s-1}}\bigr),\quad (x,\widetilde{y})\in G_{[0,1]}(y,\ybar),\ y\in V_0.
		\end{equation*}
		Combining these with \eqref{good contour eq} and \eqref{Stoke thm eq}, we finally get
		\begin{equation*}
			\int_{\Gamma_2(y,\ybar)} \omega-\int_{\Gamma_1(y, \ybar)} \omega= \OO(1)\Bigl(\exp\bigl(-C^{-1} h^{-\frac{1}{2s-1}}\bigr) + \exp\bigl(-C^{-1} h^{-1}\bigr)\Bigr).
		\end{equation*}
		So the desired result follows by noticing that $\frac{1}{2s-1}<1$ since $s>1$.
	\end{proof}
	We are now ready to prove the main result of this section:
	\begin{thm}
		\label{thm:asymp invert A}
		Let $\Omega\subset\CC^n$ be open and let $\Phi\in\GG^s(\Omega;\RR)$ be strictly plurisubharmonic in $\Omega$. For each $x_0\in\Omega$, there exists an elliptic symbol $a(x,\widetilde{y};h)\in\GG_{\rm b}^s(\neigh((x_0,\overline{x_0}),\CC^{2n}))$ realizing the following formal $\GG^{s,2s-1}$ symbol $\{a_j\}_{j=0}^{\infty}$ (see Definition \ref{defi:formal Gevrey symbol}):
		\[
		a(x,\widetilde{y};h) \sim \sum_{j=0}^{\infty} a_j(x,\widetilde{y}) h^j,
		\]
		in the sense of \eqref{amplitude a(x,y;h) Gevrey asymp}, with $a_j\in \GG_{\rm b}^s(\neigh((x_0,\overline{x_0}),\CC^{2n}))$ satisfying \eqref{a_j(x,y) as formal Gevrey symbol} and being holomorphic to $\infty$--order along the anti-diagonal $\widetilde{y}=\xbar$, such that
		\begin{equation}
			\label{eqn:Aa Gevrey}
			\begin{split}
				(A_\Gamma a)(y,\ybar;h) &= \frac{C_n}{h^n} \int_{\Gamma(y,\ybar)} e^{\frac{2}{h}\phi(y,\ybar;x,\widetilde{y})} a(x,\widetilde{y};h)\,dx\,d\widetilde{y} \\
				&= 1 + \OO(1)\exp\bigl(-C^{-1}h^{-\frac{1}{2s-1}}\bigr),
			\end{split}
			\quad y\in\neigh(x_0,\CC^n).
		\end{equation}
		Here $\phi$ is defined in \eqref{phase phi defn} with the polarization $\Psi$ of $\Phi$ satisfying \eqref{Psi as polarization Gevrey} and \eqref{Psi almost holomorphic Gevrey}; $\Gamma(y,\ybar)$ is a good contour with respect to the phase function $(x,\widetilde{y})\mapsto \phi(y,\ybar;x,\widetilde{y})$. The restriction of the $a_j$'s to the anti-diagonal $\widetilde{y} = \xbar$ are uniquely determined for $j\in\NN$, satisfying both recursive formulas \eqref{recursive formula HS} and \eqref{recursive formula eq}. 
	\end{thm}
	
	\begin{proof}
		We have already constructed the amplitude $a(x,\widetilde{y};h)\in\GG_{\rm b}^s(\widetilde{U})$ realizing a formal $\GG^{s,2s-1}$ symbol $\{a_j\}_{j=0}^\infty$ such that \eqref{amplitude a(x,y;h) Gevrey asymp} and \eqref{amplitude dbar estimate} hold. Here $\widetilde{U}$ is given in \eqref{complex neighborhood tilde U},
		\[
		\widetilde{U} = \{(y+z,\ybar-\zbar) : y\in U,\ |z|<\rho\},
		\]
		with $U\Subset\Omega$ being a small neighborhood of $x_0\in\Omega$, $\rho>0$ small. Therefore, we only need to prove \eqref{eqn:Aa Gevrey}. Let us set $U$ sufficiently small such that Proposition \ref{prop: independent of contour Gevrey} holds (with $V_0$ there replaced by $U$). It then suffices to prove \eqref{eqn:Aa Gevrey} with a particular choice of the good contour $\Gamma_0(y,\ybar)$, $y\in U$, defined in \eqref{good contour: affine}, for which we recall as follows
		\[
		\Gamma_0(y,\ybar) : B_{\CC^n}(0,\rho) \owns z \mapsto (y+z,\ybar-\zbar) \in \CC^{2n} .
		\]
		Recalling \eqref{recursive HS integral form} we write
		\begin{equation}
			\label{eqn:A Gamma_0 final}
			(A_{\Gamma_0}a)(y,\ybar;h) = \frac{1}{h^n} \int_{B_{\CC^n}(0,\rho)} e^{\frac{i}{h} f(y,z)} a(y+z,\ybar-\zbar;h)\,L(dz) ,
		\end{equation}
		with $f(y,z)$ given in \eqref{eqn:f(y,z)} satisfying 
		\[
		f(y,z) = 2i \Phi_{x\overline{x}}''(y) \zbar\cdot z + g(y,z),\quad g(y,z) = \OO(|z|^3).
		\]
		It follows from the strict plurisubharmonicity of $\Phi$ on $\Omega$ that we can assume $\rho>0$ to be small enough so that there exists $C_{U}>0$ such that
		\begin{equation}
			\label{Im f positive uniformly}
			\Im f(y,z) \geq C_{U}^{-1} |z|^2,\quad y\in U,\ |z|<\rho .
		\end{equation}
		Noting that
		\[
		f(y,z) = 2i (\Psi(y+z,\ybar) + \Psi(y,\ybar-\zbar) - \Psi(y+z,\ybar-\zbar) - \Psi(y,\ybar))
		\]
		and that $\Psi\in\GG_{\rm b}^s(\widetilde{U})$, we have for some $C_{\Psi,\widetilde{U}} > 0$,
		\begin{equation}
			\label{f(y,z) Gevrey estimate}
			|\partial_{\Re z}^\alpha \partial_{\Im z}^\beta f(y,z)| \leq C_{\Psi,\widetilde{U}}^{1+|\alpha|+|\beta|} \alpha!^s \beta!^s,\quad y\in U,\ |z|<\rho .
		\end{equation}
		Similarly, since $a\in\GG_{\rm b}^s(\widetilde{U})$, we have for some $C_{a,\widetilde{U}} > 0$,
		\begin{equation}
			\label{a(y,z) Gevrey estimate}
			|\partial_{\Re z}^\alpha \partial_{\Im z}^\beta (a(y+z,\ybar-\zbar))| \leq C_{a,\widetilde{U}}^{1+|\alpha|+|\beta|} \alpha!^s \beta!^s,\quad y\in U,\ |z|<\rho .
		\end{equation}
		We now apply Theorem \ref{stationary phase lemma Gevrey lem} to the integral in \eqref{eqn:A Gamma_0 final} to obtain an asymptotic expansion as $h\to 0^+$, we have for each $N\in\NN$, 
		\begin{equation}
			\label{eqn:A Gamma_0 asymp expansion}
			\biggl|(A_{\Gamma_0}a)(y,\ybar;h) - \sum_{j=0}^{N-1} h^j (L_{j,y} a)(y,\ybar)\biggr| \leq C^{1+N} N!^{2s-1} h^N,\quad y\in U.
		\end{equation}
		We emphasize that the constant $C>0$ here can be chosen uniform in $y\in U$ thanks to the estimates \eqref{Im f positive uniformly}-\eqref{a(y,z) Gevrey estimate}.  
		Since Theorem \ref{stationary phase lemma Gevrey lem} differs from \cite[Theorem 7.7.5]{hormander1985analysis} only in the form of the remainder estimate, we note that the operators $L_{j,y}$, $j\in\NN$ are identical to the expressions in \eqref{eqn:L_k,y operators}, given by
		\[
		\begin{split}
			(L_{j,y} a)(y,\ybar) = &\frac{\pi^n}{2^n \det(\Phi_{x\overline{x}}''(y))}\sum_{\nu-\mu=j}\sum_{2\nu\geq 3\mu}\frac{i^{\mu}}{2^\nu \mu! \nu!} \\ &\qquad\left(\big(\Phi_{x\overline{x}}''(y)\big)^{-1}\partial_z\cdot\partial_{\zbar}\right)^\nu \big(g(y,z)^\mu a(y+z,\ybar-\zbar;h)\big)\bigg\lvert_{z=0}
		\end{split} .
		\]
		For a complete deduction of the expressions of $L_{j,y}$ using \eqref{eqn:Lj}, we refer to \cite[Section 2]{hitrik2022smooth}, whereas we recall the key calculation as follows
		\[
		\big(\Hess_{\RR_z} (f(y,0))\big)^{-1} \partial_{\Re z, \Im z} \cdot \partial_{\Re z, \Im z} = i^{-1} \big(\Phi_{x\overline{x}}''(y)\big)^{-1}\partial_z\cdot\partial_{\zbar} ,
		\]
		where $\Hess_{\RR_z} (f(y,0))$ is the ($2n\times 2n$) Hessian matrix of $f(y,z)$ with respect to real coordinates $(\Re z, \Im z)\in\RR^{2n} \cong \CC^n$ at $z=0$. Let us proceed to recall the asymptotic expansion $a(x,\widetilde{y};h)\sim\sum a_k(x,\widetilde{y}) h^k$ satisfying \eqref{amplitude a(x,y;h) Gevrey asymp} and write
		\[
		(L_{j,y} a)(y,\ybar) = \sum_{k=0}^{N-1-j} 	h^k (L_{j,y} a_k)(y,\ybar) + \bigg(L_{j,y} \big(a - \sum_{k=0}^{N-1-j} a_k h^k\big)\bigg)(y,\ybar) .
		\] 
		It follows from \eqref{amplitude a(x,y;h) Gevrey asymp} that for some $\widetilde{C}_U >0$ we have
		\begin{equation}
			\label{tail a(x,y) Gevrey estimates}
			\begin{gathered}
				\biggl|\partial_{\Re z,\Im z}^\alpha \bigg(a(y+z,\ybar-\zbar;h) - \sum_{k=0}^{N-1-j} h^k a_k(y+z,\ybar-\zbar)\bigg)\bigg\lvert_{z=0} \biggr| \\
				\leq \widetilde{C}_U^{1+N-j} (N-j)!^{2s-1} h^{N-j} \widetilde{C}_U^{|\alpha|} \alpha!^s, \quad \alpha\in\NN^{2n},\ y\in U.
			\end{gathered}
		\end{equation}
		In view of \eqref{estimate L_j u}, together with \eqref{f(y,z) Gevrey estimate} and \eqref{tail a(x,y) Gevrey estimates}, we conclude that
		\begin{equation}
			\label{estimate L_j tail a(x,y)}
			\begin{split}
				\biggl| \bigg(L_{j,y} \big(a - \sum_{k=0}^{N-1-j} a_k h^k\big)\bigg)(y,\ybar) \biggr| &\leq C^{1+N} j!^{2s-1} (N-j)!^{2s-1} h^{N-j} \\
				&\leq C^{1+N} N!^{2s-1} h^{N-j},\quad y\in U,
			\end{split}
		\end{equation}
		for some $C>0$ (uniform in $y\in U$). Recalling the recursive formula \eqref{recursive formula HS} we get
		\begin{equation}
			\label{apply recursive HS}
			\sum_{j=0}^{N-1} h^j \sum_{k=0}^{N-1-j} 	h^k (L_{j,y} a_k)(y,\ybar) = \sum_{\ell=0}^{N-1} h^\ell \sum_{j+k = \ell} (L_{j,y} a_k)(y,\ybar) = 1,\quad N\geq 1.
		\end{equation}
		Combing \eqref{eqn:A Gamma_0 asymp expansion}, \eqref{estimate L_j tail a(x,y)} and \eqref{apply recursive HS}, we conclude for each $N\geq 1$,
		\begin{equation}
			\label{eqn:Aa-1 Gevrey remainder}
			|(A_{\Gamma_0}a)(y,\ybar;h) - 1| \leq (1+N) C^{1+N} N!^{2s-1} h^N \leq \widetilde{C}^{1+N} N!^{2s-1} h^N, \quad y\in U.
		\end{equation}
		The remainder in \eqref{eqn:Aa Gevrey} then follows by optimizing the upper bound in \eqref{eqn:Aa-1 Gevrey remainder}, setting $N = [ (\widetilde{C}h)^{-\frac{1}{2s-1}} ]$ in view of Lemma \ref{lem:Gevrey minimize}.
	\end{proof}

	\section{Approximate reproducing property and proof of Theorem \ref{main thm}}\label{Sec proof of the main theorem}
	Let $V\Subset\Omega$ be a small open neighborhood of $x_0\in\Omega$ with smooth boundary. Let $\Psi\in\GG_{\rm b}^s (\neigh((x_0,\overline{x_0}),\CC^{2n}))$ be an almost holomorphic extension of the weight function $\Phi$ satisfying \eqref{Psi as polarization Gevrey} and \eqref{Psi almost holomorphic Gevrey}. We assume that $V$ is small enough so that $\Psi$, as well as the amplitude $a$ introduced in Theorem \ref{thm:asymp invert A}, are defined in a neighborhood of the closure of $V\times \rho(V)$, with $\rho(x)=\xbar$ denoting the complex conjugation.
	
	We set, for $u\in L_\Phi^2(V) = L^2(V,e^{-2\Phi(x)/h} L(dx))$,
	\begin{equation}\label{local Bergman kernel eq}
		\widetilde{\Pi}_{V} u(x) = \frac{1}{h^n} \int_{V} e^{\frac{2}{h}\Psi(x,\overline{y})} a(x,\overline{y};h) u(y) e^{-\frac{2}{h}\Phi(y)}\, L(dy).
	\end{equation}
	In view of the estimate \eqref{Psi basic estimate} together with the Schur test we have
	\begin{equation}
		\label{Pi_V is O(1)}
		\widetilde{\Pi}_V=\OO(1): L^2_{\Phi}(V)\rightarrow L^2_{\Phi}(V).
	\end{equation}
	Letting $u\in H_\Phi(V)$, we can rewrite \eqref{local Bergman kernel eq} into the polarized expression
	\begin{equation*}
		\widetilde{\Pi}_{V} u(x) = \frac{C_n}{h^n} \iint_{\Gamma_V} e^{\frac{2}{h}(\Psi(x,\widetilde{y}) - \Psi(y,\widetilde{y}))} a(x,\widetilde{y};h) u(y)\, dy d\widetilde{y},
	\end{equation*}
	with $C_n = (i/2)^n$ as $L(dx)=C_n dx d\xbar$. Here the contour $\Gamma_V$ is given by
	\begin{equation}
		\label{eqn:Gamma_V anti diagonal contour}
		\Gamma_V=\{ (y, \widetilde{y})\in V\times \rho(V):  \widetilde{y}=\oo{y}\}.
	\end{equation}
	
	\subsection{Weak reproducing property}
	
	In this section we aim to prove an analogue of \cite[Theorem 3.1]{hitrik2022smooth}, namely the approximate reproducing property of $\widetilde{\Pi}_V$ in the weak formulation, with an improved error bound of the same form as in \eqref{eqn:Aa Gevrey}.
	\begin{thm}
		\label{thm:weak reproducing}
		There exists a small open neighborhood $W \Subset V$ of $x_0$ with $C^{\infty}$ boundary such that for every $\Phi_1 \in C(\Omega; \mathbb{R})$,
		$\Phi_1 \leq \Phi$, with $\Phi_1 < \Phi$ on $\Omega \setminus \overline{W}$, and there exists $C$ such that for all $u\in H_{\Phi}(V)$, $v\in H_{\Phi_1}(V)$, we have
		\begin{equation}
			\label{eqn:weak reproducing Gevrey}
			{(\widetilde{\Pi}_V u,v)_{L^2_{\Phi}(V)}= (u,v)_{H_{\Phi}(V)}}+ \OO(1)\exp\bigl(-C^{-1}h^{-\frac{1}{2s-1}}\bigr)\,  \|u\|_{H_{\Phi}(V)}\, \|v\|_{H_{\Phi_1}(V)}.
		\end{equation}
	\end{thm}
	We shall follow \cite[Section 3]{hitrik2022smooth} to prove Theorem \ref{thm:weak reproducing}, where the key ingredient is a contour deformation argument. We would like to highlight that the improved error term in \eqref{eqn:Aa Gevrey}, and the more quantitative $\overline{\partial}$--estimates \eqref{Psi almost holomorphic Gevrey} and \eqref{amplitude dbar estimate} for $\Psi$ and $a$, allow us to get the improved error bound in \eqref{eqn:weak reproducing Gevrey}, compared with the smooth case treated in \cite{hitrik2022smooth}.

	Let $W\Subset V_1\Subset V_2\Subset V$ be open neighborhoods of $x_0$ with smooth boundaries. It has been proved in \cite[Section 4]{deleporte2022analytic} that for $u\in H_\Phi(V)$, $v\in H_{\Phi_1}(V)$,
	\begin{equation}\label{eqn:weak reproducing rewrite inner product}
		(\widetilde{\Pi}_V u,v)_{L^2_{\Phi}(V)} = \int_{V_1} \widetilde{\Pi}_{V_2} u(x) \overline{v(x)} e^{-2\Phi(x)/h} L(dx) + \OO(1)e^{-\frac{1}{Ch}} \|u\|_{H_\Phi(V)} \|v\|_{H_{\Phi_1}(V)} .
	\end{equation}
	Following \cite[Section 4]{deleporte2022analytic}, we work with a \emph{coherent states decomposition} of $v\in H_\Phi (V)$. We recall from \cite[Proposition 2.2]{deleporte2022analytic} that there exists $\eta>0$ such that for any $v\in H_{\Phi}(V)$ and $x\in V_1\Subset V$, 
	\begin{equation*}
		v(x) = \int_V v_z(x) dz\,d\oo{z} + \OO(1) \|v\|_{H_{\Phi}(V)} e^{\frac{1}{h}(\Phi(x) - \eta)}, \quad x\in V_1.
	\end{equation*}
	Here
	\begin{equation}
		\label{v_z def eq}
		v_z(x) = \frac{1}{(2\pi h)^n} e^{\frac{i}{h}(x-z)\cdot \theta(x,z)} v(z)\chi(z) \det(\partial_{\oo{z}}\theta(x,z)) \in \mathrm{Hol}(V),
	\end{equation}
	and $\theta(x,z)$ depends holomorphically on $x\in V$ with
	\begin{equation}
		\label{v_z localization}
		-\Im\left((x-z)\cdot \theta(x,z)\right) + \Phi(z) \leq \Phi(x) - \delta |x-z|^2, \quad x,z\in V,
	\end{equation}
	for some $\delta >0$. The cut-off function $\chi\in C^{\infty}_0(V;[0,1])$ satisfies $\chi = 1$ in $V_2$. 
	
	\noindent
	It has been shown in \cite[Section 4]{deleporte2022analytic} that for $v\in H_{\Phi_1}(V)$ we have, with $W\Subset W_1\Subset V_1$,
	\begin{equation}
		\label{v v_z decomposition reduced}
		v(x) = \int_{W_1} v_z(x) dz\,d\oo{z} + \OO(1) e^{-\frac{1}{Ch}} \|v\|_{H_{\Phi_1}(V)} e^{\frac{1}{h}\Phi(x)}, \quad x\in V_1.
	\end{equation}
	As a result of \eqref{eqn:weak reproducing rewrite inner product}, \eqref{v v_z decomposition reduced}, and \eqref{Pi_V is O(1)}, we get for $u\in H_\Phi(V)$, $v\in H_{\Phi_1}(V)$,
	\begin{equation}\label{eqn:weak reproducing rewrite inner product 2}
		(\widetilde{\Pi}_V u, v)_{L^2_{\Phi}(V)} = \int_{W_1} (\widetilde{\Pi}_{V_2} u , v_z)_{L_\Phi^2(V_1)} \,dz d\zbar + \OO(1)e^{-\frac{1}{Ch}} \|u\|_{H_\Phi(V)} \|v\|_{H_{\Phi_1}(V)}.
	\end{equation}
	
	
	The next proposition is from \cite[Proposition 3.2]{hitrik2022smooth} (cf. \cite[Proposition 4.2]{deleporte2022analytic}), which plays a crucial role in the proof of Theorem \ref{thm:weak reproducing}.
	\begin{prop}\label{prop:G_z}
		Let $\delta >0$ be small and let us set for $z\in V$, $(x,\widetilde{x},y,\widetilde{y}) \in V\times \rho(V)\times V \times \rho(V) \subset \mathbb{C}^{4n}$,
		\begin{align}
			\label{G_z defn}
			G_z(x,\widetilde{x},y,\widetilde{y}) =& 2\Real \Psi(x,\widetilde{y}) - 2\Real \Psi(y,\widetilde{y}) + \Phi(y) + F_{z}(\widetilde{x}) - 2\Real \Psi(x,\widetilde{x}) \\
			=& 2\Real \varphi(y,\widetilde{x}; x,\widetilde{y}) - 2\Real \Psi(y,\widetilde{x}) + \Phi(y) + F_{z}(\widetilde{x}),
		\end{align}
		where
		\begin{equation}
			\label{eqn:F_z defn}
			F_{z}(\widetilde{x}) = \Phi(\overline{\widetilde{x}}) - \delta |\widetilde{x} - \overline{z}|^2.
		\end{equation}
		Then the $C^{\infty}$ function $G_{z}$ has a non-degenerate critical point at $(z,\oo{z},z,\oo{z})$ of signature $(4n,4n)$, with the critical value $0$. The following submanifolds of $\mathbb{C}^{4n}$ are good contours for $G_{z}$ in a neighborhood of $(z,\oo{z},z,\oo{z})$, i.e., they are both of real dimension $4n$, pass through the critical point, and are such that the Hessian of $G_z$ along the contours is negative definite:
		\begin{enumerate}
			\item The product contour
			\begin{equation*}
				\label{Gamma_1 product contour}
				\Gamma_{V}\times \Gamma_{V}=\{(x,\widetilde{x},y,\widetilde{y});\,\widetilde{x} = \overline{x},\,\,\widetilde{y} = \overline{y},\,\, x\in V,\,\,y\in V\}.
			\end{equation*}
			\item The composed contour
			\begin{equation}
				\label{Gamma_2 composed contour}
				\{(x,\widetilde{x},y,\widetilde{y});\, (y,\widetilde{x}) \in \Gamma_{V},\, (x,\widetilde{y}) \in \Gamma(y, \widetilde{x})\}.
			\end{equation}
			Here $\Gamma(y,\widetilde{x}) \subset \mathbb{C}^{2n}_{x,\widetilde{y}}$ is a good contour for the $C^{\infty}$ function $(x,\widetilde{y}) \mapsto {\rm Re}\,\varphi(y,\widetilde{x};x,\widetilde{y})$ described in \eqref{good contour C 2n}.
		\end{enumerate}
	\end{prop}
	
	Following \cite{hitrik2022smooth}, we will perform the contour deformation argument for the inner product $(\widetilde{\Pi}_{V_2}u,v_z)$, in view of \eqref{eqn:weak reproducing rewrite inner product 2}.
	\begin{prop}\label{prop:weak reproducing for v_z}
		There exists an open neighborhood $W_1\Subset V_1$ of $x_0$ and a constant $C>0$ such that uniformly for $z\in W_1$, we have,
		\begin{equation}\label{weak reproducing property eq}
			(\widetilde{\Pi}_{V_2} u, v_z)_{L^2_{\Phi}(V_1)} = (u,v_z)_{H_{\Phi}(V_1)} + \OO(1)\exp\bigl(-C^{-1}h^{-\frac{1}{2s-1}}\bigr) \|u\|_{H_{\Phi}(V)} |v(z)| e^{-\Phi(z)/h},
		\end{equation}
		where $v_z$ is given in \eqref{v_z def eq}.
	\end{prop}
	
	\begin{proof}
		Rewriting the Lebesgue measure on $\CC^n$ as $L(dx) = C_n dx d\xbar$, where $C_n=(i/2)^n$ (see the discussion after \eqref{eqn:f(y,z)} for the orientation), we express the scalar product in the space $H_\Phi(V_1)$ in the polarized form
		\begin{equation}\label{eqn:polarized inner product}
			(f , g)_{H_\Phi(V_1)} = C_n \int_{\Gamma_{V_1}} f(x) g^* (\widetilde{x}) e^{-\frac{2}{h}\Psi(x,\widetilde{x})} \,dx d\xbar.
		\end{equation}
		Here the contour $\Gamma_{V_1}$ is defined as in \eqref{eqn:Gamma_V anti diagonal contour}, and $g^*(\widetilde{x}) := \overline{g(\overline{\widetilde{x}})} \in H_{\widehat{\Phi}}(V_1)$, $\widehat{\Phi}(\widetilde{x}) = \Phi(\overline{\widetilde{x}})$ provided $g\in H_\Phi(V_1)$. We can therefore write, in view of \eqref{local Bergman kernel eq},
		\begin{equation}
			\label{polarized form weak reproducing}
			\begin{split}
				(\widetilde{\Pi}_{V_2} u, v_z)_{L^2_{\Phi}(V_1)} = &\frac{C_n^2}{h^n} \int_{\Gamma_{V_1}}\biggl( \int_{\Gamma_{V_2}} e^{\frac{2}{h}(\Psi(x,\widetilde{y})-\Psi(y,\widetilde{y}))} a(x,\widetilde{y};h) u(y) dy d\widetilde{y} \biggr)  v_z^* (\widetilde{x}) e^{-\frac{2}{h}\Psi(x,\widetilde{x})} dx d\widetilde{x} \\
				= &\iint_{\Gamma_{V_1}\times\Gamma_{V_2}} \omega,
			\end{split}
		\end{equation}
		where $\omega$ is a $(4n,0)$--differential form on $\CC^{4n}$ as follows
		\[
		\omega = f(x,\widetilde{x},y,\widetilde{y})\, dx \wedge d\widetilde{x} \wedge dy \wedge d\widetilde{y},
		\]
		with $f\in \GG^s(V\times \rho(V)\times V\times \rho(V))$ given by 
		\begin{equation}
			\label{eqn:f(x,xtilde,y,ytilde)}
			f(x,\widetilde{x},y,\widetilde{y}) = \frac{C_n^2}{h^n} e^{\frac{2}{h}(\Psi(x,\widetilde{y})-\Psi(y,\widetilde{y})-\Psi(x,\widetilde{x}))} a(x,\widetilde{y};h) u(y) v_z^* (\widetilde{x}) .
		\end{equation}
		In view of \eqref{v_z def eq} and \eqref{v_z localization} we have
		\[
		\left|v_z^* (\widetilde{x})\right| \leq \frac{\OO(1)}{h^n} |v(z)| e^{-\Phi(z)/h} e^{F_z(\widetilde{x})/h} ,\quad \widetilde{x}\in \rho(V_1),
		\] 
		with $F_z$ defined in \eqref{eqn:F_z defn} being strictly plurisubharmonic on $\rho(V_1)$. It then follows from \eqref{G_z defn}, using \cite[Proposition 2.3]{deleporte2022analytic}, that
		\begin{equation}
			\label{estimate:f(x,xtilde,y,ytilde)}
			\left|e^{\frac{2}{h}(\Psi(x,\widetilde{y})-\Psi(y,\widetilde{y})-\Psi(x,\widetilde{x}))} u(y) v_z^* (\widetilde{x})\right| \leq \frac{\OO(1)}{h^{2n}} \|u\|_{H_\Phi(V)} |v(z)|e^{-\Phi(z)/h} e^{G_z(x,\widetilde{x},y,\widetilde{y})/h}.
		\end{equation} 
		Proposition \ref{prop:G_z} shows that the product contour $\Gamma_1 := \Gamma_{V_1}\times\Gamma_{V_2}$ and the composed contour $\Gamma_2$ given in \eqref{Gamma_2 composed contour} are both good contours for $G_z$, for all $z$ in a small neighborhood of $x_0$. Arguing as in the proof of Proposition \ref{prop: independent of contour Gevrey} we see that there exists a $\CI$ homotopy between contours $\Gamma_1$ and $\Gamma_2$, and an associated union of the ``intermediate" good contours $\Sigma\subset\CC^{4n}$. We notice first that, for some $C>0$ we have
		\begin{equation}
			\label{sides contour Sigma}
			\partial\Sigma - (\Gamma_1 - \Gamma_2) \subset \{ (x,\widetilde{x},y,\widetilde{y}) : G_z(x,\widetilde{x},y,\widetilde{y}) \leq -1/C \},\quad z\in \neigh(x_0,\CC^n).
		\end{equation}
		Applying Stokes' formula (see the paragraph after \eqref{Stoke thm eq}), we therefore obtain for all $z$ in a small neighborhood of $x_0$, using \eqref{estimate:f(x,xtilde,y,ytilde)} and \eqref{sides contour Sigma}, that
		\begin{equation}
			\label{eqn:Stokes}
			\begin{split}
				\iint_{\Gamma_1} \omega - \iint_{\Gamma_2} \omega = &\iiint_\Sigma \overline{\partial}f \wedge dx \wedge d\widetilde{x} \wedge dy \wedge d\widetilde{y} \\
				&+ \OO(e^{-1/Ch}) \|u\|_{H_\Phi(V)} |v(z)|e^{-\Phi(z)/h} .
			\end{split}
		\end{equation}
		We shall proceed to estimate $\overline{\partial} f$, which is distinct from \cite[Proposition 3.3]{hitrik2022smooth} in view of \eqref{amplitude dbar estimate} and \eqref{Psi almost holomorphic Gevrey}. Using \eqref{eqn:f(x,xtilde,y,ytilde)} we compute
		\[
		\partial_{\xbar} f = \frac{C_n^2}{h^n} e^{\frac{2}{h}(\Psi(x,\widetilde{y})-\Psi(y,\widetilde{y})-\Psi(x,\widetilde{x}))}  u(y) v_z^* (\widetilde{x}) \left(\frac{2}{h} (\partial_{\xbar} \Psi(x,\widetilde{y}) - \partial_{\xbar} \Psi(x,\widetilde{x})) a + \partial_{\xbar} a \right)
		\]
		We then conclude from \eqref{estimate:f(x,xtilde,y,ytilde)}, Proposition \ref{prop:G_z}, and $\overline{\partial}$--estimates \eqref{amplitude dbar estimate}, \eqref{Psi almost holomorphic Gevrey} that
		\[
		\begin{split}
			|\partial_{\xbar} f| \leq &\frac{\OO(1)}{h^{3n+1}} \|u\|_{H_\Phi(V)} |v(z)|e^{-\Phi(z)/h} e^{-\dist((x,\widetilde{x},y,\widetilde{y}),(z,\zbar,z,\zbar))^2/Ch} \\ &\cdot\big(e^{-|\xbar-\widetilde{y}|^{-\frac{1}{s-1}}/C} + e^{-|\xbar-\widetilde{x}|^{-\frac{1}{s-1}}/C}\big)
		\end{split}
		\]
		Noting that we can estimate
		\[
		\begin{gathered}
			|\xbar - \widetilde{y}|\leq |\xbar - \zbar| + |\widetilde{y}-\zbar| = |x-z| + |\widetilde{y}-\zbar| \leq \sqrt{2}\dist((x,\widetilde{x},y,\widetilde{y}),(z,\zbar,z,\zbar)), \\
			|\xbar - \widetilde{x}|\leq |\xbar - \zbar| + |\widetilde{x}-\zbar| = |x-z| + |\widetilde{x}-\zbar| \leq \sqrt{2}\dist((x,\widetilde{x},y,\widetilde{y}),(z,\zbar,z,\zbar)).
		\end{gathered}
		\]
		It then follows by adjusting constant $C>0$ that
		\begin{equation}
			\label{dbar estimate f(x,xtilde,y,ytilde)}
			\begin{gathered}
				|\partial_{\xbar} f| \leq \frac{\OO(1)}{h^{3n+1}} \|u\|_{H_\Phi(V)} |v(z)|e^{-\Phi(z)/h} \exp\left(-\frac{1}{Ch}\big(t^2 + h t^{-\frac{1}{s-1}}\big)\right), \\
				\textrm{with}\quad t = t(x,\widetilde{x},y,\widetilde{y};z) = \dist((x,\widetilde{x},y,\widetilde{y}),(z,\zbar,z,\zbar)) > 0,
			\end{gathered}
		\end{equation} 
		and we note that $\partial_{\xbar} f = 0$ at $(z,\zbar,z,\zbar)$. We can therefore conclude from \eqref{dbar estimate f(x,xtilde,y,ytilde)} and \eqref{the minimum}, with a larger $C>0$, that
		\begin{equation}
			\label{dbar estimate f(x,xtilde,y,ytilde) final}
			|\partial_{\xbar} f| \leq \OO(1)\exp\left(-C^{-1} h^{-\frac{1}{2s-1}}\right) \|u\|_{H_\Phi(V)} |v(z)|e^{-\Phi(z)/h} .
		\end{equation}
		The bound \eqref{dbar estimate f(x,xtilde,y,ytilde) final} also holds for $\partial_{\overline{\widetilde{x}}} f$, $\partial_{\ybar} f$ and $\partial_{\overline{\widetilde{y}}} f$, by similar computations and estimates, and noticing that $u$ and $v_z^*$ are both holomorphic. We conclude from \eqref{eqn:Stokes} that there exists an open neighborhood $W_1\Subset V_1$ of $x_0$ and a constant $C>0$ such that uniformly for $z\in W_1$, we have
		\begin{equation}
			\label{Gamma_1 - Gamma_2 desired remainder}
			\iint_{\Gamma_1} \omega = \iint_{\Gamma_2} \omega + \OO(1)\exp\left(-C^{-1} h^{-\frac{1}{2s-1}}\right) \|u\|_{H_\Phi(V)} |v(z)|e^{-\Phi(z)/h}.
		\end{equation}
		Let us write in view of \eqref{Gamma_2 composed contour},
		\begin{equation}
			\label{integration along Gamma_2}
			\begin{split}
				\iint_{\Gamma_2} \omega = &\,C_n \int_{\Gamma_{V_1}} \biggl(\frac{1}{(2ih)^n}\int_{\Gamma(y,\widetilde{x})\cap (V_1\times\rho(V_1))} e^{\frac{2}{h}\phi(y,\widetilde{x};x,\widetilde{y})} a(x,\widetilde{y};h) \,dx d\widetilde{y} \biggr) u(y) v_z^*(\widetilde{x}) e^{-\frac{2}{h}\Psi(y,\widetilde{x})} dy d\widetilde{x} \\
				& + \OO(1)\exp\left(-C^{-1} h^{-\frac{1}{2s-1}}\right) \|u\|_{H_\Phi(V)} |v(z)|e^{-\Phi(z)/h} ,
			\end{split}
		\end{equation}
		where we recalled $C_n = (i/2)^n$ and noticed that $dx d\widetilde{x} dy d\widetilde{y} = (-1)^n dx d\widetilde{y} dy d\widetilde{x}$. We observe in \eqref{integration along Gamma_2} that in the outer contour of integration $(y,\widetilde{x})\in \Gamma_{V_1} \iff \widetilde{x}=\ybar$, applying Theorem \ref{thm:asymp invert A} we obtain therefore that the inner integral there is equal to $1+\OO(1)\exp\left(-C^{-1} h^{-\frac{1}{2s-1}}\right)$, provided that $V_1$ is small enough. It follows that \eqref{integration along Gamma_2} reduces to, in view of \eqref{eqn:polarized inner product}, 
		\[
		\iint_{\Gamma_2} \omega = (u,v_z)_{H_\Phi(V_1)} + \OO(1)\exp\left(-C^{-1} h^{-\frac{1}{2s-1}}\right) \|u\|_{H_\Phi(V)} |v(z)|e^{-\Phi(z)/h}.
		\]
		Combining this with \eqref{Gamma_1 - Gamma_2 desired remainder} and \eqref{polarized form weak reproducing}, we get \eqref{weak reproducing property eq}.
	\end{proof}
	
	Let us now finish the proof of Theorem \ref{thm:weak reproducing} following the argument in \cite[Section 4]{deleporte2022analytic}. Setting $W\Subset W_1\Subset V$ with $W_1$ as in Proposition \ref{prop:weak reproducing for v_z}, we combine \eqref{eqn:weak reproducing rewrite inner product 2} and \eqref{weak reproducing property eq} to obtain 
	\[
	(\widetilde{\Pi}_V u, v)_{L^2_{\Phi}(V)} = \int_{W_1} (u , v_z)_{L_\Phi^2(V_1)} \,dz d\zbar + \OO(1)\exp\left(-C^{-1} h^{-\frac{1}{2s-1}}\right) \|u\|_{H_\Phi(V)} \|v\|_{H_{\Phi_1}(V)}.
	\]
	In view of \eqref{v v_z decomposition reduced} we can also write
	\[
	(u, v)_{H_\Phi(V)} = \int_{W_1} (u , v_z)_{L_\Phi^2(V_1)} \,dz d\zbar + \OO(1) e^{-\frac{1}{Ch}} \|u\|_{H_\Phi(V)} \|v\|_{H_{\Phi_1}(V)}.
	\]
	We therefore complete the proof of Theorem \ref{thm:weak reproducing} by comparing these two equations.
	
	\subsection{Completing the proof of Theorem \ref{main thm}}
	Let $\chi_1\in C^{\infty}(\Omega; [0,\infty))$ be such that $\chi_1>0$ on $\Omega\setminus \oo{W}$, and set
	\begin{equation*}
		\Phi_1(x)=\Phi(x)-\delta\chi_1(x).
	\end{equation*}
	Here $\delta>0$ is sufficiently small, so that $\Phi_1$ is strictly plurisubharmonic on $V$. Our first goal is to pass from the reproducing property in weak formulation \eqref{eqn:weak reproducing Gevrey} to a weighted $L^2$ norm estimates:  
	\begin{equation*}
		\|(\widetilde{\Pi}_V-1)u \|_{L^2_{\Phi}(V)}=\OO(1) \exp\bigl(C^{-1}h^{-\frac{1}{2s-1}}\bigr) \|u\|_{H_{\Phi_1}(V)}, \quad\mbox{ for any } u\in H_{\Phi_1}(V).
	\end{equation*}
	Since the projection $\widetilde{\Pi}_V$ may not preserve the weighted space $L^2_{\Phi_1}(V)$, we shall introduce another weight $\Phi_2$, such that
	\begin{equation}\label{boundedness of Pi eq}
		\widetilde{\Pi}_V=\OO(1): L^2_{\Phi_1}(V)\rightarrow L^2_{\Phi_2}(V).
	\end{equation}
	This weight function $\Phi_2$ is defined as $\Phi_2=\Phi-\chi_2$, and $\chi_2\in C^{\infty}(\Omega; [0,\infty))$ is given by
	\begin{equation}\label{infimal convolution eq}
		\chi_2(x)=\inf_{y\in V} \Bigl(\frac{|x-y|^2}{2C}+\delta\chi_1(y) \Bigr)
	\end{equation}
	where $C>0$ is some sufficiently large constant. It follows readily that
	\begin{equation*}
		\Phi_2 \leq \Phi  \mbox{ in } \Omega, \quad \Phi_2 < \Phi \mbox{ in } \Omega\setminus \oo{W}
	\end{equation*}
	and
	\begin{equation}\label{two weights comparison eq}
		\Phi_1 \leq \Phi_2 \mbox{ in } V.
	\end{equation}
	In addition, the boundedness in \eqref{boundedness of Pi eq} can be deduced from \eqref{infimal convolution eq} and the Schur test. As proved in \cite[Section 4]{hitrik2022smooth}, the weight function $\Phi_2$ is still plurisubharmonic on $V$ when $\delta>0$ is sufficiently small. Without loss of generality, we shall assume $V$ is pseudoconvex in what follows.

	Let
	\begin{equation*}
		\Pi_{\Phi_2}: L^2_{\Phi_2}(V) \rightarrow H_{\Phi_2}(V)
	\end{equation*}
	be the orthogonal projection. We introduce the following proposition for later use.
	\begin{prop}\label{Pi_V almost holomorphic Gevrey prop}
		We have
		\begin{equation*}
			\Pi_{\Phi_2} \widetilde{\Pi}_V - \widetilde{\Pi}_V = \OO(1)\exp\bigl(C^{-1}h^{-\frac{1}{2s-1}}\bigr): L^2_{\Phi_1}(V) \rightarrow L^2_{\Phi_2}(V).
		\end{equation*}	
	\end{prop}
	
	\begin{proof}
		Given $f\in L^2_{\Phi_1}(V)$, by applying H\"ormander's $L^2$ estimates for the pseudoconvex open set $V$ and the strictly plurisubharmonic weight $\Phi_2$ (\cite[Proposition 4.2.5]{hormander1994convexity}), we obtain 
		\begin{equation*}
			\|(\Pi_{\Phi_2}-1)\widetilde{\Pi}_Vf\|_{L^2_{\Phi_2}(V)}= \OO(h^{1/2}) \|\oo{\partial}\,\widetilde{\Pi}_Vf\|_{L^2_{\Phi_2}(V)}.
		\end{equation*} 
		Thus, it is sufficient to prove 
		\begin{equation}\label{Pi projection is almost holomorphic eq}
			\oo{\partial}\circ \widetilde{\Pi}_V = \OO(1)\exp\bigl(C^{-1}h^{-\frac{1}{2s-1}}\bigr): L^2_{\Phi_1}(V) \rightarrow L^2_{\Phi_2, (0,1)}(V),
		\end{equation}
		where $L^2_{\Phi_2, (0,1)}(V)$ is the space of $(0, 1)$ forms.
		
		Recalling the definition of $\widetilde{\Pi}_V$ in \eqref{local Bergman kernel eq}, we have 
		\begin{align*}
			\oo{\partial}\,\widetilde{\Pi}_Vf(x)=\frac{1}{h^n} \int_{V} \oo{\partial}_x \left(e^{\frac{2}{h}\left(\Psi(x,\overline{y})-\Phi(y)\right)} a(x,\overline{y};h)\right) f(y)\, L(dy).
		\end{align*}
		In view of the Schur test, to establish \eqref{Pi projection is almost holomorphic eq} it suffices to show
		\begin{align}\label{Schur test estimate eq}
			\sup_{x, y\in V}\left|\partial_{\oo{x}}\left(e^{\frac{2}{h}\left(\Psi(x,\overline{y})-\Phi(y)\right)} a(x,\overline{y};h)\right) e^{-\frac{1}{h}\Phi_2(x)+\frac{1}{h}\Phi_1(y)}\right|=\OO(1)\exp\bigl(C^{-1}h^{-\frac{1}{2s-1}}\bigr).
		\end{align}
		A straightforward computation gives
		\begin{align*}
			\Bigl|\partial_{\oo{x}}\Bigl(&e^{\frac{2}{h}(\Psi(x,\overline{y})-\Phi(y))} a(x,\overline{y};h)\Bigr) e^{-\frac{1}{h}\Phi_2(x)+\frac{1}{h}\Phi_1(y)}\Bigr|\\
			&=e^{\frac{1}{h}\left(2\Re\Psi(x,\oo{y})-\Phi(x)-\Phi(y)+\chi_2(x)-\delta\chi_1(y)\right)}\Bigl|\frac{2}{h}\partial_{\oo{x}}\Psi(x, \oo{y})a(x, \oo{y}; h)+\partial_{\oo{x}}a(x, \oo{y}; h) \Bigr| 
		\end{align*}
		Note by \eqref{good contour C 2n} and \eqref{infimal convolution eq} we get
		\begin{equation*}
			2\Re\Psi(x,\oo{y})-\Phi(x)-\Phi(y)+\chi_2(x)-\delta\chi_1(y)\leq -\frac{|x-y|^2}{2C}.
		\end{equation*}
		Meanwhile, by \eqref{amplitude dbar estimate} and \eqref{Psi almost holomorphic Gevrey} we have
		\begin{equation*}
			\Bigl|\frac{2}{h}\partial_{\oo{x}}\Psi(x, \oo{y})a(x, \oo{y}; h)+\partial_{\oo{x}}a(x, \oo{y}; h) \Bigr|=\frac{\OO(1)}{h} \exp\big(-C^{-1}|x-y|^{-\frac{1}{s-1}}\big).
		\end{equation*}
		Therefore, by setting $t=|x-y|$ we obtain
		\begin{align*}
			\Bigl|\partial_{\oo{x}}\Bigl(e^{\frac{2}{h}(\Psi(x,\overline{y})-\Phi(y))} a(x,\overline{y};h)\Bigr) e^{-\frac{1}{h}\Phi_2(x)+\frac{1}{h}\Phi_1(y)}\Bigr|
			=\frac{\OO(1)}{h}\exp\Bigl(-\frac{1}{2Ch} \bigl(t^2+ht^{-\frac{1}{s-1}}\bigr) \Bigr).
		\end{align*}
		In view of \eqref{the minimum}, we conclude \eqref{Schur test estimate eq} with a larger $C>0$, completing the proof.
	\end{proof}
	
	With all there preparations, we are now ready to prove \eqref{reproducing property eq}. We shall first prove a reproducing property with the auxiliary weight $\Phi_1$, that is, 
	\begin{equation}\label{reproducing property with auxiliary weight eq}
		\|(\widetilde{\Pi}_V-1)u \|_{L^2_{\Phi}(V)}=\OO(1)\exp\bigl(C^{-1}h^{-\frac{1}{2s-1}}\bigr) \|u\|_{H_{\Phi_1}(V)}, \quad \mbox{ for any } u\in H_{\Phi_1}(V).
	\end{equation}
	
	To this end, we apply Theorem \ref{thm:weak reproducing} with $u\in H_{\Phi_1}(V)\subset H_{\Phi}(V)$ and 
	\begin{equation*}
		v=\Pi_{\Phi_2}(\widetilde{\Pi}_V-1)u=\Pi_{\Phi_2}\widetilde{\Pi}_Vu-u\in H_{\Phi_2}(V),
	\end{equation*}
	in addition with $\Phi_2$ in place of $\Phi_1$. Then in view of \eqref{boundedness of Pi eq} and \eqref{two weights comparison eq} we obtain
	\begin{align*}
		\bigl((\widetilde{\Pi}_V-1) u, \Pi_{\Phi_2}(\widetilde{\Pi}_V-1)u\bigr)_{L^2_{\Phi}(V)}=& \OO(1)\exp\bigl(-C^{-1}h^{-\frac{1}{2s-1}}\bigr)\,  \|u\|_{H_{\Phi}(V)} \|\Pi_{\Phi_2}(\widetilde{\Pi}_V-1)u\|_{H_{\Phi_2}(V)}\\
		=&\OO(1)\exp\bigl(-C^{-1}h^{-\frac{1}{2s-1}}\bigr)\,  \|u\|_{H_{\Phi_1}(V)}^2.
	\end{align*}
	Note
	\begin{equation*}
		\Pi_{\Phi_2}(\widetilde{\Pi}_V-1)u= (\widetilde{\Pi}_V-1) u+\bigl(\Pi_{\Phi_2} \widetilde{\Pi}_V - \widetilde{\Pi}_V\bigr)u. 
	\end{equation*}
	To achieve \eqref{reproducing property with auxiliary weight eq} it suffices to prove
	\begin{equation*}
		\bigl((\widetilde{\Pi}_V-1) u, (\Pi_{\Phi_2} \widetilde{\Pi}_V - \widetilde{\Pi}_V)u\bigr)_{L^2_{\Phi}(V)}=\OO(1)\exp\bigl(-C^{-1}h^{-\frac{1}{2s-1}}\bigr)\,  \|u\|_{H_{\Phi_1}(V)}^2,
	\end{equation*}
	which follows directly from \eqref{Pi_V is O(1)} and Proposition \ref{Pi_V almost holomorphic Gevrey prop}. Thus \eqref{reproducing property with auxiliary weight eq} is established.
	
	Note the bound \eqref{reproducing property with auxiliary weight eq} is the counterpart statement to the estimate (5.5) in \cite{deleporte2022analytic}. We may therefore repeat the same argument as in \cite[Section 5]{deleporte2022analytic}, using H\"ormander's $L^2$ estimates, to conclude the proof of Theorem \ref{main thm}. Letting $U\Subset W \Subset V$ be neighborhoods of $x_0$ with smooth boundaries, we obtain
	\begin{equation*}
		\|(\widetilde{\Pi}_V-1)u \|_{L^2_{\Phi}(U)}=\OO(1)\exp\bigl(-C^{-1}h^{-\frac{1}{2s-1}}\bigr) \|u\|_{H_{\Phi}(V)}, \quad \mbox{ for any } u\in H_{\Phi}(V).
	\end{equation*} 
	So the proof of Theorem \ref{main thm} is completed.
	
	\appendix
	
	\section{Some basic results about factorials and binomials}\label{Appendix: combinatorial facts}
	In this section we review some basic results concerning factorials and binomials. 
	\begin{lem}
		\label{lem:multiply factorials}
		Suppose that $k_1,\ell_1, k_2,\ell_2 \in \NN$ satisfy $k_1 + \ell_1 = k_2 + \ell_2$, then
		\[
		|k_1 - \ell_1| > |k_2 - \ell_2| \implies k_1!\ell_1! > k_2!\ell_2! 
		\]
	\end{lem}
	
	\begin{proof}
		We may assume that $k_1\geq \ell_1$, $k_2\geq \ell_2$. If $|k_1 - \ell_1| > |k_2 - \ell_2|$, then $k_1 > k_2 \geq \ell_2 > \ell_1$ due to the assumption $k_1 + \ell_1 = k_2 + \ell_2$. Let $d = k_1 - k_2 = \ell_2 - \ell_1 > 0$, then 
		\[
		\frac{k_1!}{k_2!} = k_1 \cdots (k_1-d+1) > \ell_2 \cdots (\ell_2 -d+1) = \frac{\ell_2!}{\ell_1!}, 
		\]
		which proves our desired result.
	\end{proof}
	
	\begin{corr}
		For any multi-index $\alpha = (\alpha^1,\cdots,\alpha^n)\in\NN^n$, if we denote $|\alpha|=\alpha^1 + \cdots + \alpha^n$, then 		\[
		\alpha! = \alpha^1!\cdots\alpha^n! \leq |\alpha|! .
		\]
	\end{corr}
	
	\noindent
	\emph{Proof.} We can apply Lemma \ref{lem:multiply factorials} successively to see that 
	\[
	\qquad\;\alpha^1!\alpha^2!\cdots\alpha^n! \leq 0!(\alpha^1+\alpha^2)!\cdots\alpha^n! \leq \cdots \leq \underbrace{0!\cdots 0!}_{n-1} (\alpha^1+\cdots\alpha^n)! = |\alpha|!. \qquad\,\, \qed
	\]
	
	\begin{corr}\label{cor:factorial product max}
		Suppose that $m_1,\ldots,m_k\in \NN$ satisfy $m_j\geq a$, $j=1,\ldots,k$, for some $a\in\NN$. Writing $M=m_1+\cdots m_k$, then we have
		\[
		m_1!\cdots m_k! \leq a!^{k-1} (M - a(k-1))!
		\]
	\end{corr}
	
	\begin{proof}
		We may assume that $a\leq m_1\leq \cdots\leq m_k$. Applying Lemma \ref{lem:multiply factorials} successively, we get $m_1!m_k!\leq a!(m_1+m_k-a)!$, $m_2!(m_1+m_k-a)!\leq a!(m_1+m_2+m_k-2a)!$, \ldots, $m_{k-1}!(m_1+\cdots+m_{k-2}+m_k-(k-2)a)! \leq a!(M-(k-1)a)!$. Therefore, $m_1!\cdots m_k! \leq \underbrace{a!\cdots a!}_{k-1} (M - a(k-1))! = a!^{k-1} (M - a(k-1))!$.
	\end{proof}
	
	\begin{prop}
		\label{prop:partition multi-index counting}
		For any multi-index $\alpha\in\NN^n$ and positive integer $m$, let
		\[
		P(\alpha,m) := \left\{ \{\alpha_j\}_{j=1}^m : \alpha_1 + \cdots + \alpha_m = \alpha,\ \ \alpha_j\in\NN^n,\ 1\leq j\leq m \right\}.
		\]
		Then 
		\[
		\# P(\alpha,m) = \binom{\alpha + (m-1)\mathbbm{1}}{\alpha} = \prod_{k=1}^n \binom{\alpha^k + m-1}{\alpha^k} .
		\]
		where $\mathbbm{1}=(1, \cdots, 1)$ and $\alpha = (\alpha^1,\cdots,\alpha^n)$. In particular, $\# P(\alpha,m)\leq 2^{|\alpha|+n(m-1)}$.
	\end{prop}
	
	\begin{proof}
		Writing $\alpha_j = (\alpha_j^1,\cdots,\alpha_j^n)\in\NN^n$, $1\leq j\leq m$, we have
		\[
		\{\alpha_j\}_{j=1}^m \in P(\alpha,m) \iff \alpha_1^k + \cdots + \alpha_m^k = \alpha^k,\quad 1\leq k\leq n.
		\]
		The desired result then follows from the basic counting identity: 
		\begin{equation}
			\label{partition of integer}
			\#\{ l_1,\ldots,l_m \in\NN : l_1 + \cdots + l_m = a\} = \binom{a +m-1}{a} .
		\end{equation}
	\end{proof}
	
	\begin{lem}
		Suppose that $N\in\NN$, $h>0$ and $C>0$, then we have
		\begin{equation}
			\label{basic inequality 1}
			\forall t>0,\quad e^{-\frac{C t^2}{2h}}t^{N}  \leq C^{-N/2} (N!)^{1/2} h^{N/2}
		\end{equation}
	\end{lem}
	
	\begin{proof}
		Let $f(t) = e^{-\frac{C t^2}{2h}}t^{N}$, then we compute
		\[
		f'(t) = \left( \frac{N}{t} - \frac{Ct}{h} \right) e^{-\frac{C t^2}{2h}}t^{N} = (N - Ch^{-1}t^2) e^{-\frac{C t^2}{2h}} t^{N-1}.
		\]
		It follows that
		\[
		t\lessgtr (Nh/C)^{1/2} \implies f'(t)\gtrless 0,
		\]
		thus $f(t)$ has a maximum on $t\in(0,\infty)$, which is attained at $t_c := (Nh/C)^{1/2}$, i.e.,
		\[
		f(t)\leq f(t_c) = C^{-N/2} (N/e)^{N/2} h^{N/2} \leq C^{-N/2} (N!)^{1/2} h^{N/2},
		\]
		where we used the fact that $N!\geq (N/e)^N$ as a result of the following version of Stirling's approximation:
		\begin{equation}
			\label{Stirling apprx}
			\sqrt{2\pi N} \left( \frac{N}{e} \right)^N e^{\frac{1}{12N + 1}} < N! < \sqrt{2\pi N} \left( \frac{N}{e} \right)^N e^{\frac{1}{12N}},
		\end{equation}
		which holds for all $N\geq 1$.
	\end{proof}
	
	\begin{lem}
		\label{lem:Gevrey minimize}
		Suppose that $C, h>0$, $\sigma\geq 1$, and let $a_k = C^{k} k!^{\sigma} h^k$, $k\in \mathbb{N}$, then
		\begin{equation}
			\label{seq a_k decrease increase}
			a_k \leq a_{k-1}\quad \text{if }k\leq [(Ch)^{-\frac{1}{\sigma}} ]; \quad a_k\geq a_{k-1} \quad \text{if }k > [ (Ch)^{-\frac{1}{\sigma}} ],
		\end{equation}
		where $[x]$ denotes the greatest integer less than or equal to $x$. Furthermore, there exists $h_0>0$ such that for all $0<h<h_0$,
		\begin{equation}
			\label{basic inequality 2}
			e^{-\sigma(Ch)^{-\frac{1}{\sigma}}}\leq \min_{k\in\NN} a_k \leq e^{2\sigma} e^{-\frac{3}{4} \sigma(Ch)^{-\frac{1}{\sigma}}}.
		\end{equation}
	\end{lem}
	
	\begin{proof}
		Let us justify \eqref{seq a_k decrease increase} directly, noting that $a_k/a_{k-1} = Ch k^{\sigma}$,
		\[
		k \leq [ (Ch)^{-\frac{1}{\sigma}} ] \implies k \leq (Ch)^{-\frac{1}{\sigma}} \iff a_k/a_{k-1} \leq 1; 
		\]
		\[
		k > [(Ch)^{-\frac{1}{\sigma}} ] \implies k > (Ch)^{-\frac{1}{\sigma}} \iff a_k/a_{k-1} > 1.
		\]
		It follows that the minimum  $\min_{k\in\NN} a_k$ is attained at $k = k_c := [ (Ch)^{-\frac{1}{\sigma}} ]$, for which we can compute, in view of \eqref{Stirling apprx},
		\[
		\begin{split}
			\log a_{k_c} &= k_c\log(Ch) + \sigma\log(k_c !) \\
			&> k_c\log(Ch) + \sigma\left( k_c\log k_c - k_c + \frac{1}{2}\log k_c \right) \\
			&= \sigma\left( k_c\log\big( (Ch)^{\frac{1}{\sigma}}k_c\big) - k_c + \frac{1}{2}\log k_c \right).
		\end{split}
		\]
		Recalling that $(Ch)^{-\frac{1}{\sigma}} < k_c + 1$, and using the inequality $\log(1+x)\leq x$ for all $x\geq 0$, we get
		\[
		\log\big( (Ch)^{\frac{1}{\sigma}}k_c\big) > \log\left( \frac{k_c}{k_c + 1}\right) = -\log (1+k_c^{-1}) \geq -k_c^{-1} .
		\]
		It then follows that, if $(Ch)^{-\frac{1}{\sigma}} \geq e^2 + 1$ (thus $k_c>e^2$),
		\[
		\log a_{k_c} > \sigma\left( -1 - k_c + \frac{1}{2}\log k_c \right) \geq -\sigma(Ch)^{-\frac{1}{\sigma}}.
		\]
		This gives the lower bound in \eqref{basic inequality 2} if we set $(Ch_0)^{-\frac{1}{\sigma}} = e^2 + 1$. It remains to justify the upper bound in \eqref{basic inequality 2}, again by \eqref{Stirling apprx} we have
		\[
		\begin{split}
			\log a_{k_c} &< \sigma\bigg( k_c\log\big( (Ch)^{\frac{1}{\sigma}}k_c\big) - k_c + \frac{1}{2}\log k_c + \log(\sqrt{2\pi}) + \frac{1}{12k_c} \bigg) \\
			&< \sigma \bigg( - (Ch)^{-\frac{1}{\sigma}} +\frac{1}{2}\log \big((Ch)^{-\frac{1}{\sigma}}\big) + 2\bigg),
		\end{split}
		\]
		where we used $(Ch)^{-\frac{1}{\sigma}}-1<k_c\leq (Ch)^{-\frac{1}{\sigma}}$. The upper bound in \eqref{basic inequality 2} then follows by recalling the basic inequality $\log x < x/2$, $x>0$.
	\end{proof}
	
	\section{From asymptotic to exact Bergman projections: proof of Corollary \ref{corr: asymp to exact Bergman}}\label{Appendix: prove corollary}
	
	This section is devoted to the proof of Corollary \ref{corr: asymp to exact Bergman}, showing that the operator $\widetilde{\Pi}_V$ defined in \eqref{Pi_V eq} with $V$ being a small neighborhood of $x_0\in\Omega$, enjoying the local reproducing property \eqref{reproducing property eq}, gives an approximation for the orthogonal projection $\Pi: L^2(\Omega,e^{-2\Phi/h}L(dx)) \to H_\Phi(\Omega)$, up to a $\GG^s$--type small error of the same form as in \eqref{reproducing property eq}, locally near $x_0\in\Omega$. The arguments presented below are essentially well-known and follow \cite{BBSj08} closely. See also \cite[Appendix A]{deleporte2022analytic} for the corresponding discussion in the case of real analytic weights, and \cite[Section 5]{hitrik2022smooth} for the case of smooth weights.
	
	Our starting point is to deduce pointwise estimates from the weighted $L^2$ estimates \eqref{reproducing property eq} in Theorem \ref{main thm}, using Proposition \ref{prop:L2 to pointwise}. We shall apply Proposition \ref{prop:L2 to pointwise} to a function of the form
	\begin{equation*}
		f = \widetilde{\Pi}_V u - u,\quad u\in H_\Phi(V),
	\end{equation*}
	satisfying (in view of \eqref{reproducing property eq})
	\begin{equation}\label{appendix L2 estimate f}
		\|f\|_{L_\Phi^2(U)} = \OO(1)\exp\bigl(-C^{-1}h^{-\frac{1}{2s-1}}\bigr) \|u\|_{H_{\Phi}(V)} .
	\end{equation} 
	Arguing as in the proof of Proposition \ref{Pi_V almost holomorphic Gevrey prop}, using the $\overline{\partial}$--estimates \eqref{amplitude dbar estimate} and \eqref{Psi almost holomorphic Gevrey}, with the help of \eqref{Psi basic estimate}, we obtain 
	\begin{equation}\label{appendix dbar x estimate of K 1}
		\Bigl|\partial_{\oo{x}}\Bigl(e^{\frac{2}{h}\Psi(x,\overline{y})} a(x,\overline{y};h)\Bigr) e^{-\frac{1}{h}\Phi(x)-\frac{1}{h}\Phi(y)}\Bigr| = \OO(1)\exp\bigl(-C^{-1}h^{-\frac{1}{2s-1}}\bigr) .
	\end{equation}
	Recalling the expression \eqref{Pi_V eq}, we then get by the Cauchy--Schwarz inequality,
	\begin{equation}
		\label{appendix dbar estimate f}
		\left|h\overline{\partial} f (x)\right| = \OO(1)\exp\bigl(-C^{-1}h^{-\frac{1}{2s-1}}\bigr) e^{\frac{\Phi(x)}{h}} \|u\|_{H_\Phi(V)},\quad x\in U
	\end{equation}
	Letting $\widetilde{U}\Subset U$ be an open neighborhood of $x_0$, we conclude from Proposition \ref{prop:L2 to pointwise} and \eqref{appendix L2 estimate f}, \eqref{appendix dbar estimate f} that
	\[
	\left|\widetilde{\Pi}_V u(x) - u(x)\right| = \OO(1)\exp\bigl(-C^{-1}h^{-\frac{1}{2s-1}}\bigr) e^{\frac{\Phi(x)}{h}} \|u\|_{H_\Phi(V)},\quad x\in \widetilde{U}.
	\]
	Let us rewrite the above inequality in the following form
	\begin{equation}
		\label{appendix local reproducing}
		\begin{split}
			u(x) = &\frac{1}{h^n} \int_{V} e^{\frac{2}{h}\Psi(x,\overline{y})} a(x,\overline{y};h) u(y) e^{-\frac{2}{h}\Phi(y)}\, L(dy) \\
			& + \OO(1)\exp\bigl(-C^{-1}h^{-\frac{1}{2s-1}}\bigr) e^{\frac{\Phi(x)}{h}} \|u\|_{H_\Phi(V)}
		\end{split},\quad x\in \widetilde{U},\ \; u\in H_\Phi(V),
	\end{equation}
	which is often called the local reproducing property. We shall apply \eqref{appendix local reproducing} to $u\in H_\Phi(\Omega)$. To this end, let us recall from \cite[Proposition 2.3]{deleporte2022analytic} that
	\begin{equation}
		\label{appendix u(x)}
		|u(x)|= \OO(1)h^{-n} e^{\frac{\Phi(x)}{h}} \|u\|_{H_\Phi(\Omega)},\quad x\in V.
	\end{equation}
	Let $\chi\in\CIc(V;[0,1])$ be such that $\chi=1$ near $\overline{U}$. In view of \eqref{appendix u(x)} and \eqref{Psi basic estimate}, we see that
	\begin{equation}\label{appendix 1-chi}
		\begin{gathered}
			\left| \frac{1}{h^n}\int_{V} e^{\frac{2}{h}\Psi(x,\overline{y})} (1-\chi(y)) a(x,\overline{y};h) u(y) e^{-\frac{2}{h}\Phi(y)}\, L(dy) \right|  \\
			\leq \OO(1)e^{-\frac{1}{Ch}} e^{\frac{\Phi(x)}{h}} \|u\|_{H_\Phi(\Omega)} ,\quad x\in\widetilde{U} .
		\end{gathered} 
	\end{equation}
	Combining \eqref{appendix local reproducing} and \eqref{appendix 1-chi}, we obtain, given $u\in H_\Phi(\Omega)$,
	\begin{equation}\label{appendix local reproducing Omega}
		\begin{split}
			u(y) = & \int_\Omega \widetilde{K}(y,\zbar)\chi(z)u(z)e^{-\frac{2}{h}\Phi(z)} L(dz) \\
			&+ \OO(1)\exp\bigl(-C^{-1}h^{-\frac{1}{2s-1}}\bigr) e^{\frac{\Phi(y)}{h}} \|u\|_{H_\Phi(\Omega)}
		\end{split},\quad y\in\widetilde{U},
	\end{equation}
	where we denote by 
	\begin{equation}
		\label{appendix tilde K}
		\widetilde{K}(y,\zbar) = \frac{1}{h^n} e^{\frac{2}{h}\Psi(y,\zbar)} a(y,\zbar;h),\quad (y,z)\in V\times V.
	\end{equation}
	Following \cite{BBSj08}, \cite[Appendix A]{deleporte2022analytic} and \cite[Section 5]{hitrik2022smooth}, we shall apply \eqref{appendix local reproducing Omega} to the function $y\mapsto K(y,\xbar) = \overline{K(x,\ybar)}\in H_\Phi(\Omega)$, where $K(x,\ybar)e^{-\frac{2}{h}\Phi(y)}$ is the Schwartz kernel of the Bergman projection $\Pi: L^2(\Omega, e^{-2\Phi/h}L(dx))\to H_\Phi(\Omega)$. For the properties of $K(x,\ybar)$, we refer the reader to \cite[Appendix A]{deleporte2022analytic} and the references given there. We obtain therefore 
	\begin{equation}
		\label{appendix K(y,xbar)}
		\begin{split}
			K(y,\xbar) = & \int_\Omega \widetilde{K}(y,\zbar)\chi(z)K(z,\xbar) e^{-\frac{2}{h}\Phi(z)} L(dz) \\
			&+ \OO(1)\exp\bigl(-C^{-1}h^{-\frac{1}{2s-1}}\bigr) e^{\frac{\Phi(x)+\Phi(y)}{h}} 
		\end{split},\quad x, y\in\widetilde{U}.
	\end{equation}
	Taking the complex conjugates, in view of the Hermitian property $K(x,\ybar) = \overline{K(y,\xbar)}$, and using the following expression for the Bergman projection $\Pi$:
	\[
	\Pi u(x) = \int_\Omega K(x,\ybar)u(y)e^{-\frac{2}{h}\Phi(y)} \,L(dy),\quad u\in L^2(\Omega,e^{-2\Phi/h}L(dx)),
	\]
	we can rewrite \eqref{appendix K(y,xbar)} as follows,
	\begin{equation}
		\label{appendix K(x,ybar)}
		K(x,\ybar) = \Pi\bigl(\widehat{K}(\cdot,\ybar)\chi\bigr)(x) +  \OO(1)\exp\bigl(-C^{-1}h^{-\frac{1}{2s-1}}\bigr)e^{\frac{\Phi(x)+\Phi(y)}{h}}  ,\quad x, y\in\widetilde{U} ,
	\end{equation}
	where, in view of \eqref{appendix tilde K} and \eqref{Psi as polarization Gevrey},
	\begin{equation}
		\label{appendix Hat K}
		\begin{gathered}
			\widehat{K}(z,\ybar) = \overline{\widetilde{K}(y,\zbar)} = \frac{1}{h^n} e^{\frac{2}{h}\widehat{\Psi}(z,\ybar)} b(z,\ybar;h), \\
			\widehat{\Psi}(z,\ybar) := \overline{\Psi(y,\zbar)},\quad  b(z,\ybar;h) := \overline{a(y,\zbar;h)} .  
		\end{gathered}
	\end{equation}
	We aim to show that $\Pi\bigl(\widehat{K}(\cdot,\ybar)\chi\bigr)(x)$ is close to $\widehat{K}(x,\ybar)\chi(x) = \widehat{K}(x,\ybar)$ for $x, y\in \widetilde{U}$ up to an error of the same size as in \eqref{appendix K(x,ybar)}. For this purpose, we consider the function
	\[
	\Omega\owns x\mapsto u_y(x) = \widehat{K}(x,\ybar)\chi(x) - \Pi\bigl(\widehat{K}(\cdot,\ybar)\chi\bigr)(x),
	\]
	which solves the following $\overline{\partial}$--equation in $\Omega$,
	\[
	\overline{\partial} u_y = \overline{\partial} \bigl(\widehat{K}(\cdot,\ybar)\chi\bigr) = \chi \overline{\partial}\widehat{K}(\cdot,\ybar) + \widehat{K}(\cdot,\ybar)\overline{\partial}\chi .
	\]
	Applying H\"ormander's $L^2$--estimate to this $\overline{\partial}$--equation on the pseudoconvex domain $\Omega$ with the plurisubharmonic function $\Phi$, see \cite[Proposition 4.2.5]{hormander1994convexity}, we obtain
	\begin{equation}
		\label{appendix u_y L2 integral}
		\int_\Omega |u_y(x)|^2 e^{-\frac{2}{h}\Phi(x)} L(dx) = \OO(h) \int_\Omega \frac{1}{c(x)} |\overline{\partial}_x(\widehat{K}(x,\ybar)\chi(x))|^2  e^{-\frac{2}{h}\Phi(x)} L(dx),
	\end{equation}
	uniform in $y\in\widetilde{U}$, with $c(x)$ given in \eqref{Phi strict plurisubharmonic}. We note that
	\[
	|\overline{\partial}_x(\widehat{K}(x,\ybar)\chi(x))|^2 \asymp |\overline{\partial}\chi(x)|^2 |\widehat{K}(x,\ybar)|^2 + |\chi(x)|^2 |\overline{\partial}_x \widehat{K}(x,\ybar)|^2
	\]
	In view of \eqref{appendix Hat K}, \eqref{Psi basic estimate} and the fact that $\supp \overline{\partial}\chi \subset V\setminus U$, we get
	\begin{equation}
		\label{appendix D chi part}
		\int_\Omega \frac{1}{c(x)} |\overline{\partial}\chi(x)|^2 |\widehat{K}(x,\ybar)|^2 e^{-\frac{2}{h}\Phi(x)} L(dx) = \OO(1)e^{-C^{-1}h^{-1}} e^{\frac{2}{h}\Phi(y)},\quad y\in\widetilde{U}.
	\end{equation}
	Similar as \eqref{appendix dbar x estimate of K 1}, we also have, using \eqref{amplitude dbar estimate}, \eqref{Psi almost holomorphic Gevrey} and \eqref{Psi basic estimate},
	\begin{equation}\label{appendix dbar x estimate of K 2}
		\Bigl|\partial_{\oo{x}}\Bigl(e^{\frac{2}{h}\overline{\Psi(y,\xbar)}} \overline{a(y,\overline{x};h)}\Bigr) e^{-\frac{1}{h}\Phi(x)-\frac{1}{h}\Phi(y)}\Bigr| = \OO(1)\exp\bigl(-C^{-1}h^{-\frac{1}{2s-1}}\bigr) .
	\end{equation}
	We omit again the details and refer to the proof of Proposition \ref{Pi_V almost holomorphic Gevrey prop}. Then by \eqref{appendix Hat K},
	\begin{equation}
		\label{appendix chi part}
		\int_\Omega \frac{1}{c(x)} |\chi(x)|^2 |\overline{\partial}_x \widehat{K}(x,\ybar)|^2 e^{-\frac{2}{h}\Phi(x)} L(dx) = \OO(1)e^{-C^{-1}h^{-\frac{1}{2s-1}}} e^{\frac{2}{h}\Phi(y)},\quad y\in\widetilde{U}.
	\end{equation} 
	Combining \eqref{appendix u_y L2 integral}, \eqref{appendix D chi part} and \eqref{appendix chi part}, we obtain
	\begin{equation}
		\label{appendix u_y L2 norm}
		\|u_y\|_{L^2(\Omega, e^{-2\Phi/h}L(dx))} = \OO(1)\exp\bigl(-C^{-1}h^{-\frac{1}{2s-1}}\bigr) e^{\frac{\Phi(y)}{h}},\quad y\in\widetilde{U} .
	\end{equation}
	To pass from \eqref{appendix u_y L2 norm} to a pointwise estimate, we shall apply again Proposition \ref{prop:L2 to pointwise} to $u_y$. Let us first note that
	\[
	h\overline{\partial}_z u_y(z) = h\overline{\partial}_z \widehat{K}(z,\ybar),\quad z\in U,\ y\in\widetilde{U}.
	\]
	It then follows from \eqref{appendix Hat K} and \eqref{appendix dbar x estimate of K 2} that
	\begin{equation}
		\label{appendix dbar u_y pointwise}
		\sup_{z\in U} |	h\overline{\partial}_z u_y(z)|e^{-\frac{1}{h}\Phi(z)} = \OO(1)\exp\bigl(-C^{-1}h^{-\frac{1}{2s-1}}\bigr) e^{\frac{\Phi(y)}{h}},\quad y\in\widetilde{U} .
	\end{equation}
	We conclude from \eqref{L2 and dbar to pointwise}, \eqref{appendix u_y L2 norm} and \eqref{appendix dbar u_y pointwise} that, uniformly for $x, y\in\widetilde{U}$,
	\begin{equation}
		\label{appendix u_y pointwise}
		|u_y(x)| = \left| \widehat{K}(x,\ybar)\chi(x) - \Pi\bigl(\widehat{K}(\cdot,\ybar)\chi\bigr)(x) \right| = \OO(1)\exp\bigl(-C^{-1}h^{-\frac{1}{2s-1}}\bigr) e^{\frac{\Phi(x)+\Phi(y)}{h}} .
	\end{equation}
	Combining \eqref{appendix K(x,ybar)} and \eqref{appendix u_y pointwise}, also recalling \eqref{appendix Hat K}, we obtain
	\begin{equation*}
		K(x,\ybar) = \overline{\widetilde{K}(y,\xbar)} + \OO(1)\exp\bigl(-C^{-1}h^{-\frac{1}{2s-1}}\bigr) e^{\frac{\Phi(x)+\Phi(y)}{h}},\quad x, y\in\widetilde{U},
	\end{equation*}
	taking the complex conjugates and using the Hermitian property of $K$, we get
	\begin{equation}\label{appendix tilde K to K}
		K(y,\xbar) = \widetilde{K}(y,\xbar) + \OO(1)\exp\bigl(-C^{-1}h^{-\frac{1}{2s-1}}\bigr) e^{\frac{\Phi(x)+\Phi(y)}{h}},\quad x, y\in\widetilde{U} .
	\end{equation}
	Switching the variables $x$ and $y$ in \eqref{appendix tilde K to K}, we complete the proof of Corollary \ref{corr: asymp to exact Bergman}.

\end{document}